\documentclass[a4paper,12pt]{article}
\usepackage{hyperref}
\usepackage{amsmath}
\usepackage{amssymb}
\usepackage{amsthm}
\usepackage{bbm}
\usepackage[margin=1in]{geometry}
\usepackage{mathrsfs}
\usepackage[color=yellow]{todonotes}
\usepackage{tikz}
\usepackage{enumitem}
\usetikzlibrary{arrows}
\usetikzlibrary{shapes}
\usepackage[utf8]{inputenc}
\usepackage[backend=biber,style=alphabetic, maxbibnames=99]{biblatex}
\newtheorem{theorem}{Theorem}[section]
\newtheorem{lemma}[theorem]{Lemma}

\newtheorem{proposition}[theorem]{Proposition}
\newtheorem{definition}[theorem]{Definition}
\newtheorem{corollary}[theorem]{Corollary}
\newtheorem{remark}[theorem]{Remark}
\newtheorem{assumptions}[theorem]{Assumptions}
\newtheorem{assumption}[theorem]{Assumption}

\newcommand{\EE}{\mathbb{E}}

\newcommand{\IN}{\mathbb{N}}

\newcommand{\PP}{\mathbb{P}}
\newcommand{\IP}{\mathbb{P}}

\newcommand{\IQ}{\mathbb{Q}}

\newcommand{\RR}{\mathbb{R}}
\newcommand{\IR}{\mathbb{R}}

\newcommand{\VV}{\mathbb{V}}

\newcommand{\D}{{\cal D}}

\newcommand{\T}{{\cal T}}

\newcommand{\U}{{\cal U}}

\newcommand{\dd}{{\mathbbm{d}}}

\newcommand{\bigO}[1]{{\mathcal {O}\left(#1\right)}}

\newcommand{\W}{{\mathbf W}} %%% I AM CHANGING THIS HERE TO SAVE CHANGING ALL THE \W's TO SOMETHING ELSE, BUT WANT CONSISTENT NOTATION FOR ONE-DIMENSIONAL AND MULTI-DIMENSIONAL HISTORICAL PROCESSES

\newcommand{\1}{{\bf {1}}}
\def\v{\boldsymbol} %for vectors in bold text, use as \v{x

\def\v{\boldsymbol} %for vectors in bold text, use as \v{x}

\def\dim{\mathbbm{d}}

\def\Ca{($\mathscr{C}1$)}
\def\Cb{($\mathscr{C}2$)}
\def\Cc{($\mathscr{C}3$)}

\def\epsilon{\varepsilon}

\begin{document}

\title{\itshape On the effects of a wide opening in the domain of the
(stochastic) Allen-Cahn equation and the motion of hybrid zones.}
\author{Alison M.~Etheridge, Mitchell D.~Gooding, Ian Letter}
\maketitle
\begin{abstract}
We are concerned with a special form of the (stochastic) Allen-Cahn
equation, which can be seen as a model of hybrid zones in population 
genetics. Individuals in the population can be of one of three types; 
$aa$ are fitter than $AA$, and both are fitter than the 
$aA$ heterozygotes. The {\em hybrid zone} is the  
region separating a subpopulation consisting entirely of $aa$
individuals from one consisting of $AA$ individuals. 
%that are maintained by weak selection against 
%heterozygotes and its interplay with other factors, 
We investigate the interplay between the motion of the hybrid zone and the 
shape of the habitat, both with and without genetic drift (corresponding
to stochastic and deterministic models respectively). In the deterministic model, we investigate the effect of a wide opening and provide some explicit 
sufficient conditions under which
the spread of the advantageous type is 
halted, and complementary conditions under which 
it sweeps through the whole population. As a standing example, we 
are interested in the outcome of the advantageous population
passing through an isthmus. 
We also identify
rather precise
conditions under which genetic drift breaks down the structure of the
hybrid zone, complementing previous work that identified conditions
on the strength of genetic drift under which the structure of the
hybrid zone is preserved. 

Our results demonstrate that, even in cylindrical domains, it can be misleading
to caricature allele frequencies by one-dimensional travelling waves, and 
that the strength of genetic drift plays an important role in determining 
the fate of a favoured allele.
%This refers to a phenomenon whereby an equilibrium state can be 
%achieved in the population, even when a dominant species starts in 
%an unbounded part of the space. 
%We then turn to a stochastic analogue of the model which incorporates
%genetic drift. If drift is sufficiently weak, then the progress
%of an advantageous type can still be blocked by a wide opening in the domain,
%but we are able to rather precisely identify conditions under which the 
%drift overwhelms selection. %This reflects the fact that natural selection is
%more effective when population density is higher.
% A novelty of our result is that we can give explicit and sufficient conditions on the spatial domains that exhibit such a property, describing then a phase transition for blocking. We also sketch how we can define a version of the Spatial $\Lambda$-Fleming Viot process with selection but with \textit{reflecting} boundary conditions, for which an analogous result can be proved in a weak noise regime. Furthermore, we show that, in a strong noise regime, hybrid zones are destroyed by genetic drift.
\end{abstract}
\tableofcontents \newpage 

\section{Introduction}
\label{introduction}

We are interested in a particular bistable reaction-diffusion equation, that
can be seen as providing a simple model for a so-called hybrid zone in population genetics,
and a stochastic analogue that captures the
randomness stemming from bounded population density.
Specifically, our main focus will be
\[ 
(AC_\epsilon) = \begin{cases} 
\partial_t u^\epsilon = \Delta u^\epsilon +  \frac{1}{\epsilon^2} 
u^\epsilon(1-u^\epsilon)(2u^\epsilon -(1-\nu \epsilon)) 
& x \in \Omega,\; t>0, \\ 
\partial_n u^\epsilon = 0 &  x \in \Omega,\; t>0, \\  
u^\epsilon(x,0)= \1_{x_1 \geq 0} & x \in \Omega, 
\end{cases} 
\]
with $\nu\in (0,\infty)$, $\epsilon$ a small parameter, 
$\Omega$ an unbounded domain in $\IR^\dim$, and $\partial_nu^\epsilon$ the
normal derivative at the boundary. In this model, as we 
explain below, the hybrid zone is the narrow region in which the 
solution takes values in $(\delta, 1-\delta)$ (for some small $\delta$). 

In previous work, \cite{etheridge/freeman/penington:2017} considered the 
case of a symmetric potential ($\nu=0$), 
whereas the doctoral thesis of the second author, \cite{gooding:2018}, 
considered 
the asymmetric case; both were concerned with populations evolving in 
the whole Euclidean space. Here we shall 
ask about propagation of solutions through other unbounded domains. We
are particularly interested in the question of when the spread of a 
population may be halted, for example when passing through an isthmus, 
and whether this will change in the presence of noise. 
Theorems~\ref{teo:simplifyversion} and~\ref{no_blocking} 
identify conditions under which such `blocking' can and cannot occur for a 
simple domain illustrated in Figure~\ref{fig:omega} below.
Section~\ref{sec:mordom}
indicates how to extend these results to a multitude of other domains.
Theorem~\ref{noisy circles}
discusses a stochastic analogue 
of~$(AC_\epsilon)$ based on the spatial $\Lambda$-Fleming-Viot process. 
For $\Omega=\IR^d$, 
\cite{etheridge/freeman/penington:2017} 
and~\cite{gooding:2018}, both identify noisy regimes in which the
behaviour is the same as for the deterministic equation; here we 
complement those results by also identifying
conditions under which the noise is strong enough
to break down the structure induced by the potential in~$(AC_\epsilon)$.

\subsection{The deterministic case} \label{intro:detcase}

Our starting point is
the following special case of the 
Allen-Cahn equation on a domain $\Omega\subseteq \IR^{\dim}$,
\[ 
(AC) = \begin{cases} 
\partial_t u = \Delta u +  \v{s}u(1-u)(2u-(1-\gamma)) & 
x =(x_1,x_2,.., x_{\mathbbm{d}})\in \Omega, \; t>0, \\ 
\partial_n u= 0 &  x \in \partial\Omega, \; t>0, \\  
u(x,0)= \1_{x_1 \geq 0} 
& x \in \Omega, 
\end{cases} 
\]
where $\v{s}>0$ and $\gamma\in (0,1)$ are constants. 
%\cite{etheridge/freeman/penington:2017} consider the case 
%$\gamma=0$ and $\v{s}$ large, corresponding to applying a 
%diffusive scaling in order to view the population over large 
%spatial and temporal scales.
\cite{berestycki/bouhours/chapuisat:2016} consider a general bistable 
nonlinearity $f(u)$ in place of the particular cubic term that arises
in our application, and they take the domain to be 
`cylinder-like':
\begin{equation}
\label{domain}
\Omega=\left\{(x_1,x'), x_1\in\IR, x'\in 
\phi(x_1)\subseteq\IR^{\dim -1}\right\}.
\end{equation}
For our equation, their results show that depending on the geometry of 
the domain, we can have different long-term behaviours of the solution of equation $(AC)$.
\begin{theorem}[\cite{berestycki/bouhours/chapuisat:2016}, Theorems 1.4, 1.5, 1.6, 1.7, paraphrased]
Depending on the geometry of the domain $\Omega$ we have one of three possible asymptotic behaviours of the solution of equation $(AC)$.
\begin{enumerate}
    \item there can be {\em complete invasion}, 
that is $u(x,t)\to 1$ as $t\to\infty$ for every $x \in \Omega$.
\item there can be {\em blocking} of the solution, meaning that 
$u(x,t)\to u_\infty(x)$ as $t\to\infty$, with $u_\infty(x)\to 0$ as $x_1\to -\infty$.
\item there can be {\em axial partial propagation}, 
meaning that $u(x,t)\to u_\infty(x)$ as $t\to\infty$, with 
$\inf_{x\in\IR\times B_R} u_\infty(x) > c >0$ for some $R>0$, 
where $B_R$ is the ball of radius $R$ centred at $0$ in $\IR^{\dim -1}$.
\end{enumerate}
Which behaviour is observed depends on the geometry of the domain $\Omega$.
There will be complete invasion if $\Omega$ is decreasing as $x_1\to-\infty$;
axial partial propagation if it 
contains a straight cylinder of sufficiently large cross-section; and 
there can be blocking 
if there is an abrupt change in the geometry.
%in the $x_1$-direction, the cross-section of
%which is a ball of sufficiently large radius.
\end{theorem}

We remark that in \cite{berestycki/bouhours/chapuisat:2016} 
the convention is to consider invasions from 
$-\infty$ to $+\infty$; our choice, which is the opposite,
makes it easier to borrow results from~\cite{gooding:2018}.
Our results complement those 
of~\cite{berestycki/bouhours/chapuisat:2016}, while being more 
quantitative in the conditions imposed on the geometry of the domain
and allowing for the inclusion of noise, corresponding to `genetic drift'. 
However, our results do not contain or imply the ones 
of~\cite{berestycki/bouhours/chapuisat:2016}. 
We will introduce a new parameter $\varepsilon$ to the equation, 
which prevents a direct 
comparison between the two sets of results.

Since hybrid zones are typically narrow compared to the range of
the population, we take $\v{s}$ to be large, which necessitates taking
$\gamma$ to be small if the motion of the hybrid zone is not to be 
unreasonably fast. As we shall see, the motion of the hybrid zone can
be blocked if the domain $\Omega$ has an abrupt wide 
opening. Although we shall consider more general domains in
Section~\ref{sec:mordom}, to introduce the main ideas we shall 
begin by focusing on a domain of
the form shown in Figure~\ref{fig:omega}; so that, in the notation
above, $\phi(x)\equiv 1_{\|x'\|<R_0}$ for $x_1<0$, 
$\phi(x)\equiv 1_{\|x'\|<r_0}$ for $x_1>0$, where $r_0<R_0$.
\begin{figure*}[t!] 
    \centering
        \includegraphics[scale=0.8]{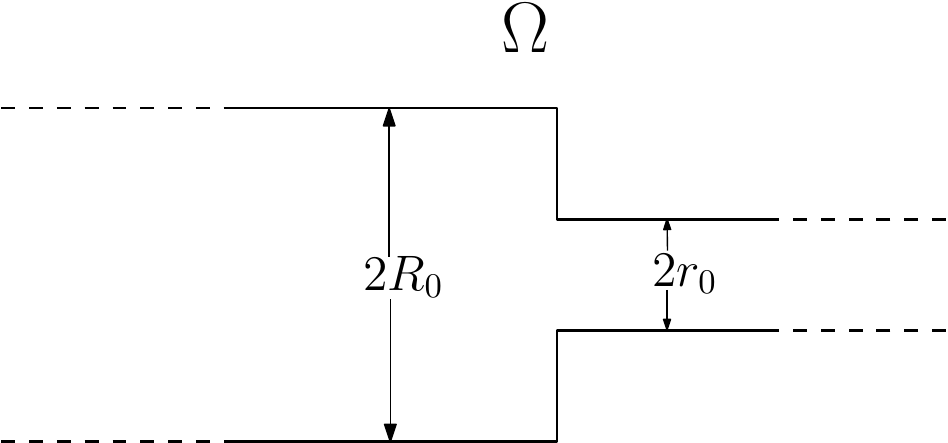}
    \caption{The domain $\Omega$ of Theorem~\ref{teo:simplifyversion}.}
    \label{fig:omega}
\end{figure*} 

To understand how Equation~(AC) relates to hybrid zones, 
let us sketch its derivation from 
a biological model. Consider a diploid population (individuals
carry chromosomes in pairs) in which  
a trait subject to natural selection is determined by a single 
bi-allelic genetic locus. We denote the 
alleles by $\{a, A\}$, so that there are three possible types in 
our population: the {\em homozygotes}
$aa$ and $AA$, and the {\em heterozygotes} $aA$.
The relative fitnesses of individuals of the different types are given by

\begin{center}{\begin{tabular}{c|c|c}
$aa$ & $aA$ & $AA$\\ \hline
$1+\widetilde{s}$ & $1$ & $1+\theta \widetilde{s}$,
\end{tabular}}
\end{center}

\noindent
where $\widetilde{s}$ is a small positive constant and $\theta\in (0,1]$. In other words,
homozygotes are fitter than heterozygotes, and, if $\theta<1$,
individuals carrying $aa$ are fitter than those carrying $AA$.
We suppose that the population is at Hardy-Weinberg equilibrium, so that 
the proportion of individuals of types $aa$, $aA$, and $AA$ are $w^2$, $2w(1-w)$, and $(1-w)^2$
respectively, where $w$ is the proportion of $a$-alleles in the population.
During reproduction, each individual produces a large (effectively infinite) 
number of germ cells (carrying the same genetic
material as the parent), which then split into gametes (carrying just one copy of each chromosome).
Each offspring is formed by fusing two gametes picked at random from this pool.
In an infinite population, the proportion of $a$-alleles in the offspring population will then be
that in the pool of gametes. Assuming that individuals produce a 
number of gametes proportional to their relative fitness, this is calculated to be
$$\frac{(1+\widetilde{s})w^2+w(1-w)}{1+\widetilde{s} w^2+\theta \widetilde{s} (1-w)^2},$$
and for $\widetilde{s}$ small, we see that the change in frequency of $a$-alleles over
a single generation, obtained by subtracting $w$ from this quantity, is
$$\widetilde{s} w(1-w)\big(w(1+\theta)-\theta\big) +\mathcal{O}(\widetilde{s}^2).$$
Equation~$(AC)$ is recovered by adding dispersal of offspring, 
setting $\widetilde{s}(1+\theta)/2=\v{s}/N$, $\gamma=(1-\theta)/(1+\theta)$, measuring time in units of 
$N$ generations, and letting $N$ tend to infinity.

\cite{etheridge/freeman/penington:2017} consider the case in which $\theta=1$, and so $\gamma=0$,
corresponding to both homozygotes being equally fit.  They work on the whole of $\IR^{\dim}$. 
To understand the behaviour of the population over large spatial and temporal scales, they
apply a diffusive rescaling. The equation becomes
\[ %\begin{equation} 
(SAC_\epsilon) = 
\begin{cases} \partial_t v^\epsilon = \Delta v^\epsilon +  
\frac{1}{\epsilon^2} v^\epsilon(1-v^\epsilon)(2v^\epsilon-1) 
	& x \in \IR^{\dim}, \quad t>0, \\ 
v^\epsilon(x,0)= p(x). 
\end{cases} 
\] %\end{equation}
To state their result we need some notation.
Let $\{\v{\Gamma}_t:S^{\dim -1}\to\IR^{\dim}\}_{0\leq t\leq T}$
be a family of smooth embeddings of the
surface of the unit sphere in $\IR^{\dim}$ to $\IR^{\dim}$, 
evolving according to
mean curvature flow. That is, writing $\mathbf{n}_t(s)$ for the unit
inward normal vector to $\v{\Gamma}_t$ at $s$, and $\kappa_t(s)$ for the 
mean curvature of $\v{\Gamma}_t$ at $s$,
$$\partial_t \v{\Gamma}_t(s)=\kappa_t(s)\mathbf{n}_t(s).$$
In the biologically relevant case, $\dim=2$, mean curvature flow
is just curvature flow.
We think of this process as defined up to the fixed time 
$\mathscr{T}$ at which it first develops a singularity.
Let $d(x,t)$ be the signed distance from $x$ to $\v{\Gamma}_t$, 
chosen to be negative 
inside $\v{\Gamma}_t$ and positive outside. Note that, as sets,
$$\v{\Gamma}_t=\{x\in\IR^\dim : d(x,t)=0\}.$$
We require some regularity assumptions on the initial condition $p$ 
of~$(SAC_\epsilon)$. Set $\Gamma=\{x\in\IR^{\dim}: p(x)=\frac{1}{2}\}$;
we shall take $\v{\Gamma}_0=\Gamma$. We assume that
\begin{enumerate}[leftmargin=1.1cm]
\item[{\Ca}] $\Gamma$ is $C^{\alpha}$ for some $\alpha>3$.
\item[{\Cb}] For $x$ inside $\Gamma$, $p(x)>\tfrac{1}{2}$. 
For $x$ outside $\Gamma$, $p(x)<\tfrac{1}{2}$.
\item[{\Cc}] There exist $\eta,\mu>0$ such that, for all $x\in\IR^{\dim}$, 
$|p(x)-\frac{1}{2}|\geq \mu\,\big(\text{dist}(x,\Gamma)\wedge \eta\big)$.
\end{enumerate}
The following result, proved using probabilistic techniques
in~\cite{etheridge/freeman/penington:2017}, is
a special case of Theorem~3 of~\cite{chen:1992}.
\begin{theorem} \label{theorem ac to cf}
Let $v^\epsilon$ 
solve~$(SAC_\epsilon)$ with initial condition $p$
satisfying the conditions {\Ca}-{\Cc}, and define $\mathscr{T}$, $d(x,t)$ as above. 
Fix $T^*\in (0,\mathscr{T})$ and let 
$k\in\IN$. There exist 
$\epsilon_\dim(k)>0$, and $a_\dim(k),c_\dim(k)\in(0,\infty)$ such 
that for all $\epsilon\in(0,\epsilon_\dim)$ and $t$ satisfying 
$a_\dim\epsilon ^2 |\log \epsilon |\leq t\leq T^*,$
\begin{enumerate}
\item for $x$ such that $d(x,t)\leq c_\dim\epsilon |\log \epsilon|$, we have 
$v^\epsilon (t,x)\geq 1-\epsilon^k$;
\item for $x$ such that $d(x,t)\geq -c_\dim\epsilon |\log \epsilon|$, we have 
$v^\epsilon (t,x)\leq \epsilon^k$.
\end{enumerate}
\end{theorem}
\begin{remark}
In fact~\cite{chen:1992}
and ~\cite{etheridge/freeman/penington:2017}
choose $p<\tfrac{1}{2}$ within the domain enclosed by $\Gamma$ and 
$p>\tfrac{1}{2}$ outside. Since Theorem~\ref{theorem ac to cf} 
concerns the symmetric equation~$(SAC_\epsilon)$, our statement here 
is equivalent. 
%we
%could equally have chosen $p(x)$ to be bigger than $1/2$ inside the 
%domain, and greater outside, and in the limit, the motion of the interface 
%between the regions where $v^\epsilon$ is close to zero and where 
%it is close to one would still converge to mean curvature flow.
\end{remark}
In \cite{gooding:2018} the approach of~\cite{etheridge/freeman/penington:2017}
is modified to apply to the case when the homozygotes are not equally fit. 
The proof of Theorem~\ref{theorem ac to cf}
compares the solution of~$(SAC_\epsilon)$ to the solution to the one-dimensional 
equation started from a Heaviside initial condition, which has a stable limiting form.
To understand the results of~\cite{gooding:2018}, it is also instructive to consider the 
one-dimensional version of~$(AC)$.
Note that the one-dimensional equation
$$\partial_tu=\frac{\sigma^2}{2}\partial_{xx}u
+\v{s}u(1-u)\big(2u-(1-\gamma)\big)$$
has a travelling wave solution of the form:
\begin{equation}
\label{travelling wave}
u(x,t)=\frac{1}{2}\left(1-\tanh\left(\sqrt{\frac{\v{s}}{2\sigma^2}}(x-ct)\right)\right),
\end{equation}
where the wavespeed is $c=\gamma\sigma\sqrt{\v{s}/2}$.
%If we apply a diffusive rescaling, as we did in the symmetric case, then the diffusion coefficient is unchanged, but 
%$s$ will be multiplied by $1/\epsilon^2$, and so in order to see a finite wavespeed, we must also scale $\gamma$ by $\epsilon$.
This tells us that if we scale $\sigma$ and/or $\v{s}$, then we may also have to scale $\gamma$ in order to 
obtain a finite wavespeed.

\cite{gooding:2018} considers the equation
\begin{equation} 
\label{mitch equation}
%(AC_\epsilon) = 
\begin{cases} 
\partial_t w^\epsilon = \epsilon^{1-\ell}\Delta w^\epsilon +  \frac{1}{\epsilon^{1+\ell}} 
w^\epsilon(1-w^\epsilon)(2w^\epsilon-(1-\gamma_\epsilon)) 
& x \in \IR^{\dim},\; t>0, \\
w^\epsilon(x,0)= p(x) & x \in \IR^{\dim}, 
\end{cases} 
\end{equation}
where $\gamma_\epsilon = %\widetilde{\nu} 
\nu\epsilon^{\widetilde{\ell}}$ for some non-negative $\nu$ and 
% $\widetilde{\nu}$, 
$\widetilde{\ell}$,
with the additional condition that %$\widetilde{\nu}<1$ 
$\nu<1$ when $\widetilde{\ell}=0$, and $\ell=\min(\widetilde{\ell},1)$.

Notice that with these parameters, the one-dimensional wave has speed of $\mathcal{O}(1)$ if $\widetilde{\ell}\leq 1$ and 
tending to zero as $\epsilon^{\widetilde{\ell}-1}$ if $\widetilde{\ell}>1$.
To state the analogue of Theorem~\ref{theorem ac to cf} in this case, 
we have to modify our assumptions on $\Gamma$:
\begin{enumerate}[leftmargin=1.1cm]
\item[{\Ca}'] $\Gamma$ is $C^{\alpha}$ for some $\alpha>3$.
\item[{\Cb}'] For $x$ inside $\Gamma$, $p(x)<\tfrac{1+\gamma_\epsilon}{2}$. For $x$ outside $\Gamma$, 
$p(x)>\tfrac{1+\gamma_\epsilon}{2}$.
\item[{\Cc}'] There exist $\eta,\mu>0$ such that, for all $x\in\IR^{\dim}$, 
$|p(x)-\frac{1+\gamma_\epsilon}{2}|\geq \mu\,
\big(\text{dist}(x,\Gamma)\wedge \eta\big)$.
\end{enumerate}
We define 
\begin{equation}
\label{defn of nu epsilon}
\nu_\epsilon= \left\{\begin{array}{ll}
\nu & %$\widetilde{\nu}$ 
\mbox{if }\widetilde{\ell}\leq 1, \\
\gamma_\epsilon/\epsilon &\mbox{if }
\widetilde{\ell}\in (1,2], \\
0 &\mbox{if }\widetilde{\ell}>2.
\end{array}\right.
\end{equation}
\begin{theorem}[\cite{gooding:2018}, Theorem~2.4]
\label{mitch theorem}
Let $w^\epsilon$ solve Equation~(\ref{mitch equation}) with initial 
condition $p$ satisfying {\Ca}'-{\Cc}', and let 
\begin{equation}
\label{curvature plus constant flow}
\partial_t \widetilde{\v{\Gamma}}_t(s)
=(\nu_\epsilon+\kappa_t(s))\mathbf{n}_t(s),
\end{equation}
until the time $\mathscr{T}$ at which $\widetilde{\v{\Gamma}}$ 
develops a singularity.
Write $d$ for the signed distance to $\v{\Gamma}$ (chosen to be 
positive outside $\v{\Gamma}$).
Fix $T^*\in (0,\mathscr{T})$. 
Let $k\in\IN$. There exists 
$\epsilon_\dim(k)>0$, and $a_\dim(k),c_\dim(k)\in(0,\infty)$ such 
that for all $\epsilon\in(0,\epsilon_\dim)$ and $t$ satisfying 
$a_\dim\epsilon ^{1+\alpha} |\log \epsilon |\leq t\leq T^*,$
\begin{enumerate}
\item for $x$ such that $d(x,t)\geq c_\dim\epsilon |\log \epsilon|$, we have 
$v^\epsilon (t,x)\geq 1-\epsilon^k$;
\item for $x$ such that $d(x,t)\leq -c_\dim\epsilon |\log \epsilon|$, we have 
$v^\epsilon (t,x)\leq \epsilon^k$.
\end{enumerate}
\end{theorem}
\begin{remark}
For $\widetilde{\ell}\leq 2$, 
$\nu_\epsilon$ in~(\ref{defn of nu epsilon}) 
and~(\ref{curvature plus constant flow}) 
corresponds to the one-dimensional wavespeed derived above. 
For $\widetilde{\ell}>2$, the wavespeed converges to zero 
sufficiently quickly as $\epsilon\to 0$ that we can directly focus on 
mean curvature flow, and not include the small correction 
in~(\ref{curvature plus constant flow})
corresponding to the one-dimensional wavespeed for the result to hold.
\end{remark}
 
When $\widetilde{l}=1$, Theorem~\ref{mitch theorem}
is a special case of Theorem~1.3 of~\cite{alfaro/hilhorst/matano:2008}
who consider the generation and propagation of a sharp interface for 
more general versions of the Allen-Cahn equation with a slightly
unbalanced 
bistable nonlinearity.

Note that $\mathbf{n}_t(s)$ is the inward facing normal, 
so the constant flow along the normal 
determined by $\nu_\epsilon$
causes contraction of the boundary, and so expansion of the region 
in which the solution is close to one. 

%\[ 
%(AC_\epsilon) = 
%\begin{cases} 
%\partial_t w^\epsilon = \Delta w^\epsilon +  \frac{1}{\epsilon^2} 
%w^\epsilon(1-w^\epsilon)(2w^\epsilon-(1-\nu \epsilon)) 
%& x \in \IR^{\dim},\; t>0, \\
%w^\epsilon(x,0)= p(x) & x \in \IR^{\dim}. 
%\end{cases} 
%\]

We are now in a position to understand why the expansion 
of a population might be blocked by an isthmus.
As advertised, we shall focus on the case $\widetilde{\ell}=1=\ell$:
%blocking can occur. 
\[ 
(AC_\epsilon) = \begin{cases} 
\partial_t u^\epsilon = \Delta u^\epsilon +  \frac{1}{\epsilon^2} 
u^\epsilon(1-u^\epsilon)(2u^\epsilon -(1-\nu \epsilon)) 
& x \in \Omega,\; t>0, \\ 
\partial_n u^\epsilon = 0 &  x \in \Omega,\; t>0, \\  
u^\epsilon(x,0)= \1_{x_1 \geq 0} & x \in \Omega, 
\end{cases} 
\]
with $\Omega$ as in Figure~\ref{fig:omega}. 
This domain has the advantage of notational simplicity, while allowing us
to introduce all the key ideas required to understand conditions 
for blocking/invasion in much more general domains in the next section.
%We work in such a domain due to its notational simplicity, and because it allows us to explain all the factors acting in the blocking/invasion dynamic, making generalisations to other domains easy. We will explain what happens in more general domains in the next section.

Our first result shows how the behaviour 
of $(AC_\varepsilon)$ on the domain $\Omega$ can be very different from 
that on the whole of Euclidean space. 
As $\epsilon\to 0$, the solution will look increasingly like the 
indicator function of a region whose boundary evolves according 
to~(\ref{curvature plus constant flow}), with the additional 
condition that it is perpendicular to the boundary $\partial\Omega$ of 
$\Omega$ where the two intersect. 
If $r_0=R_0$ then the solution will simply propagate from right 
to left. Indeed it will converge to a travelling wave, whose
exact form is found by substituting into equation~(\ref{travelling wave}).
%solving the one dimensional equation, 
%see~(\ref{one dimensional wave}).
If $R_0\gg r_0$, then the solution will try to spread out 
from the opening at $x_1=0$. 
Approximating the solution as above, the interface between the 
region where $u$ is close to $0$ and where it is close to $1$ 
is pushed to the left by the constant normal flow, while being 
pushed back by the mean curvature (a kind of `surface tension'). 
In the limit, these two forces will balance if $\nu=\kappa$,
that is precisely when the interface to the left of
the origin is a hemispherical shell
of radius $r=\frac{\dim-1}{\nu}$. 
This is, of course, the same as the radius at which the forces will balance for a solution on the whole of $\IR^d$ started from the indicator function of a 
sphere. However, working on all of $\IR^d$ does not give a satisfactory example
of blocking, as any perturbation above or below this radius will result in 
complete invasion or extinction of the dominant phenotype respectively.
The blocking that we see in domains such as $\Omega$ is robust to this type
of perturbation. 

\begin{theorem} \label{teo:simplifyversion}
Let $u^\epsilon$ denote the solution to Equation~$(AC_\epsilon)$
with $\Omega$ as in Figure~\ref{fig:omega}.
Suppose $r_0 < \frac{ \mathbbm{d}-1}{\nu} \wedge R_0$. 
Define 
\[ N_{\mathbbm{r}} = \{ x \in \Omega \, :\,  \| x \| = \mathbbm{r},\; x_1 <0 \}, \]
where $R_0> \mathbbm{r} > r_0$, and let $\widehat{d}(x)$ be the signed 
(Euclidean) distance 
of any point $x\in \Omega$ to $N_{\mathbbm{r}}$ (chosen to be negative
as $x_1\to-\infty$). 
Let $k \in \mathbb{N}$.  Then there is $\widehat{\epsilon}(k)>0$ and 
$M(k)>0$ such that for all $\epsilon \in (0,\widehat{\epsilon})$, and all $t \geq 0$,
\[\mbox{if } x=(x_1,...,x_{\mathbbm{d}})\in \Omega \mbox{ is such that } 
d(x)\leq -M(k) \epsilon |\log(\epsilon)|\mbox{ then } 
u^\epsilon (x,t) \leq \epsilon^k. \]
%where $u^\epsilon$ is the solution to $(AC_\epsilon)$.
\end{theorem}
In other words, if the aperture $r_0$ is too small, and $R_0$ is
large enough, then,
for sufficiently small $\epsilon$, blocking occurs.
The converse is also true in the following sense.

\begin{theorem}
\label{no_blocking}
Let $u^\varepsilon$ denote the solution to Equation~$(AC_\varepsilon)$
with $\Omega$ as in Figure~\ref{fig:omega}.
Suppose $r_0 > \frac{ \mathbbm{d}-1}{\nu}$, then for all $x \in \Omega$
% \cap\{x : x_1 < -r_0\}$ 
and $\delta > 0$ 
there is \mbox{$\widehat{t} :=\widehat{t}(x_1, R_0,r_0) > 0$} and $\widehat{\varepsilon}$ such that, 
for all $\varepsilon \in (0,\widehat{\varepsilon})$ and $t \geq \widehat{t} $ we have 
$u^\varepsilon(t,x)\geq 1-\delta$. 
\end{theorem}
Together, Theorem~\ref{teo:simplifyversion} and Theorem~\ref{no_blocking} 
say that there is a sharp transition at the critical radius 
$(\mathbbm{d}-1)/\nu$. A generalisation to a large
family of domains can be found in Theorem~\ref{saw-tooth} and Theorem~\ref{saw-tooth-invasion} below.

\begin{remark}
While, for simplicity and clarity of the proofs, we state 
Theorem~\ref{teo:simplifyversion} as a one-sided inequality, 
as a by-product of our methods one can obtain a lower bound 
on the right hand side of the domain. That is, in 
the context of Theorem~\ref{teo:simplifyversion}, if $x$ is such that 
$x_1 \geq M(k) \epsilon |\log(\epsilon)|$, then 
$u^\epsilon(x,t) \geq 1-\epsilon^k$. The proof follows exactly the 
same arguments as Theorem~\ref{no_blocking}; 
for more details see Theorem~2.4 in~\cite{gooding:2018}. 
Biologically, the function 
$u^\epsilon(x,t)$ models the proportion of a particular type in the population at position $x$ and time $t$. Hence 
Theorem~\ref{teo:simplifyversion} says that, if we have blocking, we 
have coexistence of the two alleles, 
with a narrow interface of width $\mathcal{O}(\epsilon |\log(\epsilon)|)$ 
between the $aa$ and $AA$ homozygotes near the opening of 
the domain. Note this is in sharp contrast to Theorem~\ref{no_blocking}, 
where we have fixation of the type $a$-allele across the whole population. 
\end{remark}

\subsubsection*{Other domains}

The domain $\Omega$ of Theorem~\ref{teo:simplifyversion} is very special. 
\cite{matano/nakamura/lou:2006} consider a plane curve evolving
according to Equation~(\ref{curvature plus constant flow}) 
in a two-dimensional cylinder with 
a periodic saw-toothed boundary. More precisely, they take
$$\Omega_\delta=\left\{ (x,y)\in\IR^2: x\in \big(-H-h_\delta(y), H+h_\delta(y)\big)\right\},$$ 
where 
%
%that is $\Omega\subseteq\IR^2$ is of the 
%form~(\ref{domain}), with 
%$\omega(x_1)=\big(-H-g(x_1), H+g(x_1)\big)$ where 
$H$ is a positive 
constant and $h_\delta(y)=\delta h_1(y/\delta)$ with $h_1$ being smooth, $1$-periodic, satisfying
\begin{equation}
\label{form of g for matano et al}
h_1(0)=h_1(1)=0, \quad h_1(y)\geq 0, \quad\forall y\in\IR;
\end{equation}
see Figure~\ref{matano domain} for an illustration.
They impose the additional condition that
at the points where they meet, the curve and the boundary of the domain are 
perpendicular,
%when the curve touches the boundary of the cylinder, it does so at 
%a right-angle, 
and their convention is that the normal to the curve points 
into the `bottom half' of the domain. 
This corresponds to the limit of the solution $u^\epsilon$ 
of $(AC_\epsilon)$ on their domain as 
$\epsilon\to 0$.
\begin{figure}
\begin{center}
(i)\includegraphics[height=2in]{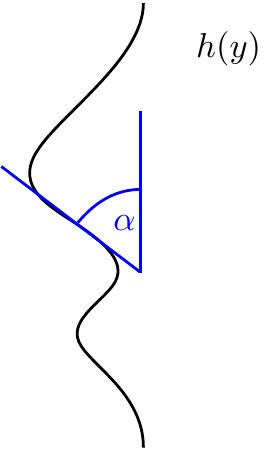}
\hspace{15 mm}
(ii) \includegraphics[height=2in]{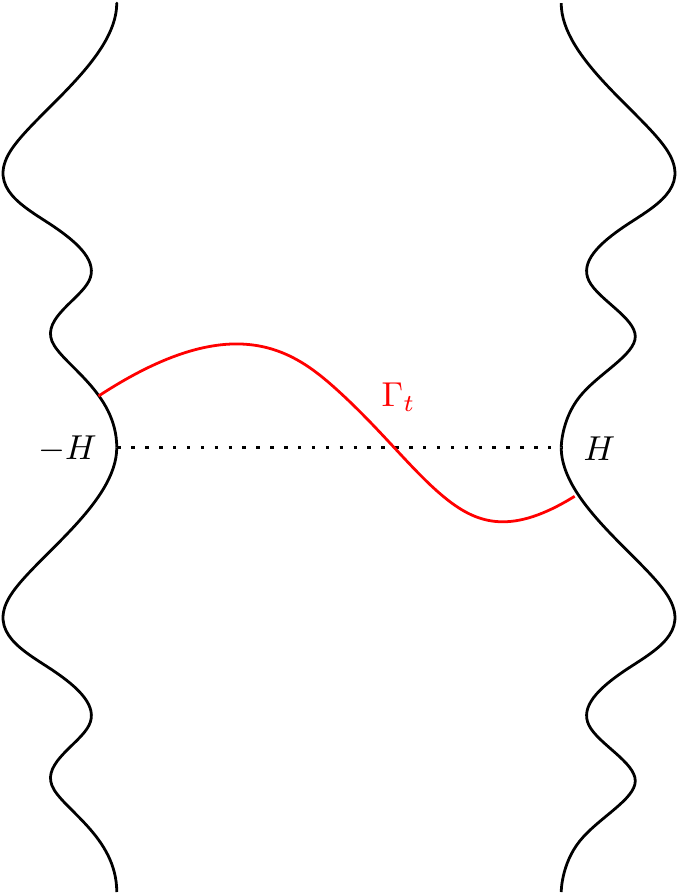}
\end{center}
\caption{(i) An example of function $h$ for the boundary of the domain 
(ii) An example of a periodic saw-toothed domain considered by
\cite{matano/nakamura/lou:2006}}
\label{matano domain}
\end{figure}
There will be no travelling wave solution to their 
equation, in the classical sense, unless the cylinder is flat, and so
they define a solution to be a {\em periodic} travelling wave if
${\v{\Gamma}}_{t+T_\delta}(s)={\v{\Gamma}}_t(s)+\delta$ for some 
$T_\delta>0$. Its effective speed is then $c_\delta:=\delta/T_\delta$.
Recalling that $h_\delta(y)=\delta h_1(y/\delta)$ and letting $\delta\to 0$, they then
investigate the homogenisation limit of the travelling wave, with 
corresponding speed $c_0 := \lim_{\delta \rightarrow 0} c_\delta$.
They show, in particular, that $c_0>0$
if $\nu H > \sin\alpha$ where $\alpha$ is
determined by $\tan\alpha=\max_y h'(y)$, but that the wave is
blocked for small enough $\delta$ when $\nu H < \sin\alpha$. In  
Section~\ref{sec:mordom} we shall sketch the proof of 
the following multidimensional analogue
of this blocking result for solutions to~$(AC_\varepsilon)$ in 
cylindrical domains
from our approach. 
%In our notation, $\Omega\subseteq\IR^2$ is of the 
%form~(\ref{domain}), with 
%$\omega(x_1)=\big(-H-g(-x_1), H+g(-x_1)\big)$.
\begin{theorem}
\label{saw-tooth}
Suppose that $u^\epsilon$ solves $(AC_\epsilon)$ where 
$\Omega\subseteq\IR^\dim$ is defined as in~(\ref{domain}) with
$$\phi(x_1)=\big\{\|x'\|\leq H+h(-x_1)\big\},$$
and $h$ being a positive, $C^1$ (not necessarily periodic) function. 
Suppose that,
\begin{equation} 
\inf_{z>0} \left\{  H + h(z) - 
%\sin(\tan^{-1}(h'(z))
\left( \frac{\mathbbm{d}-1}{\nu} \right)
\frac{h'(z)}{\sqrt{1+h'(z)^2}} \right\} < 0.  
\label{blocking condition saw  domain} 
\end{equation}
Fix $k\in\mathbb{N}$.
There exist $x_0<0$, $\widehat{\epsilon}(k)>0$ and $M(k)>0$ 
such that for all $\epsilon \in (0,\widehat{\epsilon})$ 
and $t \geq 0$,
\[\mbox{if } x=(x_1,...,x_{\mathbbm{d}})\in \Omega \mbox{ is such that } 
x_1 \leq x_0-M(k) \epsilon |\log(\epsilon)|\mbox{ then } 
u^\epsilon (x,t) \leq \epsilon^k. \]
\end{theorem}
In other words, the solution is blocked if the cylindrical domain
$\Omega$ opens out too quickly. Indeed, in the proof of 
Theorem~\ref{saw-tooth} we compute the angle at which the 
boundary \textit{opens up}. We shall see that 
condition~(\ref{blocking condition saw  domain}) ensures that this 
angle is `big enough'. As a consequence, 
when~(\ref{blocking condition saw  domain}) holds, we can insert a portion 
of a spherical shell of radius less than $(\dim -1)/\nu$ into the 
domain in such a way that expanding the 
shell radially one stays within the domain, at least for a short time. 
With this we can adapt the analysis performed for 
Theorem~\ref{teo:simplifyversion} and conclude that there is blocking. 
Conversely, if the angle is not big enough, we have invasion.
\begin{theorem} \label{saw-tooth-invasion}
Suppose that $u^\epsilon$ solves $(AC_\epsilon)$ where 
$\Omega\subseteq\IR^\dim$ is defined as in~(\ref{domain}) with
$$\phi(x_1)=\big\{\|x'\|\leq H+h(-x_1)\big\},$$
and $h$ being a positive, $C^1$ function. Suppose that
\begin{equation} \inf_{z>0} \left\{  H + h(z) - 
\left( \frac{\mathbbm{d}-1}{\nu} \right)\frac{h'(z)}{\sqrt{1+h'(z)^2}} \right\} > 0   \label{invasion condition saw  domain} \end{equation}
 then for all $x \in \Omega$
% \cap\{x : x_1 < -r_0\}$ 
and $\delta > 0$ 
there is \mbox{$\widehat{t} :=\widehat{t}(x_1, R_0,r_0) > 0$} and $\widehat{\varepsilon}$ such that, for all $\varepsilon \in (0,\widehat{\varepsilon})$ and $t \geq \widehat{t} $ we have 
$u^\varepsilon(t,x)\geq 1-\delta$. 
\end{theorem}

\begin{remark}
One can check that for $d=2$, if $H\nu<\sin\alpha$, then, for 
small enough $\delta$, Equation~(\ref{blocking condition saw  domain})
is satisfied and so we see blocking for the domain $\Omega_\delta$, 
whereas if $H\nu>\sin\alpha$ 
Equation~(\ref{invasion condition saw  domain}) is satisfied 
for any $\delta>0$ and 
there is no blocking. Thus we have recovered a multi-dimensional 
analogue of Theorem~2.1 in~\cite{matano/nakamura/lou:2006}, 
with weaker conditions on the function $h$.
\end{remark}

While conditions~(\ref{blocking condition saw  domain}) 
and~(\ref{invasion condition saw  domain}) are hard to verify in general, 
there are cases where it is trivial to determine if one of them holds. 
For example, narrowing domains, or those that only narrow beyond the point
at which their diameter is first smaller than $2(\mathbbm{d}-1)/\nu$, 
clearly satisfy~(\ref{invasion condition saw  domain}). See Figure~\ref{figure:sawdomex2} for examples of such domains. 

\begin{figure}[h!]
\begin{center}
(i)\includegraphics[height=1.2in]{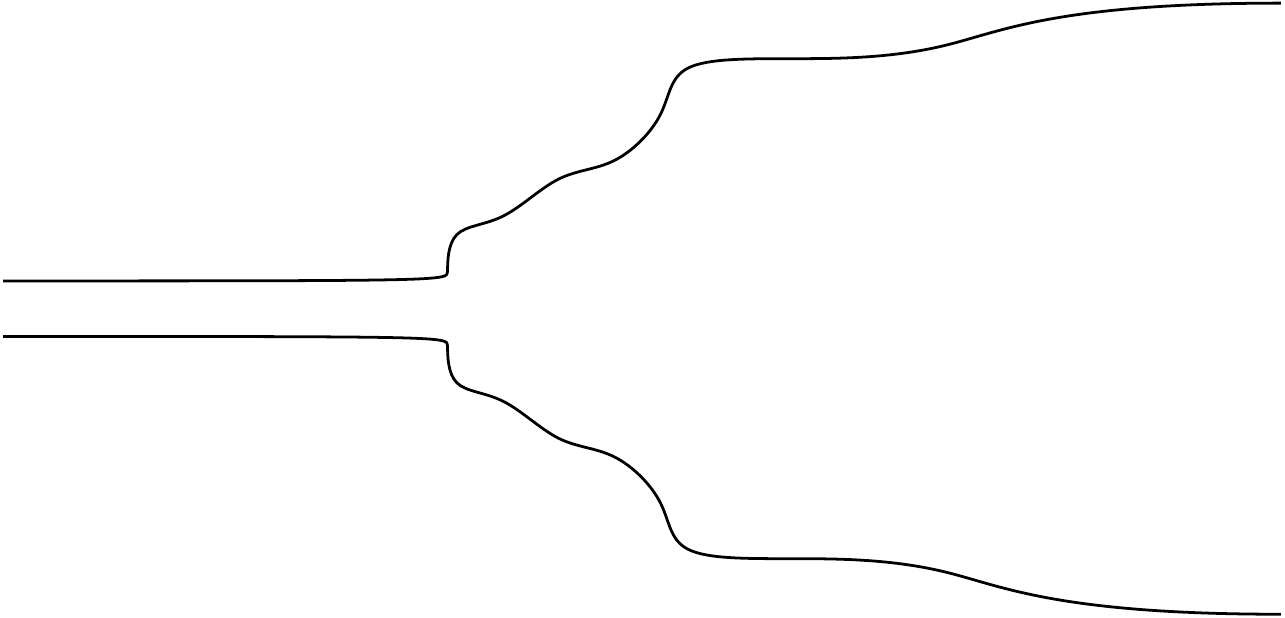}
\hspace{0.5in}
(ii) \includegraphics[height=1.2in]{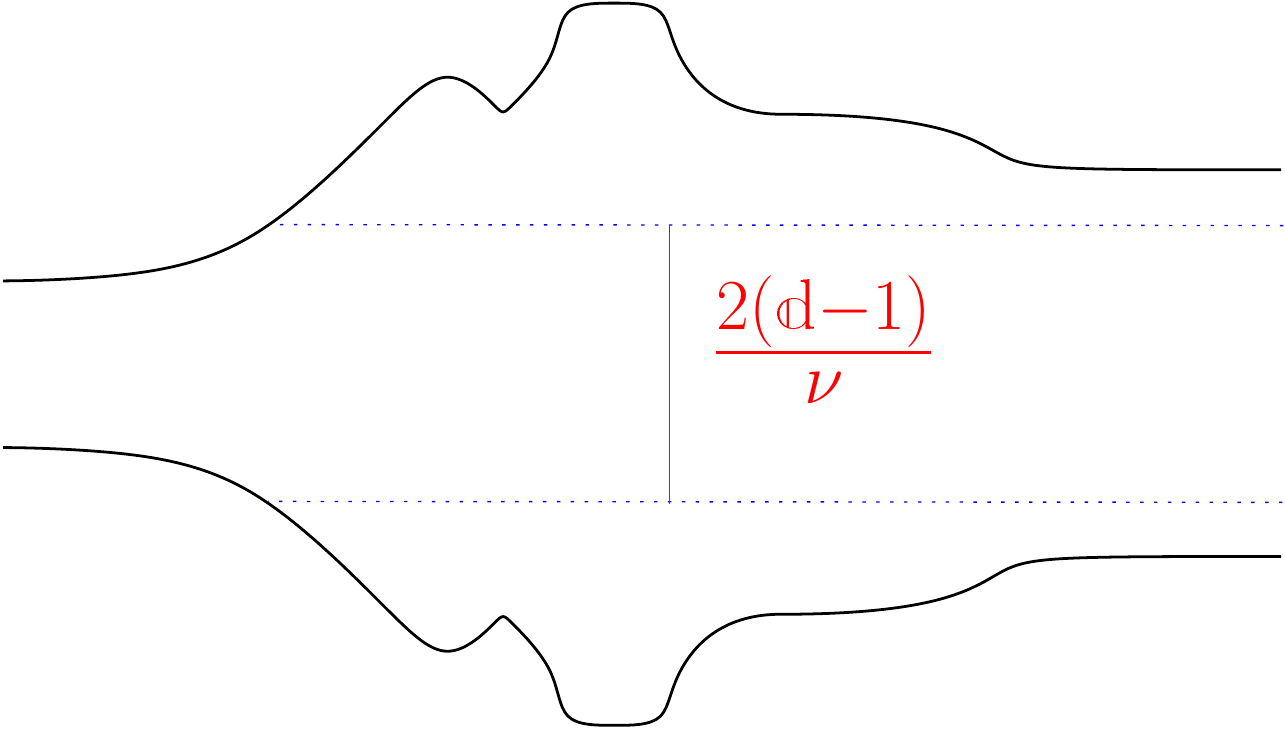}
\end{center}
\caption{(i) Example of a domain for which the opening only gets narrower.
(ii) Example of a domain for which the opening gets smaller
once it is less than $(\mathbbm{d}-1)/\nu$.}
\label{figure:sawdomex2}
\end{figure} 

\subsection{Adding noise}
In the previous subsection, we justified using 
equation~$(AC_\epsilon)$ to model the motion of a hybrid 
zone in population genetics. However, a deterministic equation 
like this rests on the assumption of infinite population density.
We should like to understand the effects of the random fluctuations
caused by reproduction in a finite population. 

Several ways of introducing noise into the Allen-Cahn equation have been 
considered in the literature. \cite{funaki:1999} considers~$(SAC_\epsilon)$
on a bounded domain in $\IR^2$,
and with an additional additive noise term of the
form $\xi^\epsilon(t)/\epsilon$, where $\xi^\epsilon$ is centred
and smooth in $t$, but behaves like white noise in the limit as 
$\epsilon\to 0$. Once again the solution generates an 
interface as $\epsilon\to 0$, which Funaki calls {\em randomly perturbed
motion by curvature}. \cite{alfaro/antonopoulo/karali/matano:2018} 
consider~$(SAC_\epsilon)$ with the same form of additive noise 
as~\cite{funaki:1999}, this time on a bounded domain in $\IR^\dd$. Their
results show that, just as in the deterministic case, an interface 
develops in a very short time, and that the law of motion of the interface is now given by mean curvature flow perturbed by white noise. The profile of the
solution near the interface is not destroyed
by the random noise, as long as the noise depends only on the time 
variable. \cite{lee:2018} considers a space-time noise, but the noise
is smooth in space, and although it is shown that an interface is 
generated, the law of motion of the interface is not established. 
\cite{hairer/ryser/weber:2014} consider the equation
\[ \partial_t u= (\Delta u+u-u^3)dt+\sigma d W, \]
on $\IR^2$, where $W$ is a space-time white noise, mollified in space. Setting
$v=(1+u)/2$, we recover $(SAC_\epsilon)$ with $\epsilon=1$ and
an additional mollified white noise. They show that, if the mollifier 
is removed, then the solution 
converges weakly to zero, but that if the intensity of the noise 
simultaneously converges to zero sufficiently quickly, 
they recover the solution to the deterministic equation. 
In other words, unless the noise is small, it
can completely destroy the structure
of the deterministic equation. 

Additive white noise (with or without a spatial component) 
is not a good model for randomness due to reproduction
in a biological population, usually called {\em genetic drift}, and so
these papers 
do not resolve the question
of whether hybrid zones will still evolve (approximately) according to
curvature flow in a population evolving in a two dimensional space.
In one spatial dimension, one can justify modelling genetic drift by
adding a noise term of the form $\sqrt{u(1-u)}dW$ (for a space-time 
white noise $W$). In that setting, \cite{gooding:2018}, building 
on~\cite{funaki:1995}, investigates the fluctuations in the 
position of the hybrid zone (see also~\cite{lee:2018}). However,
the corresponding equation has no solution in two dimensions, and the equation 
obtained by replacing white noise with a mollified white noise does not 
arise naturally as a limit of an individual based model. 
In~\cite{etheridge/freeman/penington:2017} 
and~\cite{gooding:2018}, a variant of the {\em spatial $\Lambda$-Fleming-Viot
process} is used to overcome this problem. 
It is shown that, at least if the genetic drift is sufficiently
weak, the (approximate) structure of the deterministic equation is 
preserved. In Section~\ref{sec:stochastic}
we use an approach that mimics that used to study the interaction of 
genic selection with spatial structure
in~\cite{etheridge/freeman/penington/straulino:2017} to provide a
stochastic analogue of Theorem~\ref{teo:simplifyversion}. Furthermore, 
we prove a complementary result, in which we identify rather precisely 
the relative strength of 
genetic drift and selection that results in breakdown in the structure of the
deterministic equation. 

The key to understanding blocking in the presence of noise is to 
establish whether a stochastic analogue of 
Theorem~\ref{mitch theorem}
holds on the whole of 
Euclidean space, and so, for the purposes of this introduction, we shall
take $\Omega=\IR^\dim$.

First, we define a version of the Spatial $\Lambda$-Fleming-Viot process 
with selection that 
provides a stochastic analogue of the solution to 
Equation~$(AC_\epsilon)$. We omit details of the construction, 
which mirrors the approach taken in~\cite{etheridge/veber/yu:2020} 
in the case of genic selection. 
At each time $t$, the random function 
$\{w_t(x) : \, x\in \mathbb{R}^\mathbbm{d} \}$ 
is defined, up to a Lebesgue null set of $\mathbb{R}^\mathbbm{d}$, by
\begin{equation*}
w_t(x):= \{ \text{proportion of type } a \text{ alleles at spatial position }x \text{ at time } t \}.
\end{equation*}
In other words, if we sample an allele from the point $x$ at time $t$, the probability that it is of type $a$ is $w_t(x)$.
\begin{remark}
\label{slfvs only lebesgue ae}
As is usual for the spatial $\Lambda$-Fleming-Viot processes, $w_t(x)$ will only be defined up 
to a Lebesgue-null set $\mathcal{N}$. Since it is convenient to extend the definition of $w_t(x)$ 
to all of $\mathbb{R}^\mathbbm{d}$, we set $w_t(x) =0$ for all $x \in \mathcal{N}$.
\end{remark}
A construction of an appropriate state space for 
$x\mapsto w_t(x)$ can be found in \cite{veber/wakolbinger:2015}. 
Using the identification
$$
\int_{\mathbb{R}^\mathbbm{d}} \big\{w_t(x)f(x,a)+ (1-w_t(x))f(x,A)\big\}\, dx=\int_{\mathbb{R}^\mathbbm{d}\times \{a,A\}} f(x,\kappa) M(dx,d\kappa),
$$
this state space is in one-to-one correspondence with the space
$\mathcal{M}_\lambda$ of measures on $\mathbb{R}^\mathbbm{d} \times\{a,A\}$ with `spatial marginal' Lebesgue measure,
which we endow with the topology of vague convergence. By a slight abuse of notation, we also denote the
state space of the process $(w_t)_{t\in\mathbb{R}_+}$ by 
${\cal M}_\lambda$.

\begin{definition} [Spatial $\Lambda$-Fleming-Viot process
with selection (SLFVS)]
\label{FVSdefn}
Fix $u,\gamma\in (0,1]$, $\v{s} \in (0,1/(1+\gamma))$, $\mathcal{R} > 0$. 
Let $\mu$ be a finite measure on 
$(0,\mathcal{R}]$. Let $\Pi$ be a Poisson Point Process on  $\mathbb{R}_+\times \mathbb{R}^\mathbbm{d} \times (0,\mathcal{R}]$ 
with intensity measure 
\begin{equation}\label{slfvdrive}
dt\otimes dx\otimes \mu(dr). %\nu_r(du).
\end{equation}
The {\em spatial $\Lambda$-Fleming-Viot process with selection (SLFVS)} driven by $\Pi$, with {\em selection coefficient} $\v{s}$ and
{\em impact parameter} $u$, is the $\mathcal{M}_\lambda$-valued process $(w_t)_{t\geq 0}$ with dynamics given as follows.

If $(t,x,r)\in \Pi$, a reproduction event occurs at time $t$ within the 
closed ball $B(x,r)$ of radius $r$ centred on $x$. 
With probability $1-(1+\gamma)\v{s}$ the event is {\em neutral}, in which case:
\begin{enumerate}
\item Choose a parental location $z \in \RR^\dd$ uniformly at random 
in $B(x,r)$, and a parental type, $\alpha_0$, 
according to $w_{t-}(z)$. That is $\alpha_0=a$ with probability $w_{t-}(z)$ and $\alpha_0=A$ 
with probability $1-w_{t-}(z)$.
\item For every $y\in B(x,r)$, set 
$w_t(y) = (1-u)w_{t-}(y) + u 1_{\{\alpha_0=a\}}$.
\end{enumerate}
With the complementary probability $(1+\gamma)\v{s}$ the event is {\em selective}, in which case:
\begin{enumerate}
\item Choose three `potential' parental locations $z_1, z_2, z_3 \in \RR^\dd$ 
independently and uniformly at random
from $B(x,r)$. At each of these sites sample
`potential' parental types $\alpha_1$, $\alpha_2$, $\alpha_3$, according to 
$w_{t-}(z_1), w_{t-}(z_2), w_{t-}(z_3)$, respectively. 
Let $\widehat{\alpha}$ denote the most common allelic type 
in $\alpha_1,\alpha_2,\alpha_3$, except that if precisely one of 
$\alpha_1$, $\alpha_2$, $\alpha_3$, is $a$, with probability
$\frac{2 \gamma}{3+3\gamma}$ set $\widehat{\alpha}=a$. 
\item For every $y\in B(x,r)$ set 
$w_t(y) = (1-u)w_{t-}(y) + u 1_{\{\widehat{\alpha}=a\}}$.
\end{enumerate}
\end{definition}
\begin{remark}
Sampling parental locations, and then parental types, is convenient for 
identifying the dual process of branching and coalescing ancestral
lineages that we introduce in Definition~\ref{SLFVS dual}.
However, from the perspective of the 
SLFVS it would be equivalent to sample types independently and
uniformly at random from
the region affected by the event.
\end{remark}
Before going any further, we explain the origin of the reproduction
rule
% for 
%selecting `potential' parents 
in Definition~\ref{FVSdefn}.
Comparing to our justification of Equation~$(AC)$, recalling that $w$ is
the proportion of $a$-alleles in the population, we first write
\begin{multline}
\label{decomp repro events}
\v{s}w(1-w)(2w-(1-\gamma))
\\
=\left( (1-(1+\gamma)\v{s})w + (1+\gamma)\v{s}\left(w^3+3w^2(1-w) + 
\left(\frac{2\gamma}{3(1 + \gamma)} \right)3w(1-w)^2 \right)  - w \right).
\end{multline}
In the SLFVS framework, 
reproduction events arrive as a Poisson process (as opposed to the 
deterministic generations times in our justification of~$(AC)$). 
With probability $1-(1+\gamma)\v{s}$ 
an event is neutral, so that the chance 
that offspring are of type $a$ is simply the probability $w$ that a 
randomly chosen parent is of type $a$, and we recognise the first term
on the right of~(\ref{decomp repro events}). With probability 
$(1+\gamma)\v{s}$, the
event is selective. If we sample three individuals from the population, 
the probability that the majority are type $a$ is $w^3+3w^2(1-w)$; whereas
the probability that exactly one is type $a$ is $3w(1-w)^2$. In the latter
case, we multiply further by $2\gamma/(1+\gamma)$ to recover the probability
that the offspring are type $a$, and we recognise the second and third terms 
on the right of~(\ref{decomp repro events}). 
In total then, 
Equation~(\ref{decomp repro events}) represents the change in proportion
of $a$ alleles in the portion of the population replaced during the event.
\begin{remark}
\label{two types of selective events}
We could equally have taken two types of selective events, one 
corresponding to selection against heterozygosity, and one to genic 
selection. To see why, we rewrite the part of~(\ref{decomp repro events})
corresponding to selective events as
\begin{multline}
(1+\gamma)\v{s}\left(w^3+3w^2(1-w)+\frac{2\gamma}{3(1+\gamma)}3w(1-w)^2
-w\right)\\
=\v{s}\left(w^3+3w^2(1-w)-w\right)
+\gamma\v{s}\left(w^3+3w^2(1-w)+2w(1-w)^2-w\right)\\
=\v{s}\left(w^3+3w^2(1-w)-w\right)
+\gamma\v{s}\left(w^2+2w(1-w)-w\right)
%\\
%=\v{s}\left(w^3+3w^2(1-w)-w\right)
%+\gamma\v{s}w\left(1+w(1-w)+(1-w)^2-1\right)\\
%=\v{s}\left(w^3+3w^2(1-w)-w\right)
%+\gamma\v{s}w\left(1+w(1-w)+(1-w)^2-1\right)
\end{multline}
This suggests that with probability $\v{s}$ an event
corresponds to selection against heterozygosity: three potential
parents are sampled and offspring adopt the type of the majority of 
those individuals; with probability $\gamma\v{s}$ an event corresponds
to genic selection: {\em two} potential parents are sampled and if 
either of them is type $a$, then the offspring is of type $a$.

Although this leads to the same process of allele frequencies as the 
apparently more complex mechanism that we introduced in
Definition~\ref{FVSdefn}, in our proof it will be convenient to 
have a single rule for selective events, based on three potential
parents, which will be encoded in 
the function $g$ of Equation~(\ref{ident1}) below.
\end{remark}
To study the relationship between genetic drift and selection we will 
introduce two possible scalings for the SLFVS. While the choice of the scaling parameters may seem obscure, once we have introduced 
a branching and coalescing 
dual for the SLFVS in Section~\ref{sec:stochastic}, 
the reason for these choices will become clear.
%it will become clear why we chose the parameters the way we do. 
If in both cases we fix the same values for the parameters that 
dictate selection (corresponding to 
$\v{s}_n$ and $\gamma_n$ below), the difference between the 
two scalings is entirely in the strength of the genetic drift, while
the `deterministic part' of the evolution (corresponding to dispersion and
selection) is identical. We return to this in 
Remark~\ref{stochastic scaling}.

\begin{assumption}
\label{cond on epsilon n}
For both regimes, we suppose that $\epsilon_n$ is 
a sequence such that $\epsilon_n\to 0$ and $(\log n)^{1/2}\epsilon_n\to\infty$ 
as $n\rightarrow\infty$. 
\end{assumption}

\subsubsection*{Weak noise/selection ratio}

Our first scaling is what we shall call the {\em weak noise/selection ratio} 
regime. In this regime selection overwhelms genetic drift. 
It mirrors that explored 
in~\cite{etheridge/freeman/penington:2017} and is also considered in~\cite{gooding:2018}.
For each $n\in \mathbb{N}$, and some $\beta \in (0,1/4)$, we define the finite measure 
$\mu^n$ on $(0, \mathcal R_n]$, where $\mathcal R_n = n^{-\beta}\mathcal R$, by 
$\mu^n(B)=\mu(n^{\beta}B)$ for all Borel subsets $B$ of $(0,\infty)$. 
In the weak noise/selection ratio regime 
the rescaled SLFVS is driven by the Poisson point process $\Pi^n$ on 
$\mathbb{R}^+ \times \mathbb{R}^{\mathbbm{d}} \times (0,\infty)$ with intensity measure 
\begin{equation}\label{eq:slfvs_intensity_intro_weak}
n dt\otimes n^{\beta} dx\otimes \mu^n(dr).
\end{equation}
Here the linear dimension of
the infinitesimal region $dx$ is
scaled by $n^\beta$ (so that when we integrate, the volume of a region
is scaled by $n^{\mathbbm{d} \beta}$).
Let $\nu>0$. We denote by
$u_n$ the impact parameter and by $\v{s}_n$ the selection parameter at
the $n$th stage of the scaling. They will be given by
\begin{equation}
\label{scalings_weak}
\gamma_n = \nu \epsilon_n, \qquad u_n = \frac{u}{n^{1-2\beta}}, \qquad 
\v{s}_n = 
\frac{1}{\epsilon_n^{2} n^{2\beta}}.
%\frac{1+\epsilon_n \nu}{\epsilon_n^{2} n^{2\beta}}.
\end{equation}
Adapting the proof of Theorem~1.11 in~\cite{etheridge/veber/yu:2020}, 
and arguments in Section~3 of~\cite{etheridge/freeman/penington:2017}, 
one can show that under this scaling, for large $n$,
the SLFVS will be close to
the solution of problem $(AC_\epsilon)$.

\subsubsection*{Strong noise/selection ratio}

We shall refer to our second scaling as the {\em strong noise/selection
ratio} regime. In this regime genetic drift overcomes selection. 
%In this case we consider general 
%conditions for the selection and impact parameters. 
In this scenario, we consider any sequence of impact parameters $(u_n)_{n \in \mathbb{N}} \subseteq (0,1)$.
Consider $\beta \in (0,1/2)$ and let $\widehat{u}_n := u_n n^{1-2\beta}$. 
We scale time by $n/\widehat{u}_n$ and space 
by $n^{\beta}$. At the $n$th stage of the rescaling, 
$\Pi^n$ is a Poisson measure on  
$\mathbb{R}^+ \times \mathbb{R}^{\mathbbm{d}} \times (0,\infty)$ 
with intensity measure 
\begin{equation}
\label{eq:slfvs_intensity_intro_strong}
\frac{n}{\widehat{u}_n} dt\otimes n^{\beta} dx\otimes \mu^n(dr).
\end{equation}
We consider a sequence of selection coefficients, $(\v{s}_n)_{n \in \mathbb{N}} \subseteq (0,1)$,
satisfying one of the following conditions:
\begin{align}
\label{scaling_strong_sn} 
 \begin{cases} \v{s}_n n^{2\beta} \rightarrow 0 & 
\liminf_{n \rightarrow \infty} u_n \log n < \infty \text{ or } 
\mathbbm{d} \geq 3,\\ 
\frac{\v{s}_n n^{2 \beta}}{u_n \log n} \rightarrow 0 & \liminf u_n \log n = \infty \text{ and } \mathbbm{d} = 2. \end{cases}
\end{align}
The strength of genetic drift (noise) is determined by the impact parameter.
The first case includes some choices of impact that were allowed in the
first (weak noise/selection ratio) regime; it is the strength of drift
{\em relative to selection} that matters.
In this regime, we can 
take the parameters $(\gamma_n)_{n\in \mathbb{N}}$ that dictate the 
asymmetry in our selection to 
be any sequence in $(0,1)$. 
%As we will see in Section \ref{sec:teofv_strongregime}, the restriction on $\v{s}_n$ given by (\ref{scaling_strong_sn}) is made in such a way that the set of \textit{potential parents} consist only of a single individual moving, approximately, as Brownian motion.

\subsubsection*{The stochastic result}

With the two scaling regimes defined we can state a result. 
Recall that
we are working on the whole of Euclidean space.
\begin{theorem}
\label{noisy circles}
Write $\rho_* = (\dd-1)/\nu$. Let $(w_t^n)_{t\geq 0}$ be the scaled 
SLFVS with initial condition $w_0^n(x)=\1_{B(0,\rho_*)}(x)$. Fix $t>0$.
\begin{enumerate}
\item Under the weak noise/selection ratio regime, 
%as $n$ goes to infinity, the hybrid zones in $(w_t^n)_{t \geq 0}$ reach an equilibrium in 
%the following sense. For 
for each $k\in\mathbb{N}$, there exist $n_*(k)<\infty$, 
and $d_*(k)\in(0,\infty)$, such that for all $n\geq n_*$,
% we have that
\begin{enumerate}
\item for almost every $x$ such that 
$\Vert x \Vert \geq \rho_* + d_* \epsilon_n |\log \epsilon_n|$, we have
$\mathbb{E}\left[w^n_t(x)\right]\leq \epsilon_n^k$; 
\item for almost every $x$ such that 
$\Vert x \Vert \leq \rho_* -d_* \epsilon_n |\log \epsilon_n|$, we have
$\mathbb{E}\left[w^n_t(x)\right]\geq 1- \epsilon_n^k$.
\end{enumerate}
\item Under the strong noise/selection ratio regime, there is $\sigma^2>0$ 
%as $n$ goes to infinity, $(w_t^n)_{t \geq 0}$ evolves approximately according to heat flow. 
%That is, for every $\epsilon > 0$ and $t \geq 0$, 
and $n_*$, such that for all $n \geq n_*$, and all $x\in\IR^\dim$, 
\begin{equation} 
\label{eq:bd}
\left|\mathbb{E}_{w_0}\left[w_{t}^n(x)\right]  
- \PP_x[\Vert W(\sigma^2t) \Vert \leq \rho_*] \right| \leq \epsilon,
\end{equation}
where $(W_t)_{t \geq 0}$ is a standard Brownian motion in
$\IR^\dim$ and the subscript $x$ on $\IP_x$ indicates that $W(0)=x$.
\end{enumerate}
\end{theorem}
More generally, one can show that in the strong noise/selection ratio regime
for $x\neq y$, $w_t(x)$ and $w_t(y)$ decorrelate as $n\to\infty$.
In other words, in the weak noise/selection ratio 
regime the SLFVS behaves approximately as
the deterministic equation $(AC_\epsilon)$, while in the strong noise/selection ratio regime the 
genetic drift is strong enough to overcome the effects of selection 
and it breaks down the interface. (The corresponding breakdown of the 
interface under the strong noise/selection ratio regime 
in one dimension requires us to replace
$u_n\log n$ above by $u_n n^{1/2}$.)
\begin{remark}
As we discuss in a little more detail in 
Section~\ref{sec:teofv_strongregime},
the conditions~(\ref{scaling_strong_sn}) 
are essentially optimal. In this regime, if the scaled limits of $\v{s}_n$ 
tends
to a positive constant, rather than $0$, then we expect the limit to 
behave like the weak noise/selection regime, except with $\rho^*$ 
replaced by a smaller value.
\end{remark}
Everything in this subsection, and in particular Theorem~\ref{noisy circles},
has concerned the SLFVS on the whole of $\IR^\dim$.
In the context of the focus of this paper, this result is enough to 
determine conditions under which genetic drift will break down the 
effect of the curvature flow and prevent blocking. A technical point that
we have to address is that the SLFVS 
has only previously been studied on all of $\IR^\dim$ or on a torus. 
%However, it is still a relevant question whether Theorem~\ref{noisy circles} can be generalised to understand genetic drift in a $\Omega$ such as the one in Figure~\ref{fig:omega}. 
In Appendix~\ref{SLFV with reflecting bdry}, we present an 
approach to defining the 
SLFVS on the domain $\Omega$ of Figure~\ref{fig:omega} with %to
%incorporate 
a natural analogue of the reflecting boundary condition
in~$(AC_\varepsilon)$. We call this process the SLFVS on $\Omega$. %While we defer the technical definition to Appendix~\ref{SLFV with reflecting bdry}, from the way we define the process it will be not difficult to see the following holds.
\begin{theorem} \label{teo:noisy circles in omega}
Let $\rho_* = (\dd-1)/\nu$ and suppose $r_0 < \rho_*$. 
Let $(w_t^n)_{t\geq 0}$ be the scaled SLFVS on $\Omega$ with 
initial condition $w_0^n(x)=1_{x_1 \geq 0}$.
\begin{enumerate}
\item Under the weak noise/selection ratio regime, 
%as $n$ goes to infinity, the hybrid zones in $(w_t^n)_{t \geq 0}$ presents blocking in the following sense. For 
for any $k\in\mathbb{N}$, there exist $n_*(k)<\infty$, 
and $a_*(k),d_*(k)\in(0,\infty)$ such that for all $n\geq n_*$ and all $t>0$ we have that
\begin{align*}
 &\text{for almost every } x \text{ such that } 
x_1 \leq -d_* \epsilon_n |\log \epsilon_n|, \qquad \mathbb{E}\left[w^n_t(x)\right]\leq \epsilon_n^k.
 \end{align*}
 \item Under the strong noise/selection ratio regime, a sharp 
interface does not develop 
as $n$ goes to infinity. %$(w_t^n)_{t \geq 0}$ converges weakly %evolves 
%approximately according 
%to heat flow with normal reflection on
%the boundary. That is, 
Instead, there is $\sigma^2>0$ such that
for every $\epsilon > 0$ and $t \geq 0$, 
there is a reflected Brownian motion, $(W_t)_{t \geq 0}$, 
and $n_*$, such that for all $n \geq n_*$ 
\begin{equation} \label{eq:bd:noisy circles omega}
\left|\mathbb{E}_{w_0}\left[w_{t}^n(x)\right]  - \PP_x[W(\sigma^2 t) \geq 0] 
\right| \leq \epsilon.
\end{equation}
 \end{enumerate}
\end{theorem}

Theorem~\ref{teo:noisy circles in omega} provides an adaptation of 
Theorem~\ref{noisy circles} to the domain $\Omega$ of Figure~\ref{fig:omega}.
In that case, under the weak noise/selection ratio
regime, we still see blocking, but in the strong noise/selection ratio
regime, the proportion of $a$-alleles spreads 
approximately according to heat flow in $\Omega$ (with a reflecting 
boundary condition), and, in particular, blocking no longer occurs.

%%%%%%%%%%%%%%%%%%%%%%%%%%%%%%%%%%%%%%%%%%%%%%%%%%%%%%%%%%%%%%%%%%%%%%
\subsection{Outline of the paper}

To prove Theorem~\ref{teo:simplifyversion}, 
we proceed as follows. First, in Section~\ref{sec:rep}
we provide a stochastic representation for the solution of equation
$(AC_\epsilon)$. This is entirely analogous to that 
in~\cite{gooding:2018} for the equation on $\IR^{\dim}$. The 
solution at the point $x$ at time $t$ 
is the expected result of a `voting procedure' defined on the
tree of 
paths traced out 
by branching reflecting Brownian motion in $\Omega$ up to 
time $t$, starting from a single
individual at the point $x$. Some immediate properties of the solution that
follow from this representation are presented in Section~\ref{sec:kr}. 
In particular, it allows us to compare
the solution from different initial conditions, which in turn 
allows us to bound the solution to $(AC_\epsilon)$ from 
above by that of the same
equation with a bigger initial condition. This will be convenient in
formalising the idea of the mean curvature flow `fighting against' 
the constant flow. This is developed in Section~\ref{sec:ups}. 
In Section~\ref{sec:co} we present a coupling result which is the 
key to establishing a concrete bound on the effect of the mean curvature flow. 
With this, we prove the blocking result for the equation with a larger
initial condition and so, a fortiori, for~$(AC_\epsilon)$. 
We first bound the solution over a small window of time 
in Section~\ref{sec:genint}, and then use a bootstrapping argument
in Section~\ref{sec:blo} to show that this bound is uniform in time. 
Theorem~\ref{no_blocking} is proved in Section~\ref{proof of no blocking}.
In Section~\ref{sec:mordom} we sketch the extension of
these results
to more general domains. In particular, we present key elements of
the proof of Theorem~\ref{saw-tooth}.

In Section~\ref{sec:stochastic} we turn to the 
stochastic version of our problem. 
The results in this section will be on the whole of $\IR^\dim$. As usual for
models of this type, the key to the analysis will be a dual process of
branching and coalescing lineages which we introduce in 
Section~\ref{sec:defdual}. The duality function relating this process to
the SLFVS will involve the same voting procedure that we use for the
branching reflecting Brownian motion in the
deterministic setting. 
Theorem~\ref{noisy circles},
is proved in Section~\ref{sec:stochastic}.
In the weak noise/selection ratio regime, 
the key will be to show that asymptotically we no longer see coalescence
in the dual process, which
is then well-approximated by branching Brownian motion.
This tells us that the solution to the stochastic equation will be 
close to that of~$(AC_\epsilon)$. In the strong
noise/selection ratio regime, 
branching events in the dual process are quickly annulled by coalescence and,
over long time scales, the dual is close to 
a single Brownian motion. This allows us to approximate the 
stochastic evolution by heat flow, in sharp contrast to the first regime.

Boundary conditions
have not previously been explicitly treated in the SLFVS. 
As will become clear, the details of
what happens at the boundary should not affect our results and so we 
relegate a description of how one can construct what can be reasonably called
an SLFVS with reflecting boundary conditions (for some rather special
domains) to the Appendix. 

%We shall borrow a lot of ideas and results 
%from~\cite{etheridge/freeman/penington:2017} and~\cite{gooding:2018},
%but there are some important differences in this work.

\begin{flushleft} \textbf{Acknowledgements} The third author is
supported by the ANID/Doctorado en el extranjero doctoral scholarship, 
grant number 2018-72190055. The authors are grateful to Sarah Penington for 
her careful reading of, and helpful comments on, 
a preliminary version of this work.
\end{flushleft}
\section{Deterministic Model}

\subsection{A stochastic representation of the solution to $(AC_\epsilon)$} 
\label{sec:rep}

Our first goal is to present a stochastic representation of the solution 
of equation $(AC_\epsilon)$. This is an essentially trivial adaptation of 
the representation of the solution to the corresponding equation on
$\IR^{\dim}$ presented in~\cite{gooding:2018}, which builds 
on~\cite{etheridge/freeman/penington:2017}, which in turn is closely related to 
results of~\cite{demasi/ferrari/lebowitz:1986}, but we include it here as it
will be central to what follows.

The representation is based on ternary branching reflected Brownian motion
in $\Omega$. The dynamics of this process 
has three ingredients:
\begin{enumerate}
    \item {\em Spatial motion:} 
during its lifetime each particle moves according 
to a reflected Brownian motion in $\Omega$. %That is a solution to the Stochastic differential equation
    %\begin{equation} \label{eq:refbm} W_t = x + B_t + \int_{0}^{t}\widehat{n}_{\partial \Omega}(W_s) L_t^{\partial D}(W_s), \end{equation}
    %with $\widehat{n}_{\partial \Omega}$ the (inward) normal vector of the boundary of $\Omega$ and $L_t^{\partial D}(W_s)$ the symmetric local time at the boundary of the process $W$.
    \item {\em Branching rate, $V$}: 
each individual has an exponentially distributed lifetime with parameter $V$. 
    \item {\em Branching mechanism:} when a particle dies it leaves behind 
(at the location where it died) exactly three offspring.
%a random number of offspring with probability generating function $\Phi(s) = \sum_{k=0}^{\infty} p_k s^k$. 
Conditional on the time and place of their birth, the offspring 
evolve independently of each other, and in the same manner as the parent.
\end{enumerate}
%We remark that (\ref{eq:refbm}) posses unique pathwise solution in $\Omega$, see Remark 3.1 in \cite{Lionssderc}. From now on we will always assume that the individuals have exactly three offspring, that is $\Phi(s)=s^3$. Also to 
\begin{assumptions}
\label{assumptions on W} 

~

\begin{enumerate}
\item For consistency with the PDE literature \textbf{we adopt the convention 
that all Brownian motions run at speed} $2$ (and so have infinitesimal 
generator $\Delta$). 

\item We shall write $\gamma_\epsilon := \epsilon \nu$ and set the 
branching rate $V= (1+\gamma_\epsilon)/\epsilon^2$.
\end{enumerate}
\end{assumptions}
The stochastic representation of~$(AC_\epsilon)$
is reminiscent of the classical representation of 
the solution to the Fisher-KPP equation in terms of binary branching 
Brownian motion (\cite{skorohod:1964, mckean:1968}), but now it will
depend not just on the leaves of the tree swept out by the branching 
Brownian motion, but on the whole tree. We need some notation.
We write $\W(t)$ for the {\em historical process} of the ternary branching
reflected Brownian motion; that is, $\W(t)$ records the spatial position of 
all individuals alive at time $s$ for all $s\in [0,t]$.

The set of Ulam-Harris labels 
$\mathcal{U}=\{ \emptyset \} \cup \bigcup_{k=1}^\infty \{1,2,3\}^k$
will be used to label the vertices of the infinite ternary tree, 
which consists of a single vertex of degree $1$, which we call the root, 
and every other vertex having degree $4$. 
The unique path from a given vertex to the root 
distinguishes exactly one of its four neighbours, and we consider that
neighbour to be its parent,
and the remaining three neighbours to be its offspring. The root is 
given label $\emptyset$, and if a vertex has label 
$(i_1 ,..., i_n) \in \mathcal{U}$, its offspring have labels 
$(i_1,...,i_n,1)$, $(i_1,...,i_n,2)$ and $(i_1,...,i_n,3)$. 
For example, $(3,1)$ is the first child of the third child of the 
initial ancestor $\emptyset$.
\begin{definition}
We say that $\mathcal{T}$ is a time-labelled ternary tree if $\mathcal{T}$ 
is a finite subtree of the infinite ternary tree $\mathcal{U}$, and each 
internal vertex $v$ of the tree is labelled with a time $t_v>0$, where $t_v$ 
is strictly greater than the corresponding label of the parent vertex of $v$.
\end{definition}
If we ignore the spatial positions of the individuals in our ternary
branching reflected Brownian motion, then each realisation of the process 
traces out a time-labelled ternary tree which records the relatedness 
({\em the genealogy}) between individuals, and associates a time to
each branching event. 
We use $N(t) \subset \U$ to denote the set of individuals at time $t$. 
We abuse notation to write $(W_i(t))_{i \in N(t)}$ for the spatial 
location of the individuals at time $t$ and $(W_i(s))_{0 \leq s \leq t}$ 
for the unique path that connects the leaf $i$ to the root. 
We write $\mathcal{T}(\W(t))$ for the time-labelled ternary tree determined
by the branching structure of $\W(t)$. 

Now, for a fixed function $p: \mathbb{R}^\mathbbm{d} \rightarrow [0,1]$ and 
 parameter $\gamma_\epsilon \in (0,1)$, we define a {\em voting procedure} on the 
tree $\mathcal{T}(\W(t))$:
\begin{enumerate}
    \item Each leaf $i$ of $\mathcal{T}(\W(t))$ independently votes $1$ 
with probability $p(W_i(t))$, and $0$ otherwise.
    \item At each branching point in $\mathcal{T}(\W(t))$ the vote of 
the parent particle $j$ is the {\em majority vote} of the children 
$(j,1)$,$(j,2)$ and $(j,3)$, unless precisely one of the offspring votes 
is $1$, in which case the parent votes $1$ with probability 
$\frac{2\gamma_\epsilon}{3+3 \gamma_\epsilon}$, 
otherwise the parent votes $0$.
\end{enumerate}
This defines an iterative voting procedure, which runs inwards 
from the leaves of $\mathcal{T}(\W(t))$ to the root $\emptyset$. 
\begin{definition} \label{def:vp}
With the voting procedure above, we define 
$\mathbb{V}_p^{\gamma}(\W(t))$ to be the vote associated to the 
root $\emptyset$. 
\end{definition}
Recalling Assumptions~\ref{assumptions on W}, 
we use $\PP_x^{\epsilon}$ to denote the law of $\W$ when started from 
a single individual at the point $x\in\Omega$.
We have the following representation.
\begin{proposition} \label{SREP}
Let $u^\epsilon_0:\Omega \rightarrow [0,1]$. Then 
$u^\epsilon(x,t) = \PP^\epsilon_x[\mathbb{V}_{u_0}^{\gamma}(\W(t))=1]$ 
is a solution to the problem $(AC_\epsilon)$ with initial 
condition $u^\epsilon_0$.
\end{proposition}

\begin{proof}
The proof follows a standard pattern (c.f.~\cite{skorohod:1964, McKean:1975}),
%s seen in \cite{McKean:1975}, \cite{demasi/ferrari/lebowitz:1986}, \cite{etheridge/freeman/penington:2017}, 
which we now sketch. 
To ease notation, we write
\[ \PP_x[\VV(t)=1] := \PP^\epsilon_x[\VV_{u^\epsilon_0}^{\gamma}(\W(t))=1].\]
First we check that $u^\epsilon$ satisfies~$(AC_\epsilon)$ in the interior of $\Omega$.
For this, define the function $g$ by
\begin{multline*} g(p_1,p_2,p_3) = p_1 p_2p_3 + 
p_1 p_2 (1-p_3)+p_1(1-p_2)p_3+(1-p_2)p_2p_3 \\
 + \frac{2 \gamma_\epsilon}{3+3\gamma_\epsilon} 
\Big[ (1-p_1)(1-p_2)p_3+(1-p_1)p_2(1-p_3)+p_1(1-p_2)(1-p_3) \Big].
\end{multline*}
This is the vote of a branching point given that the three 
offspring vote $1$ with probabilities $p_1,p_2,p_3$ respectively. We abuse 
notation to write
%\begin{equation}
%\label{ident1}
$g(p) := g(p,p,p)$.
% = p^3 + 3p^2(1-p) + 
%\frac{2\gamma_\epsilon}{1+\gamma\epsilon}(p(1-p)^2).
%\end{equation}
Note that we have the identity:
\begin{equation} \label{ident1} 
g (p)-p = \frac{1}{1+\gamma_\epsilon}p(1-p)(2p-(1-\gamma_\epsilon)), 
\end{equation}
in which we recognise the non-linearity in $(AC_\epsilon)$ 
divided by $(1+\gamma_\epsilon)$ ({\em c.f.}~the discussion 
below Definition~\ref{FVSdefn}).

Fix $t>0$ and $x \in \Omega$. Denoting by $S$ the first branching 
time of $\W$ and partitioning on
the events $\{S\leq h\}$, $\{S>h\}$, for small $h$ we obtain
\begin{align*}
\PP_x[\VV(t+h)=1] &= \PP_x\big[\VV(t+h)=1| S \leq h\big] \PP[S\leq h] + 
\PP_x\big[\VV(t+h)=1|S>h\big]\PP[S>h] \\ 
&= \EE_x\big[g (\PP_{W_S}[\VV(t)=1])|S \leq h\big] (1-e^{-Vh}) + 
\EE_x\big[\PP_{W_h}[\VV(t)=1]\big] e^{-Vh},
\end{align*}
where we have used the notation $V:=\epsilon^{-2}(1+\gamma_\epsilon)$ 
from Assumptions~\ref{assumptions on W}.
Now, given the regularity of the heat semigroup, and continuity of $g$,
if $h$ is small, we have that
\[ \EE_x[g (\PP_{W_S}[\VV(t)=1])|S \leq h] 
= g (\PP_{x}[\VV(t)=1])+\mathcal{O}(h). \] 
Using this we can compute:
\begin{align*}
\partial_t \PP_x[\VV(t)=1] &= 
\lim_{h\to 0}\frac{\EE_x[\PP_{W_h}[\VV(t)=1] - 
\PP_x[\VV(t)=1]]}{h}e^{-Vh}\\ & 
+
\lim_{h \rightarrow 0}
\frac{1-e^{-Vh}}{h}  \Bigg\{g (\PP_{x}[\VV(t)=1]) 
- \PP_x[\VV(t)=1]\Bigg\}\\ 
&= 
\Delta  (\PP_{\cdot}[\VV(t)=1])(x) + 
V\Big\{g(\PP_x[\VV(t)=1]) - \PP_x[\VV(t)=1]\Big\} 
%\\ 
%&= \Delta(\PP_{\cdot}[\VV(t)=1])(x) + 
%\frac{(1+\gamma_\epsilon)}{\epsilon^2}
%\left(g(\PP_x[\VV(t)=1])- \PP_x[\VV(t)=1]\right),
\end{align*}
and, substituting for $V$ and using identity~(\ref{ident1}), 
the equation follows. 

The boundary condition is inherited from the reflecting Brownian motion
in the Lipschitz domain $\Omega$, see \cite{bass/hsu:1991}.
\end{proof}

\subsection{Basic results} \label{sec:kr}

In this section we record some easy results from~\cite{gooding:2018},
adapted to our setting. 
It will be convenient to be able to refer to these results
in the proof 
of~Theorem~\ref{teo:simplifyversion}.

\subsubsection*{A one-dimensional travelling wave.} 
%\label{sec:onedim}

We will later approximate the profile of
the solution to $(AC_\epsilon)$ in a neighbourhood
of the region in which it takes the value $\frac{1-\nu\varepsilon}{2}$ 
by the one-dimensional function
\begin{equation}
\label{one dimensional wave}
p(x,t) = \Big(\exp\left(-\frac{(x+\nu t)}{\epsilon}\right)+1\Big)^{-1}. 
\end{equation}
Note ({\em c.f.}~Equation~(\ref{travelling wave})) 
that this is a travelling wave solution, with speed $-\nu$ and connecting $0$
at $-\infty$ to $1$ at $\infty$, of the one-dimensional equation
\[ \partial_t u = \Delta u + \frac{1}{\epsilon^2}u (1-u)(2u-(1-\nu \epsilon)). \]
We will need to control the `width' of the one dimensional 
wavefront for small $\epsilon$. 
This is readily obtained from the explicit form 
of~(\ref{one dimensional wave}).
\begin{lemma}[\cite{gooding:2018}, special case of Theorem~2.11]
\label{1d:decay}
For all $k \in \mathbb{N}$, and sufficiently small $\epsilon=\epsilon(k)>0$, 
\begin{enumerate}
    \item for $z \geq k \epsilon |\log(\epsilon)| - \nu t $, we have 
$p(z,t) \geq 1-\epsilon^k$;
    \item for $z \leq -k \epsilon |\log(\epsilon)| - \nu t$, we have 
$p(z,t) \leq \epsilon^k$.
\end{enumerate}
\end{lemma}
In our case, the proof of this result is a simple calculation.
We also need control of the slope of $p(z,t)$. 
\begin{proposition}[\cite{gooding:2018}, special case of Proposition~2.12]
\label{prop:slope}
Let $\epsilon$ be such that
\[ \epsilon < \min\left(\frac{1}{2\nu}, \exp\left( -\frac{36}{23}\right) \right).\]
Suppose that, for some $t \in (0,\infty),z \in \mathbb{R}$ we have
\begin{equation}
\label{5+gamma} 
%\left|\mathbb{P}_z^\epsilon\big[\mathbb{V}_p^{\gamma}(\mathbf{B}(t))=1\big] 
%- \frac{1}{2} \right| 
\left|p(z,t)-\frac{1}{2}\right|
\leq \frac{5+\gamma_\epsilon}{12}, 
\end{equation}
and let $w \in \mathbb{R}$ satisfy $|z-w| \leq \epsilon$. Then
\[ 
%\Big|\mathbb{P}_z^\epsilon\big[\mathbb{V}_p^\gamma(\mathbf{B}(t))=1\big] - 
%\mathbb{P}_w^\epsilon\big[\mathbb{V}_p^\gamma(\mathbf{B}(t))=1\big]\Big| 
\left|p(z,t)-p(w,t)\right|
\geq \frac{|z-w|}{48 \epsilon |\log(\epsilon)|}. \]
\end{proposition}
\begin{proof}[Sketch of proof]
Because we have an exact expression for $p$, it is easy to check that 
the condition~(\ref{5+gamma}) is equivalent to
$$|z-\nu t|\leq 
\epsilon\log\left(\frac{11+\gamma_\epsilon}{1-\gamma_\epsilon}\right).
$$
Also, for $|x-\nu t|\leq A\epsilon$,
$$|\partial_x p(x,t)|\geq \frac{\tfrac{1}{\epsilon}\exp(A)}{(1+\exp(A))^2}.$$
Substituting $A=1+\log((11+\gamma_\epsilon)/(1-\gamma_\epsilon))$, we see 
that for any point $y$ between $z$ and $w$
\begin{equation}
\label{bound on partial}
|\partial_yp(y,t)|\geq 
\frac{1}{48\epsilon}\frac{(11+\gamma_\epsilon)(1-\gamma_\epsilon)}{9},
\end{equation}
where we have used the bound $e<3$. Using the bound on $\epsilon$ in the 
statement of the proposition, we can bound the right hand side 
of~(\ref{bound on partial}) below by $1/(48\epsilon|\log\epsilon|)$.
Now apply the Mean Value Theorem to $p(z,t)-p(w,t)$ to complete 
the proof.
\end{proof}
Denoting by $\mathbf{B}(t)$ the historical process 
of a one-dimensional ternary branching Brownian motion with branching
rate $V=(1+\gamma_\epsilon)/\epsilon^2$ and Brownian motions run at rate $2$, 
the argument in the proof of Proposition~\ref{SREP} yields
that, for $p(x,t)$ given by~(\ref{one dimensional wave}),
\begin{equation}\label{eq:solvot} 
\PP_x^{\epsilon}[\mathbb{V}^\gamma_p(\mathbf{B}(t)) = 1] 
:= \PP_x^{\epsilon}[\mathbb{V}^\gamma_{p(\cdot,0)}(\mathbf{B}(t)) = 1]  
= p(x,t). \end{equation}
The conclusion of Proposition~\ref{prop:slope}
then becomes that for $z$ with
$$
\left|\mathbb{P}_z^\epsilon\big[\mathbb{V}_p^{\gamma}(\mathbf{B}(t))=1\big] 
- \frac{1}{2} \right| 
\leq \frac{5+\gamma_\epsilon}{12},$$ 
and 
$w \in \mathbb{R}$ satisfying $|z-w| \leq \epsilon$,
$$\Big|\mathbb{P}_z^\epsilon\big[\mathbb{V}_p^\gamma(\mathbf{B}(t))=1\big] - 
\mathbb{P}_w^\epsilon\big[\mathbb{V}_p^\gamma(\mathbf{B}(t))=1\big]\Big| 
\geq \frac{|z-w|}{48 \epsilon |\log(\epsilon)|}.$$
 
\noindent
\textbf{In what follows we will always use $\mathbf{B}(t)$ to denote 
the historical paths of a one dimensional 
ternary branching Brownian motion, whereas
$\W(t)$ will denote the historical paths of the 
multidimensional (reflected) ternary branching Brownian motion.} 

~
\subsubsection*{Results on the voting system}% \label{section:results}

From a probabilistic perspective, the effect of the potential 
in~$(AC_\epsilon)$ is captured by the voting mechanism on our tree:
the function $g(p)$ amplifies the difference between $p$ and the
unstable fixed point $(1-\gamma_\epsilon)/2$. As $\epsilon$ decreases,
we see more and more rounds of voting in the tree by time $t$, leading to 
a rapid transition in the solution $u^\epsilon$ to~$(AC_\epsilon)$ from
values close to $0$ to values close to $1$. We need to control the
amplification of $|p-(1-\gamma_\epsilon)/2|$ arising from multiple 
rounds of voting.

We assemble the results that we need related to our voting scheme. 
It is immediate from equation~(\ref{ident1}) that if 
$p<(1-\gamma_\epsilon)/2$, then $g(p)<p$, and, iterating, $g^{(n)}(p)$ will
be a decreasing sequence in $n$. Our first lemma (whose proof mimics
that of Lemma~2.9 in~\cite{etheridge/freeman/penington:2017}) 
controls how rapidly it
converges to zero as $n$ increases.

\begin{lemma}[\cite{gooding:2018}, Lemma~2.14] %\label{lem:g_iteration}
\label{lemma:biasextg}
For all $k\in\mathbb{N}$ there exists $A(k)<\infty$ such that the following hold:
\begin{enumerate}
\item for all $\epsilon \in (0,\tfrac{1-\gamma_\epsilon}{2}]$ and $n\geq A(k)|\log \epsilon |$ we have
$$g^{(n)}(\tfrac{1+\gamma_\epsilon}{2}+\epsilon)\geq 1-\epsilon^k.$$
\item for all $\epsilon \in (0,\tfrac{1+\gamma_\epsilon}{2}]$ and $n\geq A(k)|\log \epsilon |$ we have
$$g^{(n)}(\tfrac{1+\gamma_\epsilon}{2}-\epsilon)\leq \epsilon^k.$$
\end{enumerate}
\end{lemma}
For ease of reference, the proof can be found in 
Appendix~\ref{Proof of bias amplification Lemma}.

%\begin{lemma}[\cite{gooding:2018}, Lemma~2.14(2)]
%\label{lemma:biasextg}
%For all $k \in \mathbb{N}$ there exists $A(k)< \infty$ such that, 
%for all $\epsilon \in (0,\frac{1-\gamma_\epsilon}{2}]$ and 
%$n \geq A(k)\epsilon|\log(\epsilon)|$ we have
%\[ g^{(n)}\left(\frac{1-\gamma_\epsilon}{2}- \epsilon\right) 
%\leq  \epsilon^k. \]
%\end{lemma}
\noindent
It is convenient to record the following easy bound on the function $g$.
%\begin{lemma}[\cite{gooding:2018}, Lemma~2.15] 
For $0 \leq p_1,p_2,p_3 \leq \frac{1-\gamma_\epsilon}{2}$,
\begin{equation}
\label{boundvote} 
 g(p_1,p_2,p_3) \leq \max(p_1,p_2,p_3). 
\end{equation}
%\end{lemma}
\noindent
In order to exploit Lemma~\ref{lemma:biasextg}, we will need to know how small
$\epsilon$ must be for us to be able to find (with high probability) a
regular $n$-generation ternary tree inside $\v{W}(t)$. 
Let $\T^{reg}_n = \bigcup_{k \leq n} \{1,2,3\}^k \subseteq \mathcal{U}$ 
denote the $n$-level regular ternary tree, and for $l \in \mathbb{R}$ 
set $\T^{reg}_l := \T^{reg}_{\lceil l \rceil}$. 
For $\T$ a time-labelled ternary tree, we write $\T^{reg}_l \subseteq \T$ to 
mean that as subtrees of $\mathcal{U}$, $\T^{reg}_l$ is contained inside 
$\T$ (ignoring its time labels). 

The proof of the following result is a simple
modification of that of the corresponding result
(Lemma~2.10) in~\cite{etheridge/freeman/penington:2017}.
\begin{lemma}[\cite{gooding:2018}, Lemma~2.16] 
\label{lemma:bigthree}
Let $k \in \mathbb{N}$ and $A(k)$ as in Lemma~\ref{lemma:biasextg}. 
Then there exist $a(k)$ and $\widehat{\epsilon}(k)$ such that 
for all $\epsilon \in (0,\widehat{\epsilon}(k))$ and 
$t \geq a(k) \epsilon^2 | \log(\epsilon)|$ we have
\[ \PP^\epsilon\big[\mathcal{T}_{A(k) |\log(\epsilon)|}^{reg} \subseteq  
\mathcal{T}(\W(t)) \big] \geq 1 - \epsilon^k. \]
\end{lemma}
\begin{proof}[Sketch of proof]
The idea is simple. The time-length of the path to any leaf in the regular
tree of height $A(k)|\log\epsilon|$ is the sum of this number of 
independent exponentially distributed
random variables with parameter $(1+\gamma_\epsilon)/\epsilon^2$. Use a
large deviation principle for the sum of these exponentials to estimate
the probability that a particular leaf in the regular tree has {\em not}
been born
before time $t$ (and therefore is not contained in $\mathcal{T}(\W(t))$), 
and then use a union bound to control the probability that there is a leaf
of the regular tree not contained in $\mathcal{T}(\W(t))$.
\end{proof}

\subsubsection*{Monotonicity in the initial condition} %\label{sec:comp}

The following comparison result will simplify our analysis of 
the solution of $(AC_\epsilon)$. 
\begin{proposition} \label{prop:cp}
Let $u_0,v_0: \Omega \rightarrow [0,1]$ with $u_0 (x)\leq v_0(x)$ for all 
$x\in\Omega$. Then
\[ \mathbb{P}_x^\epsilon[\mathbb{V}_{u_0}^{\gamma}(\W(t))=1] 
\leq \mathbb{P}_x^\epsilon[\mathbb{V}_{v_0}^{\gamma}(\W(t))=1] 
\qquad \forall x \in \Omega. \]
\end{proposition}
\begin{proof}
An analytic proof would use the maximum principle and monotonicity, but here
we present a probabilistic proof based on a simple coupling argument. 
Since there are so many different sources of randomness in our stochastic
representation of solutions, we spell the argument out. 
Let us write $u(x,t)  :=  
\mathbb{P}_x^\epsilon[\mathbb{V}_{u_0}^{\gamma}(\mathbf{W}(t))=1]$ 
and $v(x,t):= 
\mathbb{P}_x^\epsilon[\mathbb{V}_{v_0}^{\gamma}(\mathbf{W}(t))=1]$. 
We shall say `the probability law defining $u$' (respectively $v$) to
mean the law in which the each of the leaves of $\W(t)$ 
votes $1$ with probability $u_0$ (respectively $v_0$).

Consider a given realisation of $\W(t)$. 
To determine the vote of a leaf $W_i(t)$, we sample $U_i$ 
uniformly on $[0,1]$. If $U_i\leq u_0(W_i(t))$ then the vote of
$W_i(t)$ is $1$ under the law defining $u$. Similarly, the vote of
$W_i(t)$ under the law defining $v$ is one if and only if 
$U_i\leq v_0(W_i(t))$. Notice that since, 
by assumption, $u_0(x)\leq v_0(x)$ for all $x\in\Omega$, coupling 
by using the same uniform random variables $U_i$ for determining
the leaf votes for $u$ and $v$, all votes that are $1$ under the law defining $u$ are 
also $1$ under the law defining $v$.
Now consider the votes on the interior of the tree. Suppose that the
votes of the offspring of a branch are 
$(i_1,i_2,i_3)$ under $u$ and $(j_1,j_2,j_3)$ under $v$, and 
$i_1+i_2+i_3 \leq j_1 +j_2+j_3$.
Recall that under our voting procedure, the vote is determined by 
majority voting unless exactly one vote is one. Evidently, if the 
majority vote under $u$ is $1$, then it is also $1$ under $v$; and if 
$j_1+j_2+j_3=0$, then the vote will be zero under both $u$ and $v$.
The only case of interest is when $i_1+i_2+i_3=1=j_1+j_2+j_3$. 
The vote is then determined by a Bernoulli random variable, resulting 
in a $1$ with probability $2\gamma_\epsilon/(3+3\gamma_\epsilon)$. For
the final stage of the coupling, we
use the same Bernoulli random variable for determining the vote
under $u$ and under $v$. This coupling guarantees that if the vote under 
$u$ is one, then necessarily the vote under $v$ must also be one. 
Proceeding inductively down the tree $\W$, which is a.s.~finite,
we find that 
if under $u$ the vote of the root is $1$ then under $v$ it must also be $1$, 
and the result follows.
\end{proof}

\subsection{The upper solution} \label{sec:ups}

Proposition~\ref{prop:cp} allows us to bound the solution $u$ to
$(AC_\epsilon)$ on the domain $\Omega$ of Figure~\ref{fig:omega}
by that obtained by starting from a larger initial
condition. It is convenient to start from an initial condition that is
radially symmetric on the left side of our domain (i.e.~when $x_1<0$). We set
\[ N_{\mathbbm{r}} = \{ x \in \Omega \, :\,  \| x \| = \mathbbm{r},\; x_1 <0 \}, \]
where $R_0> \mathbbm{r} > r_0$ (see Figure~\ref{fig:omegaandNr}).
We can then write $\Omega$ as the disjoint 
union
\[ \Omega = \Omega_+ \sqcup \Omega_- \sqcup N_\mathbbm{r},\]
where
\[ \Omega_+ = \{ x \in \Omega \, :\, \|x\| <\mathbbm{r}\mbox{ or } x_1 >0 \}, \]
and
\[ \Omega_- = \{ x \in \Omega \, :\, \|x \| >\mathbbm{r} \mbox{ and } x_1 < 0\}.\]
The signed distance $\widehat{d}$ of any point in $\Omega$ to 
$N_\mathbbm{r}$ is defined up to a change of sign and we 
impose $\widehat{d}(x) >0$ for all $x$ in $\Omega_+$.  
Here the distance is that inherited from $\IR^\dim$.
Note that %by the definition of our set $N_\mathbbm{r}$ and the domain $\Omega$ we have that, 
for all $x \in \Omega$ such that $x_1 <0$,
\begin{equation} \label{prop:curv}
 \widehat{d}(x) =  \mathbbm{r} - \|x\|.
\end{equation}
Let the function $\widehat{p}:=\widehat{p}(x)$ satisfy the following conditions:% $(IC)$, described as follows
\begin{enumerate}[leftmargin=1.1cm]
\item[$(\mathscr{I}1)$] $\widehat{p}(x) = 1$ for all 
$x=(x_1,...,x_{\mathbbm{d}}) \in \Omega$ such that $x_1 \geq 0$;
\item[$(\mathscr{I}2)$] $\widehat{p}(x) = \frac{1-\gamma_\epsilon}{2}$ 
for all $x \in N_\mathbbm{r}$;
\item[$(\mathscr{I}3)$] $\widehat{p}(x) > \frac{1-\gamma_\epsilon}{2}$ 
if $\widehat{d}(x) >0$, 
and $\widehat{p}<\frac{1-\gamma_\epsilon}{2}$ if $\widehat{d}(x)<0$;
\item[$(\mathscr{I}4)$] $\widehat{p}(x)$ is continuous and there exists $\mu,\lambda>0$ 
such that 
\mbox{$|\widehat{p}(x)-\frac{1-\gamma_\epsilon}{2}| 
\geq \mu\big(\text{dist}(x,N_\mathbbm{r})
\wedge \eta\big)$}.
\end{enumerate}
The first condition means that $\widehat{p}$ dominates the initial condition
for~$(AC_\epsilon)$; the second and third that 
$\widehat{p}(1-\widehat{p})\big(2\widehat{p}-(1-\gamma_\epsilon)\big)$ 
vanishes only on $N_{\mathbbm{r}}$; and the final one 
is a lower bound on the slope of $\widehat{p}$ in a neighbourhood of 
$N_{\mathbbm{r}}$.
Since $\mathbbm{r} > r_0$, one can readily construct a 
function $\widehat{p}$ that is radially symmetric in $\Omega\cap\{x_1<0\}$ and
satisfies $(\mathscr{I}1)-(\mathscr{I}4)$. (The first, second and 
fourth condition are incompatible when $\mathbbm{r} < r_0$.)

We are going to show that for 
$\mathbbm{r}<\frac{\dim -1}{\nu}$, and small enough $\epsilon$, the solution to
\[ (AC^*_\varepsilon) = 
\begin{cases} 
\partial_t u = \Delta u + \frac{1}{\epsilon^2} u(1-u)(2u-(1-\epsilon \nu)), \\ 
u_0 = \widehat{p}(x), \\ \partial_n u = 0. 
\end{cases} 
\]
is blocked. More precisely:
\begin{theorem} 
\label{teo:decayv2}
Suppose that $r_0<\mathbbm{r}<R_0\wedge \frac{\dim -1}{\nu}$ and that
$\widehat{p}(x)$ satisfies $(\mathscr{I}1)$-$(\mathscr{I}4)$.
Fix $k \in \mathbb{N}$.  Then there exist
$\widehat{\epsilon}(k)>0$ and $c(k)>0$ such that for all 
$\epsilon \in (0,\widehat{\epsilon})$, 
$t \in (0,\infty)$,
\[ 
\mbox{for }z\in\Omega\mbox{ such that } 
d(z) \leq -c(k)\epsilon |\log(\epsilon)| , \mbox{ we have }
u(z,t) \leq \epsilon^k, 
\]
where $u(z,t)$ solves~$(AC^*_\varepsilon)$.
\end{theorem}
Since $\widehat{p}(x)\geq \1_{x_1\geq 0}$, Theorem~\ref{teo:simplifyversion}
follows from Theorem~\ref{teo:decayv2} and Proposition~\ref{prop:cp}.
To prove Theorem \ref{teo:decayv2} we will control $u(x,t)$ 
close to $N_\mathbbm{r}$. The key to this 
will be a coupling argument which will allow us to control the 
change in the solution over small time intervals in terms of a 
one-dimensional problem. 
\begin{figure*}[t!] 
    \centering
        \includegraphics[scale=1]{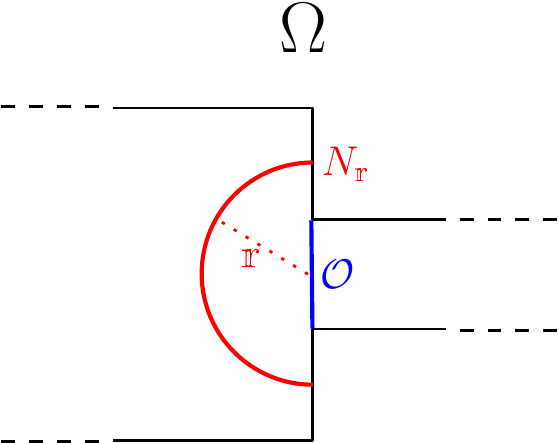}
    \caption{The domain $\Omega$ with $N_\mathbbm{r}$ and $\mathcal{O}$ represented in it.}
    \label{fig:omegaandNr}
\end{figure*} 

\subsection{Coupling around a hemispherical shell} \label{sec:co}

Our next goal is to quantify the effect 
of the mean curvature flow in `pushing back' the solution of $(AC_\epsilon)$.
We will need a notation for the `opening' in $\Omega$, which we will denote
${\mathcal O}$ (see Figure~\ref{fig:omegaandNr}). 
Recalling that we write $x\in\IR^\dim$ as $(x_1, x')$
with $x'\in\IR^{\dim -1}$,  
\[ {\mathcal O} = \{ x \in \Omega : x_1 = 0, 
%\sum_{i=2}^\dim |x_i|^2 <r_0^{\dim -1}
\|x'\|<r_0\}.\]
%, \forall 2 \leq i \leq \mathbbm{d} \}. \]
%We again refer to Figure~\ref{fig:omegaandNr} for a representation of $\mathcal{O}$ in the domain. 
The following elementary result will play the role of coupling results
in~\cite{etheridge/freeman/penington:2017} (Proposition~2.14) 
and~\cite{gooding:2018} (Proposition~2.13), but is completely
straightforward since
$N_{\mathbbm r}$ takes such a simple form.
\begin{lemma} 
\label{teo:coup}
Suppose that $r_0 < \mathbbm{r} < R_0$. Let $W$ be a Brownian motion 
in $\Omega$ and $\beta >0$ such that
\[ \beta \leq \frac{\mathbbm{r}-r_0}{2} \wedge \frac{R_0-\mathbbm{r}}{2}. \]
Define $T_\beta$ by
\[ T_\beta = \inf \{t \geq 0 : |\widehat{d}(W_t)| \geq \beta \}. \]
Then there exists a one-dimensional Brownian motion 
$\widehat{B}$, started from $\widehat{d}(W_0)$ such that, 
for $0\leq s \leq T_\beta$, 
\[ \widehat{B}_s - \frac{s(\mathbbm{d}-1)}{\mathbbm{r}+\beta} \geq \widehat{d}(W_s) \geq \widehat{B}_s - \frac{s(\mathbbm{d}-1)}{\mathbbm{r}-\beta}. \]
\end{lemma}
\begin{proof}
Note that for $s \leq T_\beta$, the first coordinate of $W_s$, satisfies
$(W_1)_s < 0$. Indeed, to reach the region in which $(W_1)_s\geq 0$, 
we must pass through ${\mathcal O}$, which requires 
$|\widehat{d}(W_s)| \geq \mathbbm{r}-r_0 >\beta$. 
In particular, by (\ref{prop:curv}), we have that 
$\widehat{d}(W_s) = \mathbbm{r} - \|W_s\|$ for all $s \leq T_\beta$.
Moreover, for $0\leq s<T_\beta$, $\|W_s\|$ is a (time changed since our 
Brownian motion runs at rate $2$) Bessel 
process, whose drift (by definition of $T_\beta$) lies
in $(\frac{\dim -1}{\mathbbm{r}+\beta}, \frac{\dim -1}{\mathbbm{r}-\beta}$), 
and the result follows.
\end{proof}

\subsection{Generation of an interface} 
\label{sec:genint}

The next step is to check that the solution to the equation with 
initial condition $\widehat{p}$ decays rapidly for small
times. The proof mimics those that guarantee the formation of an 
interface
in~\cite{etheridge/freeman/penington:2017} (Proposition~2.16)
and~\cite{gooding:2018} (Lemma~2.17).
The key is to control the maximal displacement of 
particles of the ternary branching reflected Brownian motion, with
the only twist now being to handle the fact that particles are
reflected off the boundary of the domain. 
%Things are simplified by
%the fact that we only need this result for the process started from
%an individual at $x$ with $\widehat{d}(x)<0$.
\begin{proposition} 
\label{prop:maxmov}
Let $k \in \mathbb{N}$ and $a(k)$ be given by Lemma~\ref{lemma:bigthree}. Then 
there exists $e(k)$ and $\widehat{\epsilon}(k)$ such that, for 
all $\epsilon \in (0,\widehat{\epsilon}(k))$, all 
$s \leq (a(k)+k+1)\epsilon^{2} |\log(\epsilon)|$, and all $x \in \Omega$,
\[ \PP_x^\epsilon\Big[ \exists i \in N(s): \|W_i(s) -x\| 
\geq e(k) \epsilon |\log(\epsilon)|\Big] \leq \epsilon^k. \]
\end{proposition}
\begin{proof}
Note that
\begin{align} 
\label{matrak:pm1}
\PP_x^\epsilon\Big[ \exists i \in N(s): &\|W_i(s) -x\| \geq e(k) 
\epsilon |\log(\epsilon)|\Big] \nonumber \\ 
& \leq \mathbb{E}^\epsilon\big[N(s)\big] 
\PP_x^\epsilon\Big[\|W(s)-x\| \geq e(k) \epsilon |\log(\epsilon)|\Big] 
\nonumber \\ 
& = e^{2 \frac{s(1+\epsilon \nu)}{\epsilon^2}} 
\PP_x^\epsilon\Big[\|W(s)-x\| \geq e(k) \epsilon |\log(\epsilon)|\Big] 
\nonumber \\ 
& \leq e^{2(a(k)+k+1)|\log(\epsilon)|(1+\epsilon \nu)} 
\PP_x^\epsilon\Big[\|W(s)-x\| \geq e(k) \epsilon |\log(\epsilon)|\Big] 
\nonumber\\
& \leq e^{2(a(k)+k+1)|\log(\epsilon)|(1+\epsilon \nu)}
\nonumber \\
& \times 
\PP_x^\epsilon\Big[\sup_{0\leq s\leq (a(k)+k+1)\epsilon^2|\log(\epsilon)|}
\|W(s)-x\| \geq e(k) \epsilon |\log(\epsilon)|\Big] .
\end{align}
To bound the last term, consider first $(\widehat{W}_t)_{t \geq 0}$,
defined to be a Brownian motion reflected on the boundary of 
$\Omega \cap B(x,1)$, which we take to coincide with 
$(W_t)_{t \geq 0}$ until the first hitting time of the boundary of
$\Omega \cap B(x,1)$.
%by using the same underlying Brownian motion in $\mathbb{R}^\mathbbm{d}$. 
Since $\Omega \cap B(x,1)$ is a bounded domain with Lipschitz boundary, 
Theorem~2.7 of~\cite{bass/hsu:1991} implies the existence of
constants $c_1,c_2>0$ (independent of $x$) such that,
\begin{align}
\PP_x^\epsilon\Big[\sup_{0\leq s\leq (a(k)+k+1)\epsilon^2|\log(\epsilon)|}
\|\widehat{W}(s)-x\| &\geq e(k) \epsilon |\log(\epsilon)|\Big] \leq c_1 \exp\left( \frac{-e(k)^2 |\log(\epsilon)|}{c_2(a(k)+k+1)} \right) \nonumber \\ &= \epsilon^{e(k)^2(c_2(a(k)+k+1))^{-1}}. \label{standard gaussian}
\end{align}
%The proof is complete on choosing 
Choose $e(k)$ sufficiently large that 
\begin{equation}
\label{choice of e}
-2\big(a(k)+k+1\big)(1+\epsilon\nu)+\frac{e(k)^2}{c_2\big(a(k)+k+1\big)}
>k.
\end{equation}
%%%%%%%%%%%%%%%%%%%%%%%%%%%%%%%%%%%%%%%%%
%the bound is less than $\epsilon^k$. 
Finally, choose $\widehat{\epsilon}(k)$ sufficiently small that 
$e(k)\epsilon |\log(\epsilon)| < 1$ for all $\epsilon < \widehat{\epsilon}$. 
Since $\widehat{W}(s) = W(s)$ until $\widehat{W}(s)$ hits the 
boundary of $B(x,1)$, this choice of $\widehat{\epsilon}$ ensures that
\begin{multline*}
\left\{ \sup_{0\leq s\leq (a(k)+k+1)\epsilon^2|\log(\epsilon)|}\|
\widehat{W}(s)-x\| \geq e(k) \epsilon |\log(\epsilon)| \right\} \\ 
= \left\{ \sup_{0\leq s\leq (a(k)+k+1)\epsilon^2|\log(\epsilon)|}
\|W(s)-x\| \geq e(k) \epsilon |\log(\epsilon)| \right\} . 
\end{multline*}
Substituting~(\ref{standard gaussian}) in~(\ref{matrak:pm1})
and using~(\ref{choice of e}) yields the desired bound.
\end{proof}
We are now in a position to prove the rapid decay of the solution 
to~$(AC^*_\varepsilon)$
to the left of $N_{\mathbbm{r}}$. 
\begin{proposition} 
\label{prop:interface}
Let $k \in \mathbb{N}$. There exist $\widehat{\epsilon}(k)$, $a(k)$, $b(k)>0$, 
such that, for all $\epsilon \in (0,\widehat{\epsilon}(k))$, setting
\[ \delta(k,\epsilon) := a(k) \epsilon^2 |\log(\epsilon)|\mbox{ and } 
\delta'(k,\epsilon) := (a(k)+k+1)\epsilon^2 |\log(\epsilon)|, \]
for $t \in [\delta (k,\epsilon),\delta'(k,\epsilon)]$, and $x$ with 
$\widehat{d}(x) \leq - b(k) \epsilon |\log(\epsilon)|$, 
$\PP_x^\epsilon[ \mathbb{V}_{\widehat{p}}^\gamma(\W(t))=1 ] 
\leq \epsilon^k$.
\end{proposition}
\begin{proof}
By Lemma~\ref{lemma:bigthree}, given $k \in \mathbb{N}$ there exists 
$a(k)$ and $\widehat{\epsilon}(k)$ such that, 
for all $\epsilon \in (0,\widehat{\epsilon}(k))$ 
and $t \geq a(k)\epsilon^2 |\log(\epsilon)|$,
\begin{equation}
\label{reg tree prob} 
\PP^\epsilon\big[\mathcal{T}_{A(k)|\log(\epsilon)|}^{reg} 
\subseteq \mathcal{T}(\W(t))\big] \geq 1 - \epsilon^k. 
\end{equation}
Similarly by Proposition~\ref{prop:maxmov} there exists $e(k)$ such 
that, taking a smaller $\widehat{\epsilon}(k)$ if necessary, 
for $\epsilon\in(0,\widehat{\epsilon}(k))$ and
$t \in [\delta (k,\epsilon),\delta'(k,\epsilon)]$, we have,
\begin{equation}
\label{not gone far prob}
 \PP_x^\epsilon\Big[\exists i \in N(t): \|W_i(t)-x\| \geq e(k) \epsilon |
\log(\epsilon)|\Big] \leq \epsilon^k. 
\end{equation}
Set $b(k) = 2 e(k)$. Then if $x$ is such that 
$\widehat{d}(x) \leq - b(k) \epsilon | \log(\epsilon)|$ and 
$\|W_i(t) - x\| \leq e(k) \epsilon | \log(\epsilon)|$ 
(noting that this implies that $W_1(t) \leq 0$), we obtain
(using~(\ref{prop:curv}) and the triangle inequality)
\begin{align*}
\widehat{d}(W_i(t)) = \widehat{d}(x)+\|x\|-\|W_i\|
& %\leq \widehat{d}(x) + |d(W_i(t))-d(x)| 
\leq \widehat{d}(x) + \|W_i(t)  -x \| \\ 
& \leq - b(k) \epsilon |\log(\epsilon)| + e(k) \epsilon | \log(\epsilon)| 
= -e(k) \epsilon | \log(\epsilon)|,
\end{align*}
Reducing $\widehat{\epsilon}(k)$ if necessary, 
we can ensure that $\epsilon < \mu \eta$ with $\mu$, $\eta$ as in
$(\mathscr{I}4)$, and 
$\epsilon < \mu e(k) \epsilon | \log(\epsilon)|$ for 
$\epsilon \in (0,\widehat{\epsilon}(k))$. 
Then since $\widehat{p}$ satisfies $(\mathscr{I}4)$, we have that
\[ \widehat{p}(W_i(t)) \leq \frac{1-\gamma_\epsilon}{2} - 
\mu\big( |\widehat{d}(W_i(t))| \wedge \eta \big) 
\leq \frac{1-\gamma_\epsilon}{2} - \epsilon. \]
Using~(\ref{reg tree prob})
and~(\ref{not gone far prob}), with probability at least $1-2\epsilon^k$,
$\mathcal{T}(\W(t))$ contains a regular tree with at least
$A(k)|\log(\epsilon)|$ generations, and each leaf of $\mathcal{T}(\W(t))$
votes $1$ 
with probability at most $\frac{1-\gamma_\epsilon}{2}-\epsilon$.
Recalling from Equation~(\ref{boundvote}), that if 
$p_i \leq \frac{1-\gamma_\epsilon}{2}$ for $i=1,2,3$ then 
$g_{\gamma_\epsilon}(p_1,p_2,p_3)\leq \max(p_1,p_2,p_3)$, 
we deduce that in this case, each vertex of $\mathcal{T}$ that
corresponds to a leaf of the regular tree of height $A(k)|\log(\epsilon)|$
votes $1$ with probability at most 
$\frac{1-\gamma_\epsilon}{2}-\epsilon$. Lemma~\ref{lemma:biasextg}
then gives
\[ \PP_x^\epsilon[\mathbb{V}_{\widehat{p}}^\gamma(\W(t))=1] 
\leq g_{\gamma_\epsilon}^{(\lceil A(k) |\log(\epsilon)| \rceil)}
\left(\frac{1-\gamma_\epsilon}{2}-\epsilon\right) + 2 \epsilon^k 
\leq 3 \epsilon^k. \]
\end{proof}

\subsection{Blocking (proof of Theorem~\ref{teo:decayv2})}
\label{sec:blo}

To prove Theorem \ref{teo:decayv2}, we establish
a comparison between the solution to the multidimensional problem and
that of the one-dimensional travelling wavefront~(\ref{one dimensional wave}).
This is the content of Proposition~\ref{prop:ineqonetwo} (corresponding to
Proposition~2.17 in~\cite{etheridge/freeman/penington:2017}, and 
Proposition~2.19 in~\cite{gooding:2018}), whose
proof relies on a bootstrapping argument, the key ingredient
for which is Lemma~\ref{Lemma:pushcomp}. The radial symmetry simplifies some
of the expressions in our setting.

\begin{proposition} 
\label{prop:ineqonetwo}
Let $l \in \mathbb{N}$ with $l \geq 4$ and let $a(l)$ and $\delta(l,\epsilon)$ 
be given by Proposition~\ref{prop:interface}. There exist $K_1(l)$ and 
$\widehat{\epsilon}(l,K_1) > 0$ such that, for all 
$\epsilon \in (0, \widehat{\epsilon}(l, K_1))$, and 
$t \in [\delta(l,\epsilon),\infty)$ 
we have
\[ 
\sup_{x \in \Omega}\left(\PP_x^\epsilon[
\mathbb{V}_{\widehat{p}}^\gamma(\W(t)) = 1)]- 
\PP^\epsilon_{\widehat{d}(x) - \nu t+ K_1 \epsilon |\log(\epsilon)|} 
[\mathbb{V}^\gamma_p(\mathbf{B}(t)) = 1] \right) \leq \epsilon^l, 
\]
with $\mathbb{V}^\gamma_p(\mathbf{B}(t))$ as in~(\ref{eq:solvot}).
\end{proposition}
The key point here is that after the time $\delta$ that it takes for
the interface to develop, we are comparing the multidimensional solution
to the stationary one-dimensional interface.

We extend the domain of the function $g$ from $[0,1]$ to $\mathbb{R}$ 
by setting $g(p)=0$ for $p<0$ and $g(p)=1$ for $p>1$. The following result will be proved after we have proved Proposition~\ref{prop:ineqonetwo}
\begin{lemma} 
\label{Lemma:pushcomp}
Let $l \in \mathbb{N}$ with $l \geq 4$, $K_1>0$. There exists 
$\widehat{\epsilon}(l,K_1) > 0$ such that for all 
$\epsilon \in (0,\widehat{\epsilon}(l, K_1))$, 
$x \in \Omega$, $s \in [0,(l+1)\epsilon^2 |\log(\epsilon)|]$ 
and $t \in [s,\infty)$,
\begin{multline}
\label{bootstrap}
\EE_x\left[g\left(\PP^\epsilon_{\widehat{d}(W_s)-\nu(t-s)+K_1 \epsilon|\log(\epsilon)|}
[\mathbb{V}_p^\gamma(\mathbf{B}(t-s)=1]+\epsilon^l\right)\right] 
 \\ 
\leq \frac{3 + 5 \gamma_\epsilon}{4(1+\gamma_\epsilon)}\epsilon^l + 
\EE_{\widehat{d}(x)}^\epsilon\left[g\left(\PP^\epsilon_{B_s-\nu t+K_1 \epsilon |\log(\epsilon)|}
[\mathbb{V}_p^\gamma(\mathbf{B}(t-s))=1]\right)\right] + 
1_{s \leq \epsilon^3} \epsilon^l.
\end{multline}
\end{lemma}
\begin{remark}
In the proof of Proposition~\ref{prop:ineqonetwo},
the time $s$ in Lemma~\ref{Lemma:pushcomp} will be the time of the 
first branch in our ternary branching process.
In the symmetric case considered in~\cite{etheridge/freeman/penington:2017},
the signed 
distance between a multi-dimensional Brownian motion at time $s$
and an interface
driven by mean curvature flow at time $t-s$
is coupled to a one-dimensional Brownian
motion (for small times). Here, $\widehat{d}$ is the signed 
distance to $N_{\mathbbm{r}}$, which is subject to a mixture of constant flow
and curvature flow and so we consider
$\widehat{d}(W_s)-\nu(t-s)$. 
The inequality~(\ref{bootstrap}) will be used to control one round
of branching in $\v{W}$.
\end{remark}
The proof of Proposition~\ref{prop:ineqonetwo} first compares the 
multi-dimensional solution to the one-dimensional one over a small 
time window. It then bootstraps to longer times by conditioning on the
first branch time in $\v{W}$. The point of Lemma~\ref{Lemma:pushcomp}  
is that at the branch time we have $g$ evaluated on the one-dimensional
solution plus a small error. If the solution is far from the interface 
$(1-\gamma_\epsilon)/2$, then it is easy to see that the vote will not 
be changed by the small error. The difficult case is close to the 
interface, where we make use of estimates of the slope of the 
(one-dimensional) interface from Proposition~\ref{prop:slope}.     
\begin{proof}[Proof of Proposition~\ref{prop:ineqonetwo}] 

We follow the proof of Proposition~2.19 of~\cite{gooding:2018} which, in
turn, adapts that of Proposition~2.17 
of~\cite{etheridge/freeman/penington:2017}.

Take $K_1= b(l)+l$, where $b(l)$ is given by Proposition~\ref{prop:interface}. 
Take $\widehat{\epsilon}$ less than the corresponding $\widehat{\epsilon}(l)$
in each of Lemma~\ref{1d:decay}, 
Proposition~\ref{prop:interface}, and Lemma~\ref{Lemma:pushcomp}. 

Define $\delta$ and $\delta'$ as in Proposition~\ref{prop:interface}. 
We first
show that for $\epsilon \in (0,\widehat{\epsilon})$, $t \in [\delta,\delta']$, 
and $x \in \Omega$ we have
\begin{equation}
\label{eq:caca}
\PP_x^\epsilon[\mathbb{V}^\gamma_{\widehat{p}}(\W(t))=1] 
\leq \PP^\epsilon_{\widehat{d}(x)- \nu t+K_1 \epsilon |\log(\epsilon)|} 
[\mathbb{V}^\gamma_p(\mathbf{B}(t)) = 1] + \epsilon^l. 
\end{equation}
Indeed, if $\widehat{d}(x) \leq -b(l) \epsilon |\log(\epsilon)|$ then, by 
Proposition~\ref{prop:interface}, we have 
$\PP_x^\epsilon[\mathbb{V}_{\widehat{p}}(\W(t))=1] \leq \epsilon^l$.  
On the other hand if $\widehat{d}(x) \geq -b(l) \epsilon |\log(\epsilon)|$ 
then $\widehat{d}(x)-\nu t +K_1 \epsilon |\log(\epsilon)| \geq -\nu t + l 
\epsilon |\log(\epsilon)|$ (since $K_1=b(l)+l$), and so, by 
Lemma~\ref{1d:decay}, 
the right hand side of~(\ref{eq:caca}) is greater than $1$ and hence 
the inequality obviously holds.

Now consider $t\geq\delta'$. Suppose that there exist $t \in [\delta',\infty)$ for which the 
inequality~(\ref{eq:caca})
does not hold, that is for which there is $x \in \Omega$ such that
\begin{equation}
\label{what does not happen} 
\PP_x^\epsilon[\mathbb{V}^\gamma_{\widehat{p}}(\W(t))=1] 
- \mathbb{P}^\epsilon_{\widehat{d}(x)-\nu t+K_1  \epsilon |\log(\epsilon)|}
[\mathbb{V}^\gamma_p(\mathbf{B}(t)) = 1] > \epsilon^l. 
\end{equation}
Let $T'$ be the infimum of such $t$ and choose $T \in [T',T'+\epsilon^{l+3}]$ 
for which~(\ref{what does not happen}) holds. By definition there is 
$x=x(l,\epsilon) \in \Omega$ such that
\[ \PP_x^\epsilon[\mathbb{V}^\gamma_{\widehat{p}}(\W(T))=1] 
- \mathbb{P}^\epsilon_{\widehat{d}(x)-\nu T+K_1 \epsilon |\log(\epsilon)|}
[\mathbb{V}^\gamma_p(\mathbf{B}(T)) = 1] > \epsilon^l. \]
We now seek to show that under this assumption
\begin{equation}
\label{inequality for a contradiction}
\PP_x^{\epsilon}[\mathbb{V}_{\widehat{p}}(\W(T))=1] 
\leq \frac{7+9\gamma_\epsilon}{8(1+\gamma_\epsilon)}\epsilon^l + 
\mathbb{P}^\epsilon_{\widehat{d}(x)-\nu T+
K_1 \epsilon |\log(\epsilon)|}[\mathbb{V}^\gamma_p(\mathbf{B}(T)) = 1]. 
\end{equation}
To establish a contradiction we then choose $\epsilon$
sufficiently small that 
$\frac{7+9\gamma_\epsilon}{8(1+\gamma_\epsilon)} \epsilon^l< \epsilon^l$, 
so that~(\ref{inequality for a contradiction}) 
contradicts the assumption that~(\ref{what does not happen}) is
satisified at time $T$. 

It remains to prove~(\ref{inequality for a contradiction}) under the
assumption that there exist $t\in[\delta',\infty)$ for 
which~(\ref{what does not happen}) holds.
We write $S$ for the 
time of the first branching event in $\W(t)$, and $W_S$ for 
the location of the individual that branches at that time. 
Using the strong Markov property, 
\begin{multline} 
\label{matrak:M1} 
\PP_x^\epsilon[\mathbb{V}_{\widehat{p}}^\gamma(\W(T)) = 1] 
= \mathbb{E}_x^\epsilon\left[g\left(\PP_{W_S}^\epsilon[
\mathbb{V}^\gamma_{\widehat{p}}(\W(T-S)) = 1] 
\right)\1_{S \leq T - \delta}\right] \\ 
+ \mathbb{E}_x^\epsilon[\PP_{W_{T-\delta}}^\epsilon
[\mathbb{V}^\gamma_{\widehat{p}}(\W(\delta))=1]\1_{S \geq T - \delta}].
\end{multline}
As $T-\delta \geq \delta' - \delta = (l+1)\epsilon^2 |\log(\epsilon)|$ 
and $S \sim \mathtt{Exp}((1+\gamma_\epsilon)\epsilon^{-2})$ we can
bound the second term 
in~(\ref{matrak:M1}) by
%\mathbb{E}_x^\epsilon[\PP_{W_{T-\delta}}^\epsilon[
%\mathbb{V}_{\widehat{p}}^\gamma(\W(\delta))=1]1_{S \geq T - \delta}] 
%\leq 
$\PP^\epsilon[S \geq (l+1) \epsilon^2 |\log(\epsilon)|] 
\leq \epsilon^{l+1}$. 
%\]
To bound the other term we partition over the event 
$\{ S \leq \epsilon^{l+3}\}$ (which we note has probability at 
most $(1+\gamma_\epsilon)\epsilon^{l+1}$) and its complement:
\begin{align*}
&\mathbb{E}_x^\epsilon\left[g\left(\PP_{W_S}^\epsilon
[\mathbb{V}^\gamma_{\widehat{p}}(\W(T-S)) = 1]\right) 
\1_{S \leq T - \delta}\right] \\ 
& \leq \mathbb{E}_x^\epsilon\left[g\left(
\PP_{W_S}^\epsilon[\mathbb{V}_{\widehat{p}}^\gamma(\W(T-S)) = 1] 
\right)\1_{S \leq T - \delta} \1_{S \geq \epsilon^{l+3}}\right] 
+ \PP^\epsilon[S \leq \epsilon^{l+3}] \\ 
& \leq \mathbb{E}_x^\epsilon\left[g\left(\PP^\epsilon_{\widehat{d}(W_S)-\nu (T-s)+K_1 
\epsilon|\log(\epsilon)|}[ \mathbb{V}^\gamma_p(\mathbf{B}(T-S))=1 ] 
+ \epsilon^l\right)\1_{S \leq T-\delta}\right] 
+ (1+\gamma_\epsilon)\epsilon^{l+1},
\end{align*}
where the last line follows from the minimality of $T'$ 
(and noting that if $\epsilon^{l+3} \leq S \leq T-\delta$, then 
$T-S \in [\delta,T')$), and the monotonicity of $g$.

As the path of a particle is conditionally independent of the time at which
it branches, we can condition further on $S$ to obtain
\begin{align*}
\mathbb{E}_x^\epsilon& \left[g\left(\PP^\epsilon_{\widehat{d}(W_S)-\nu(T-s)
+K_1 \epsilon|\log(\epsilon)| }[ \mathbb{V}^\gamma_p(\mathbf{B}(T-S))=1] 
+ \epsilon^l\right) \1_{S \leq T-\delta}\right] \\ 
& \leq \int_0^{(l+1)\epsilon^{2} |\log(\epsilon)|} 
(1+\gamma_\epsilon) \epsilon^{-2} e^{-(1+\gamma_\epsilon)\epsilon^{-2}s} 
\\ & \hspace{20mm}\times  
\mathbb{E}_x\left[g\left(\PP^\epsilon_{\widehat{d}(W_s)-\nu(T-s)+
K_1 \epsilon |\log(\epsilon)|}[\mathbb{V}^\gamma_p(\mathbf{B}(T-s))=1] 
+\epsilon^l\right)\right] ds \\ 
& \hspace{25mm}+ \PP^\epsilon[S \geq (l+1) \epsilon^{2} |\log(\epsilon)|]
\\ 
 & \leq \int_0^{(l+1)\epsilon^{2} |\log(\epsilon)|}  (1+\gamma_\epsilon) 
\epsilon^{-2} e^{-\epsilon^{-2}(1+\gamma_\epsilon)s} \\ 
& \hspace{20 mm} \times \mathbb{E}_{\widehat{d}(x)}\left[g 
\left(\PP^\epsilon_{B_s-\nu T +
K_1 \epsilon |\log(\epsilon)|}[\mathbb{V}^\gamma_p(\mathbf{B}(T-s))=1] 
+\epsilon^l\right) \right] ds \\ 
& \hspace{30mm} + \frac{3+5\gamma_\epsilon}{4(1+\gamma_\epsilon)}
\epsilon^l+\PP^\epsilon[S \leq \epsilon^3] \epsilon^l + \epsilon^{l+1}  \\ 
& \leq \mathbb{E}_{\widehat{d}(x)}^\epsilon\left[ 
g \left(\PP^\epsilon_{B_{S'}-\nu T+K_1  \epsilon|\log(\epsilon)|}
[\mathbb{V}^\gamma_p(\mathbf{B}(T-S')=1]\right)
\1_{S' \leq T-\delta} \right] \\
& \hspace{30mm} + 
\frac{3+5\gamma_\epsilon}{4(1+\gamma_\epsilon)} \epsilon^l+ \epsilon^{l+1} 
+ \epsilon^{l+1}(1+\gamma_{\epsilon}),
\end{align*}
where the second inequality follows by Lemma~\ref{Lemma:pushcomp}. 
For the final inequality we have written $S'$ for the time of the first 
branching event in $(\mathbf{B}(s))_{s \geq 0}$ and $B_{S'}$ for the 
position of the ancestor particle at that time.
Of course $S'$ has the same distribution as $S$, and the inequality 
follows since $T \geq \delta'$ and so 
$T-\delta \geq (l+1)\epsilon^2 |\log(\epsilon)|$.

Combining these computations, (\ref{matrak:M1}) becomes
\begin{align*}
\PP_x^\epsilon&[\mathbb{V}_p^\gamma(\W(T))=1] \\ 
&  \leq \mathbb{E}_{\widehat{d}(x)}^\epsilon\left[g 
\left(\mathbb{P}^\epsilon_{B_{S'}-\nu T+K_1  \epsilon |\log(\epsilon)| }
[\mathbb{V}^\gamma_p(\mathbf{B}(T-S')=1]\right)\1_{S'\leq T-\delta}\right] \\ 
& \hspace{50 mm} + \frac{3+5\gamma_\epsilon}{4(1+\gamma_\epsilon)} \epsilon^l
%+ 2\epsilon^{l+1} + 2\epsilon^{l+2}(1+\gamma_\epsilon) \\ 
+ \epsilon^{l+1}\big(2+\gamma_\epsilon\big) \\ 
& \leq \EE^\epsilon_{\widehat{d}(x)-\nu T+K_1  \epsilon|\log(\epsilon)|}
[\mathbb{V}^\gamma_p(\mathbf{B}(T))=1] 
+ \frac{3+5\gamma_\epsilon}{4(1+\gamma_\epsilon)} 
\epsilon^l+
%2\epsilon^{l+1}(1+\epsilon+\gamma_\epsilon \epsilon).
\epsilon^{l+1}\big(3+2\gamma_\epsilon \big),
\end{align*}
where, analogously to~(\ref{matrak:M1}), in the last line we have applied
the strong Markov propety at time $S' \wedge (T-\delta)$.
For sufficiently small $\epsilon$ we have that 
$\frac{3+5\gamma_\epsilon}{4(1+\gamma_\epsilon)}\epsilon^l
+2\epsilon^{l+1}(1+\epsilon+\gamma_\epsilon \epsilon) 
\leq \frac{7+9\gamma_\epsilon}{8(1+\gamma_\epsilon)}\epsilon^l <\epsilon^l$, 
which completes the proof.
\end{proof}
\begin{proof}[Proof of Lemma~\ref{Lemma:pushcomp}]
Our approach mirrors the proof of Lemma~2.20
of~\cite{gooding:2018}, which builds on that of
Lemma~2.18 of~\cite{etheridge/freeman/penington:2017}. 
The ideas are simple, but
they are easily lost in the notation.

Let us write, for $u \geq 0$ and $z \in \mathbb{R}$,
\begin{equation}
\label{notation for Q} 
\mathbb{Q}_{z}^{\epsilon,u} = \mathbb{P}_z^\epsilon[
\mathbb{V}^\gamma_p(\mathbf{B}(u))=1], 
\end{equation}
and in what follows we use 
$\mathbb{Q}_z^u := \mathbb{Q}_z^{u,\epsilon}$ when there is 
no confusion in doing so.

We consider three cases: % that relates how close is $x$ to $N_\mathbbm{r}$:
\begin{enumerate}
    \item $\widehat{d}(x) \leq -(l+2\mathbbm{d}(l+1)+K_1 ) \epsilon |\log(\epsilon)| $.
    \item $\widehat{d}(x) \geq (l+2\mathbbm{d}(l+1)+K_1) \epsilon |\log(\epsilon)|$.
    \item $|\widehat{d}(x)| \leq (l+2\mathbbm{d}(l+1)+K_1) \epsilon |\log(\epsilon)|$.
\end{enumerate}

\noindent
{\em Case 1:}
Let us define
\[ 
A_x = \Bigg\{ \sup_{u \in [0,s]} \|W_u -x\| \leq 2\mathbbm{d} (l+1)\epsilon 
| \log(\epsilon)| \Bigg\}. 
\]

We estimate the probability of $A_x^c$ 
exactly as in the proof of Proposition~\ref{prop:maxmov}. 
Indeed
reading off from equation~(\ref{standard gaussian}) with $l=k$,
$e(k)=2\dim (k+1)$, $a(k)=0$, we obtain that, for $\widehat{\epsilon}$ small enough, 
\begin{equation} \nonumber %\label{boundA} 
\PP_x^\epsilon\big[A^c_x\big] \leq c_1 
\epsilon^{\dim (l+1)}.
\end{equation}

Now suppose that $A_x$ occurs, then the first component of $W_s$ is
negative, and so, using equation~(\ref{prop:curv}),
\begin{multline}
\label{matrak:ineqc1}
\widehat{d}(W_s) + K_1 \epsilon|\log(\epsilon)| 
=\widehat{d}(x)+K_1\epsilon|\log(\epsilon)|+\widehat{d}(W_s)-\widehat{d}(x)\\
\leq -(l +2\mathbbm{d}(l+1))\epsilon |\log(\epsilon)| + 
|\widehat{d}(W_s) - \widehat{d}(x)| \\ 
 = -(l +2\mathbbm{d}(l+1))\epsilon |\log(\epsilon)| 
+ \big|\|W_s\| - \|x\|\big|  
\leq - l \epsilon |\log(\epsilon)|. 
\end{multline}
Therefore, reducing $\widehat{\epsilon}$ if necessary to ensure that 
$\gamma_{\widehat{\epsilon}}\in (0,1)$, applying Lemma~\ref{1d:decay}, 
and using 
the definition of $g$ and the notation~(\ref{notation for Q}),
\begin{align*} 
\mathbb{E}_x\left[g\left(\mathbb{Q}_{\widehat{d}(W_s)-\nu (t-s)+K_1\epsilon 
|\log(\epsilon)|}^{t-s}+\epsilon^l\right)\right] 
&\leq \mathbb{E}_x[g(\epsilon^l + \epsilon^l)\1_{A_x}] 
+ \PP_x[A_x^c] \\ &\leq 12 %\frac{12}{1 + \gamma_\epsilon} 
\epsilon^{2l}+\frac{4\gamma_\epsilon}{1+\gamma_\epsilon} \epsilon^l+
4 \mathbbm{d} \epsilon^{\dim (l+1)}. 
\end{align*}
Reducing $\widehat{\epsilon}$ still further if necessary, we can certainly 
arrange that for $\epsilon\in (0,\widehat{\epsilon})$
the last line is bounded above by 
$\frac{3+5 \gamma_\epsilon}{4(1+\gamma_\epsilon)} \epsilon^l$.

\noindent
{\em Case $2$:} First observe that, using the reflection principle and
the standard bound, 
$\PP[Z \geq x] \leq \exp\left(-x^2/2\right)$,
on the tail of the standard normal distribution, 
\[ \PP_{\widehat{d}(x)}\big[B_s \leq l \epsilon |\log(\epsilon)|\big]
\leq 2\IP_0\big[B_{(l+1)\epsilon^2|\log\epsilon|}<-2\dim (l+1)\epsilon |\log\epsilon |\big]
\leq 2 \epsilon^{\dim^2 (l+1)}. \]
Again applying Lemma~\ref{1d:decay}, we obtain
\begin{align*}
\mathbb{E}_{\widehat{d}(x)}\left[g\left(\mathbb{Q}_{B_s+K_1 \epsilon 
|\log(\epsilon)|-\nu t}^{t-s}\right)\right] 
&\geq \mathbb{E}_{\widehat{d}(x)}\left[g
\left(\mathbb{Q}_{B_s+K_1 \epsilon |\log(\epsilon)|- \nu t}^{t-s}\right) 
\1_{B_s \geq l \epsilon |\log(\epsilon)|}\right] -2\epsilon^{\dim^2 (l+1)}
\\ & \geq g(1-\epsilon^l)-2\epsilon^{\dim^2 (l+1)} \\ 
&= 1 - \frac{(3+\gamma_\epsilon)}{1+\gamma_\epsilon} \epsilon^{2l} + 
\frac{2}{1+\gamma_\epsilon}\epsilon^{3l} 
-  2\epsilon^{\dim^2 (l+1)},
\end{align*}
where the  final equality follows from the definition of $g$. 
Reducing $\widehat{\epsilon}$ if necessary, 
we have that for $\epsilon \in (0,\widehat{\epsilon})$, 
%\[\frac{(3-\gamma)}{1+\gamma} \epsilon^{2l} - \frac{2}{1+\gamma}\epsilon^{3l} +\epsilon^{l+1} \leq \frac{3}{4} \epsilon^l, \]
%and then we have that
\[ 
\mathbb{E}_{\widehat{d}(x)}\left[g 
\left(\mathbb{Q}_{B_s+K_1  \epsilon|\log(\epsilon)|}^{t-s}\right) \right] 
%\geq 1- \frac{3}{4} \epsilon^l 
\geq  1 - \frac{3+5\gamma_\epsilon}{4(1+\gamma_\epsilon)} \epsilon^{l}, 
\]
and in this case the right hand side of~(\ref{bootstrap})
is greater than or equal to one, while the left hand
side is less than or equal to one by definition.

{\em Case $3$:} This is the most difficult case. We combine the coupling
of Lemma~\ref{teo:coup}
with our lower bound on the slope of the one-dimensional interface.

Again, suppose that $A_x$ holds. %Furthermore, 
Since we have chosen $\widehat{\epsilon}$ small enough that 
$\mathbbm{r}-r_0 \geq (l+K_1+4\dim (l+1)) \epsilon |\log(\epsilon)|$
for all 
$\epsilon \in (0,\widehat{\epsilon})$,  
on the event $A_x$ the first component of $W_s$ is negative, 
and, arguing as in~(\ref{matrak:ineqc1}), using equation~(\ref{prop:curv}) 
we obtain
\begin{equation*}
|\widehat{d}(W_s)|
%  \leq \|W_s - x\| + |\widehat{d}(x)| \\  
%\leq (2(l+1)\mathbbm{d}+l+2\mathbbm{d}(l+1)+K_1) \epsilon | \log(\epsilon)| \\ 
\leq \epsilon |\log(\epsilon)| (4 \mathbbm{d}(l+1)+l+K_1).
\end{equation*}
Choosing $\widehat{\epsilon}$ still smaller if necessary, 
for all $\epsilon\in (0,\widehat{\epsilon})$,
\begin{equation} 
\label{desbet} 
\big(4 \mathbbm{d}(l+1)+l+K_1\big) 
\epsilon |\log(\epsilon)|
\leq \frac{\mathbbm{r}-r_0}{2} \wedge \frac{R_0-\mathbbm{r}}{2}. 
\end{equation}
We write $\beta$ for the quantity on the left hand side of~(\ref{desbet})
and use Lemma~\ref{teo:coup} to couple the reflected Brownian motion $W$ 
started at $x\in\Omega$, with a one-dimensional Brownian motion $B$ started 
at $\widehat{d}(x)$ such that, up to time $T_\beta$, we have
\begin{equation} 
%\label{ineq:coup} 
\nonumber
\widehat{d}(W_s) \leq B_s - \frac{s(\mathbbm{d}-1)}{\mathbbm{r}+\beta}. 
\end{equation}
Combining this with the monotonicity of $g$ and the fact that
$\{ T_\beta \geq s \} \subseteq A_x^c$ yields 
\begin{align}
\mathbb{E}_x& \left[g\left(\mathbb{Q}_{\widehat{d}(W_s)-\nu(t-s)
+ K_1 \epsilon|\log(\epsilon)|}^{t-s}+\epsilon^l\right) \right] 
\nonumber\\ 
& \hspace{20 mm} \leq \mathbb{E}_{\widehat{d}(x)}\left[
g\left(\mathbb{Q}_{B_s-\nu(t-s)-\frac{s (\mathbbm{d}-1)}{\mathbbm{r}+\beta}+
K_1 \epsilon |\log(\epsilon)|}^{t-s}+\epsilon^l\right)\right] 
+ \mathbb{P}_x[T_\beta \geq s] 
\nonumber \\ 
& \hspace{20 mm} \leq \mathbb{E}_{\widehat{d}(x)}\left[
g\left(\mathbb{Q}_{B_s-\nu(t-s)-\frac{s (\mathbbm{d}-1)}{\mathbbm{r}+\beta}+
K_1 \epsilon |\log(\epsilon)|}^{t-s}+\epsilon^l\right)\right] 
%& \hspace{20 mm} \leq \mathbb{E}_{\widehat{d}(x)}\left[g\left(
%\mathbb{Q}_{B_s-\nu(t-s)-\frac{(\mathbbm{d}-1)s 
%\mathbbm{r}^{-1}}{1 + \mathbbm{r}^{-1} \beta}+
%K_1 \epsilon |\log(\epsilon)|}^{t-s}+\epsilon^l\right)\right] 
+ 
4 \mathbbm{d} \epsilon^{\dim (l+ 1)}.
\label{coupling and monotonicity}
\end{align}
Reducing $\widehat{\epsilon}$ if necessary, we have 
that $\mathbbm{r}^{-1} \beta <1$ for all $\epsilon \in (0,\widehat{\epsilon})$.
Thus
\[ 
(\mathbbm{d}-1)\frac{s}{\mathbbm{r}} 
\left(\frac{1}{1+\mathbbm{r}^{-1}\beta}\right) 
\geq \frac{s(\mathbbm{d}-1)}{\mathbbm{r}}( 1 - \mathbbm{r}^{-1} \beta ), 
\]
and so
\begin{equation} 
\label{matrak:neg} 
\nu s - \frac{(\mathbbm{d}-1)s}{\mathbbm{r}+\beta} 
\leq s\left[\big(\nu - \frac{1}{\mathbbm{r}}(\mathbbm{d}-1)\big)+
\frac{\mathbbm{d}-1}{\mathbbm{r}^2} \beta \right]. 
\end{equation}
Recall that, $\mathbbm{r} < (\mathbbm{d}-1)/ \nu$ so that 
$\nu - (\mathbbm{d}-1)/\mathbbm{r} <0$, and so if
$\widehat{\epsilon}$ is small enough, we have that 
for $\epsilon \in (0,\widehat{\epsilon})$ 
\begin{equation}
\label{matrak:ineqb} 
\frac{\mathbbm{d}-1}{\mathbbm{r}^2}\beta + \epsilon | \log(\epsilon)|
< \frac{\mathbbm{d}-1}{\mathbbm{r}}-\nu. 
\end{equation}
In particular, the right hand side of (\ref{matrak:neg}) is negative. 

Define
\[
z = B_s-\nu t + s\left[(\nu - \frac{1}{\mathbbm{r}}(\mathbbm{d}-1) + 
\beta \frac{\mathbbm{d}-1}{\mathbbm{r}^2})\right]+K_1\epsilon|(\log(\epsilon)|.
\]
Observe that using~(\ref{matrak:ineqb}) we have
\begin{align*} 
%\label{ineq:push}
- s\left[\left(\nu - \frac{1}{\mathbbm{r}}(\mathbbm{d}-1)\right)+
\frac{\mathbbm{d}-1}{\mathbbm{r}^2} \beta\right] 
&= s\left[\left(\frac{1}{\mathbbm{r}}(\mathbbm{d}-1) - \nu\right) 
- \frac{\mathbbm{d}-1}{\mathbbm{r}^2}\beta \right] \nonumber \\ 
&\geq  s \epsilon |\log(\epsilon)|,
\end{align*}
so that
\begin{equation}
\label{bound on z}
z\leq B_s-\nu t +(K_1-s)\epsilon |\log\epsilon |.
\end{equation}
Consider the event:
 \[ 
E = \left\{ \left|\mathbb{Q}_{B_s-\nu t + 
s[(\nu - \frac{1}{\mathbbm{r}}(\mathbbm{d}-1))+
\frac{\mathbbm{d}-1}{\mathbbm{r}^2} \beta]+
K_1 \epsilon |\log(\epsilon)|}^{t-s} - \frac{1}{2}\right| 
\leq \frac{5 + \gamma_\epsilon}{12} \right\}. 
\]
As explained in~\cite{gooding:2018}, although it looks slightly unnatural to 
take a set centred on the value $1/2$ (about which $g$ is symmetric only
in the case when $\gamma=0$), the importance of $E$ is that it spans
the interface (where $\IQ$ takes the value $(1+\gamma_\epsilon)/2$), and 
on $E^c$, $g'(\IQ)<1$.
%$\mathbb{Q}_{B_s-\nu t + 
%s[(\nu - \frac{1}{\mathbbm{r}}(\mathbbm{d}-1))+
%\frac{\mathbbm{d}-1}{\mathbbm{r}^2} \beta]+
%K_1 \epsilon |\log(\epsilon)|}^{t-s}$.

Suppose first that $E$ occurs. 
We apply Proposition~\ref{prop:slope} with 
%\[
%z = B_s-\nu t + s[(\nu - \frac{1}{\mathbbm{r}}(\mathbbm{d}-1) + 
%\beta \frac{\mathbbm{d}-1}{\mathbbm{r}^2})]+K_1\epsilon|(\log(\epsilon)|
%\]
$z$ as above, and 
\[
w = z + s\epsilon | \log(\epsilon)| \leq B_s-\nu t +
K_1\epsilon|\log(\epsilon)|.
\]
Note that $|z-w| = s \epsilon |\log(\epsilon)| 
\leq (l+1) \epsilon^3 |\log(\epsilon)|^2$, so reducing $\widehat{\epsilon}$ if 
necessary so that $(l+1)\epsilon^3 |\log(\epsilon)|^2 \leq \epsilon$ 
for $\epsilon\in (0, \widehat{\epsilon})$, Proposition \ref{prop:slope} implies
\begin{equation} 
\label{ineq:c3m1} 
\mathbb{Q}_{B_s - \nu t + K_1 \epsilon |\log(\epsilon)| +
s [(\nu - \frac{1}{\mathbbm{r}}(\mathbbm{d}-1))+
\frac{\mathbbm{d}-1}{\mathbbm{r}^2} \beta] }^{t-s} \1_E 
\leq \left( \mathbb{Q}_{B_s - \nu t + K_1 \epsilon |\log(\epsilon)|}^{t-s} 
- \frac{s}{48}\right) \1_E. 
\end{equation}
Now suppose that $E^c$ occurs. Since 
$g'(p) = \frac{2}{(1+\gamma_\epsilon)}(1-p)(3p+\gamma_\epsilon)$, 
for $p,\delta \geq 0$ with $p+\delta \leq \frac{1-\gamma_\epsilon}{9}$ or 
$p \geq \frac{8+\gamma_\epsilon}{9}$ it is easy to check that
\begin{equation} 
\label{ineq:mv} 
g(p+\delta) \leq g(p)
+\frac{2(1+2\gamma_\epsilon)}{3(1+\gamma_\epsilon)} \delta. 
\end{equation}
Let $C_{\gamma_\epsilon} = \frac{2(1+2\gamma_\epsilon)}{3(1+\gamma_\epsilon)}$. Then, 
reducing $\widehat{\epsilon}$ if necessary so that 
$\frac{1-\gamma_\epsilon}{12}+\epsilon^l \leq \frac{1-\gamma_\epsilon}{9}$, 
for $\epsilon \in (0,\widehat{\epsilon})$, we obtain
\begin{multline} 
\label{ineq:c3m2} 
g\left(\mathbb{Q}_{B_s - \nu t + K_1 \epsilon |\log(\epsilon)| 
+ s[(\nu - \frac{1}{\mathbbm{r}}(\mathbbm{d}-1))
+\frac{\mathbbm{d}-1}{\mathbbm{r}^2} \beta] }^{t-s} 
+ \epsilon^l \right) \1_{E^c}  \\ 
\leq \left(g\Big(\mathbb{Q}_{B_s - \nu t + 
K_1 \epsilon |\log(\epsilon)|}^{t-s}\Big) + 
C_{\gamma_\epsilon} \epsilon^l\right) \1_{E^c}, 
\end{multline}
where we have used~(\ref{ineq:mv}),~(\ref{bound on z}) and the
   %~(\ref{matrak:ineqb}) and the 
monotonicity of $g$. 
Using~(\ref{coupling and monotonicity}),~(\ref{ineq:c3m1}), 
and~(\ref{ineq:c3m2}) 
we obtain
\begin{align*}
\mathbb{E}_x&\left[g\left(\mathbb{Q}_{\widehat{d}(W_s)-\nu(t-s)
+K_1 \epsilon |\log(\epsilon)|}^{t-s}+\epsilon^l\right) \right] \\
& \leq \mathbb{E}_{\widehat{d}(x)}\left[ 
g\left(\mathbb{Q}^{t-s}_{B_s-\nu t+K_1 \epsilon |\log(\epsilon)|}
-\frac{1}{48}s + \epsilon^l \right)\1_E \right] \\ 
&\hspace{20mm} + \mathbb{E}_{\widehat{d}(x)} \left[ 
\left(g\Big(\mathbb{Q}^{t-s}_{B_s-\nu t+K_1 \epsilon |\log(\epsilon)|}\Big) %%%%%%%%
+ C_{\gamma_\epsilon}\epsilon^l\right)\1_{E^c} \right] + 4 \mathbbm{d} 
\epsilon^{\dim (l+1)} \\ 
& \leq \mathbb{E}_{\widehat{d}(x)}\left[
g\left(\mathbb{Q}^{t-s}_{B_s-\nu t+K_1 \epsilon |\log(\epsilon)|} \right) 
\right] + C_{\gamma_\epsilon} \epsilon^l 
+ \epsilon^l 1_{s \leq 48 \epsilon^l} + 
4 \mathbbm{d} \epsilon^{\dim (l + 1)},
\end{align*}
where we have used that $g'(p)\leq 1+C_{\gamma_\epsilon}$ for all $p\in [0,1]$.
%provided that $s \leq 48 \epsilon^l$, where we have used
%that $|g'_\gamma(p)| \leq C_\gamma + 1$ for all $p \in [0,1]$. 
Finally notice that 
$C_{\gamma_\epsilon} + \frac{1-\gamma_\epsilon}{12(1+\gamma_\epsilon)} 
= \frac{3+5\gamma_\epsilon}{4(1+\gamma_\epsilon)}$. 
Reducing $\widehat{\epsilon}$ if necessary so that 
$4 \mathbbm{d} \epsilon^{\dim (l+1)} 
\leq \frac{1-\gamma_\epsilon}{12(1+\gamma_\epsilon)} \epsilon^l$ and 
$48 \epsilon^l \leq \epsilon^3$ for $\epsilon \in (0,\widehat{\epsilon})$ 
(which is possible since $l \geq 4$) completes the proof.
\end{proof}
\begin{proof}[Proof of Theorem \ref{teo:decayv2}]
Since we are concerned with small 
$\epsilon$, it will be sufficient to
prove the result for $k\geq 4$.

Let $K_1$ be given by Proposition~\ref{prop:ineqonetwo}, and
$\widehat{\epsilon}$ be small enough that Proposition~\ref{prop:ineqonetwo} and 
Lemma~\ref{Lemma:pushcomp} hold. Set $c(k)=k+K_1$ so that for any 
$\epsilon \in (0,\widehat{\epsilon})$ and $x \in \Omega$ such that 
$\widehat{d}(x) \leq -c(k)\epsilon|\log(\epsilon)|$ we have that
$\widehat{d}(x) + K_1 \epsilon | \log(\epsilon)| 
\leq - k \epsilon |\log(\epsilon)|$.
Now choose $a(k)$ as in Proposition~\ref{prop:ineqonetwo}. 

For $t \leq a(k) \epsilon^2 |\log(\epsilon)|$ the result holds by 
Proposition~\ref{prop:maxmov}. Indeed, since by
definition $K_1=b(k)+k=2e(k)+k$, it holds that $K_1 \geq e(k)$. From this it follows that, if $\hat{d}(x) \leq c(k)\epsilon|\log(\epsilon)|$,
\begin{align*}
 \PP_x^\epsilon\big[\mathbb{V}_{\widehat{p}}^\gamma(\W(t))=1)\big] & 
\leq \PP_x^\epsilon\Big[\exists i \in N(s) : 
\Vert W_i(s) - x \Vert \geq e(k) \epsilon |\log(\epsilon)|\Big] \leq \epsilon^k.    
\end{align*}
On the other hand, for any 
$t \in [a(k)\epsilon^2|\log(\epsilon)|,\infty)$,
\begin{align*}
 \PP_x^\epsilon[\mathbb{V}_{\widehat{p}}^\gamma(\W(t))=1)] 
&\leq \epsilon^k + \PP_{\widehat{d}(x)-\nu t+ 
K_1 |\log(\epsilon)|}^\epsilon[\mathbb{V}_p^\gamma(\mathbf{B}(t))=1] \\ 
& \leq  \epsilon^k + \PP_{-k \epsilon | \log(\epsilon)|-\nu t}^\epsilon[
\mathbb{V}_p^\gamma(\mathbf{B}(t))=1] 
\leq 2 \epsilon^k,
\end{align*}
where the last line is Lemma~\ref{1d:decay}. This completes the proof.
\end{proof}

\subsection{Invasion (proof of Theorem~\ref{no_blocking})}
\label{proof of no blocking}

We now turn to the proof of Theorem~\ref{no_blocking}. Recall that 
we are now supposing that the narrower cylinder in the domain $\Omega$
of Figure~\ref{fig:omega} has radius $r_0>(\dim -1)/\nu$. 
Our proof will mirror the `sliding' technique used in the proof of complete 
propagation in~\cite{berestycki/bouhours/chapuisat:2016}.
The key step is the
following proposition, which establishes a lower bound on the solution 
started from $(1-\epsilon)$ times 
the indicator function of a ball, whose radius is greater than 
$(\dim -1)/\nu$, sitting within $\Omega\cap\{x_1>0\}$.

\begin{proposition}
\label{circles get bigger}
Suppose that $r_0>(\dim -1)/\nu$ and set 
$$r^*=\frac{r_0+(\dim -1)/\nu}{2}.$$
Consider the solution $\widetilde{u}^\epsilon$ to $(AC_\epsilon)$ with the 
initial condition replaced by 
$\widetilde{u}^\epsilon (x,0)=(1-\epsilon)\1_{B(x^0,r^*)}(x),$
%$(1-\epsilon)$ times the indicator of the ball of radius $r^*$ centred on $x^0$,
where $x^0=(x^0_1,\mathbf{0})$ and $\mathbf{0}$ denotes the origin in 
$\IR^{\dim -1}$. 
Let $\widetilde{\v{\Gamma}}$ be the solution 
to~(\ref{curvature plus constant flow})
started from the boundary of $B(x^0, r^*)$, and  
$T_{r^*}$ be the time at which it is 
is equal to the boundary of $B(x^0, \rho^*)$
with $\rho^*$ defined by
$$\rho^*=r^*+\frac{r_0-r^*}{4}.$$ 
There exist constants $a$ and $\widehat{\epsilon}>0$ such that for all 
$\epsilon\in (0,\widehat{\epsilon})$, 
$$\widetilde{u}^{\epsilon}(x,T_{r^*})>1-\epsilon, \qquad\mbox{for all }
x\in B(x^0, \rho^*-a\epsilon |\log \epsilon|).$$
\end{proposition}

\begin{proof}[Outline of proof]

Choose $\widehat{\epsilon}$ so that $r_0-r^*>\widehat{\epsilon}
|\log(\widehat{\epsilon})|^2$.

By reducing $\widehat{\epsilon}$ if necessary, assume that it is less 
than $\widehat{\epsilon}(4)$ in each of Lemma~\ref{1d:decay},
and Lemma~\ref{lemma:bigthree};
and small enough that the conditions of 
Proposition~\ref{prop:slope} are satisfied.

Denote the signed distance of the point $y \in \Omega$
to $\widetilde{\v{\Gamma}}_t$ by $\widetilde{d}_t(y)$ (with the convention that
$\widetilde{d}_0(x^0)>0$). 
Set 
$$\beta=\frac{r^*}{2}\wedge \frac{r_0-r^*}{2},$$
and 
$$\widetilde{T}_\beta=\inf\{ t\geq 0: \widetilde{d}_t(W_t)\geq \beta\}
\wedge T_{r^*}.$$

Exactly as in Lemma~\ref{teo:coup},
for $0\leq s\leq \widetilde{T}_\beta$, there
exists a one-dimensional Brownian motion $\widetilde{B}$,
started from $\widetilde{d}(W_0)$
such that
\begin{equation} \widetilde{B}_s-\frac{s(\dim -1)}{\rho^*+\beta}\geq
\widetilde{d}(W_s)\geq \widetilde{B}_s-\frac{s(\dim -1)}{r^*-\beta}. \label{coupling for first invasion} \end{equation}

An argument entirely analogous to the proof of 
Proposition~\ref{prop:interface} (with $k=4$), where now we bound above the 
probability that a leaf votes $0$ by $\epsilon$ (corresponding to
it falling within
$B(x^0,r^*)$) plus the probability that it lies outside 
$B(x^0,r^*)$,
yields that by choosing $\widehat{\epsilon}$ smaller still if necessary, 
there exist $a, b>0$ such that for all $\epsilon\in (0, \widehat{\epsilon})$,
setting $\delta =a\epsilon^2|\log(\epsilon)|$
and $\delta'=(a+5)\epsilon^2|\log(\epsilon)|$, for $t\in [\delta(\epsilon),
\delta'(\epsilon)]$, and $x$ with $\widetilde{d}(x)>b\epsilon |\log(\epsilon)|$,
$$\IP_x^\epsilon[\mathbb{V}_{\widetilde{u}^\epsilon(\cdot,0)}(\W(t))=1]
\geq 1-\epsilon^4.$$

Arguing as in the proof of Lemma~\ref{Lemma:pushcomp},
we can show that there is $K_1>0$ such that, choosing $\widehat{\epsilon}$
still smaller if necessary,
for all 
$\epsilon \in (0,\widehat{\epsilon})$, 
$x \in \Omega$, $s \in [0,5\epsilon^2 |\log(\epsilon)|]$ 
and $t \in [s,T^*)$,
\begin{multline}
\label{bootstrap2}
\EE_x\left[g\left(\PP^\epsilon_{\widetilde{d}(W_s)-\nu(t-s)-K_1 \epsilon|\log(\epsilon)|}
[\mathbb{V}_p^\gamma(\mathbf{B}(t-s)=0]+\epsilon^4\right)\right] 
 \\ 
\leq \frac{3 + 5 \gamma_\epsilon}{4(1+\gamma_\epsilon)}\epsilon^4 + 
\EE_{\widetilde{d}(x)}^\epsilon\left[g\left(\PP^\epsilon_{B_s-\nu t-K_1 \epsilon |\log(\epsilon)|}
[\mathbb{V}_p^\gamma(\mathbf{B}(t-s))=0]\right)\right] + 
1_{s \leq \epsilon^3} \epsilon^4.
\end{multline}
The argument is once again simpler than the general case considered
in~\cite{gooding:2018}, since in this setting, over the time interval in
which we are interested,  
$\widetilde{d}$ is simply the distance to the boundary of a ball.

From this we can proceed as in the proof of 
Proposition~\ref{prop:ineqonetwo}
to show that 
there exist $K_1$ and 
$\widehat{\epsilon} > 0$ such that, for all 
$\epsilon \in (0, \widehat{\epsilon}$, $t \in [\delta(l,\epsilon),T_{r^*}]$ 
we have
\[ 
\sup_{x \in \Omega}\left(\PP_x^\epsilon[
\mathbb{V}_{\widetilde{u}^\epsilon (\cdot, 0)}^\gamma(\W(t)) = 0)]- 
\PP^\epsilon_{\widetilde{d}(x) - \nu t- K_1 \epsilon |\log(\epsilon)|} 
[\mathbb{V}^\gamma_p(\mathbf{B}(t)) = 0] \right) \leq \epsilon^4, 
\]
with $\mathbb{V}^\gamma_p(\mathbf{B}(t))$ as in~(\ref{eq:solvot}).

Finally, we can complete the proof with an argument that mirrors that
of Theorem~\ref{teo:decayv2}.
\end{proof}

We are now in a position to prove Theorem~\ref{no_blocking}. 

\begin{proof}[Proof of Theorem~\ref{no_blocking}]

The idea is simple. 
We use Proposition~\ref{circles get bigger} to provide a series of 
lower solutions to $(AC_\epsilon)$. First take $x^0=(r^*,\mathbf{0})$.
By Proposition~\ref{prop:cp}, for $\epsilon <\widehat{\epsilon}$
the solution to $(AC_\epsilon)$ at time $T_{r^*}$ dominates
$\widetilde{u}^\epsilon(\cdot, T_{r^*})$. In particular, it is at least
$1-\epsilon$ on $B(x^1, r^*)$, where 
$x^1=(r^*-(\rho^*-r^*-a\epsilon |\log(\epsilon)|), \mathbf{0})$.
By choosing $\widehat{\epsilon}$ smaller if necessary, we can certainly
arrange that $\|x^0-x^1\|\geq (r_0-r^*)/8$.
Notice also that since $r^*>(\dim -1)/\nu$, the time $T_{r^*}$ is finite. 

We can now simply iterate. The solution to $(AC_\epsilon)$ at time
$2T_{r^*}$ dominates that started from $(1-\epsilon)$ times the 
indicator function of the ball of radius $r^*$ centred on $x^1$, which is 
at least $(1-\epsilon)$ on the ball radius $r^*$ centred on $x^2$, where
$\|x^1-x^2\|\geq (r_0-r^*)/8$ and so on. 
As illustrated in Figure~\ref{sketch_invasion},
any point $x\in\Omega$ can be 
connected to the right half-space by a finite chain of balls in this way, and
the result follows. 
\end{proof}

\begin{figure}
\begin{center}
(i)\includegraphics[height=2.5in]{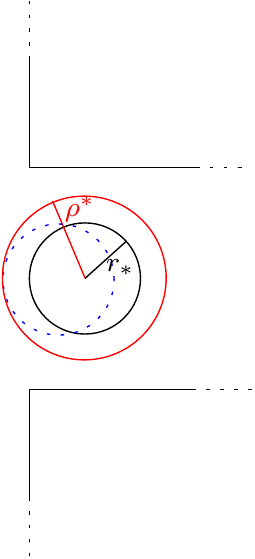}
\hspace{15 mm}
(ii) \includegraphics[height=2.5in]{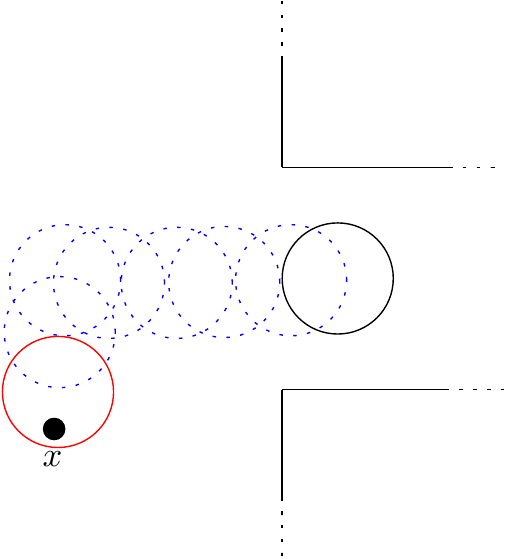}
\end{center}
\caption{(i) Illustration of Lemma~\ref{circles get bigger}: 
started from $(1-\epsilon)$ times the indicator function of the ball of radius
$r^*$, at time $T_{r^*}$ the solution exceeds $(1-\epsilon)$ on the 
larger ball, and a fortiori on the dashed ball of radius $r^*$ 
with centre shifted
to the left. (ii) Illustration of how a chain of balls constructed in this
way can link any point in $\Omega$ to the $\Omega\cap\{x:x_1>0\}$.}
\label{sketch_invasion}
\end{figure}

\subsection{Other domains} \label{sec:mordom}

The crux of the proof of Theorem~\ref{teo:simplifyversion}
was the detailed analysis, close to $N_{\mathbbm{r}}$,
of the supersolution with
initial condition $\widehat{p}$.
The key to defining the supersolution was to be 
able to completely cover
the opening $\cal O$ with a hemispherical shell of radius at most 
$(\dim -1)/\nu$, that is completely contained in $\Omega\cap\{x_1<0\}$
(i.e.~the portion of $\Omega$ to the left of the origin) and
intersects the boundary $\partial\Omega$ at right angles. Evidently, we should
be able to prove an entirely analogous result for any domain in which we can
identify an appropriate analogue of $N_{\mathbbm{r}}$ and control the solution
around it. 

The proof would go through completely unchanged for domains of the 
form $\widehat{\Omega}$ of Figure~\ref{figure:dr}, for example, provided that 
we could cover the disjoint union of openings by a single hemispherical
shell of radius strictly less than $(\dim -1)/\nu$.
To see how we can recover an analogue of Theorem~\ref{teo:simplifyversion},
for more general domains of the form~(\ref{domain}), 
we first consider another special case.
\begin{proposition}
\label{blocking in a cone}
Let 
$$\widetilde{\Omega}=\left\{(x_1, x')\subseteq\IR^\dim: x'\in \phi(x_1)
\subseteq\IR^d\right\},$$
with %%%%%%%%%
$$\phi(x_1)= \begin{cases} \Vert x'\Vert < r_0 & x_1 \geq 0 \\ \Vert x'\Vert < r_0 -x_1 \tan (\alpha) & x_1 < 0 \end{cases},$$
for some $\alpha\in (0,\pi/2]$ and suppose that 
\begin{equation}
\label{condition for blocking in omega tilde}
r_0<\frac{\dim -1}{\nu}\sin\alpha
.
\end{equation}
Let $\mathbbm{r}$ satisfy
\begin{equation}
\label{condition on alpha for conical domain}
\frac{r_0}{\sin\alpha}<\mathbbm{r}<\frac{\dim-1}{\nu},
\end{equation}
and define
$$\widetilde{N}_{\mathbbm{r}}=\Big\{x=(x_1,x')\in\widetilde{\Omega}:
x_1 < 0, \Big\| \left(x_1-\frac{r_0}{\tan\alpha}, x' \right) \Big\|=\mathbbm{r}\Big\}.$$
We write $\widetilde{d}$ for the signed distance to 
$\widetilde{N}_{\mathbbm{r}}$ (chosen to be negative as $x_1\to-\infty$).
Let $k\in \IN$. 
Then there is $\widehat{\epsilon}(k)>0$ and $a(k)$, $M(k)>0$ 
such that for all $\epsilon \in (0,\widehat{\epsilon})$, $t \in (a(k)\epsilon^2 |\log(\epsilon)|,\infty)$ we have that:
\[ \text{if} \: \: x=(x_1,\ldots,x_{\mathbbm{d}}) \: \:\text{is such that} \: \: 
\widetilde{d}(x) \leq -M \epsilon |\log(\epsilon)| \: \: \text{then} \: \: w(x,t) \leq \epsilon^k. \]
\end{proposition}
\begin{remark}
The condition~(\ref{condition for blocking in omega tilde})
becomes natural upon observing 
that any spherical shell intersecting the boundary of $\widetilde{\Omega}$
at right angles must have radius at least $r_0/\sin\alpha$,
\end{remark}

%%%%%%%%%%%%%%%%%%%%%%%%%%%%%%%%%%%%%%%%%%%%%%%%%%%%%%%%%%%%%%%%%%%%
\begin{figure}
\begin{center}
(i)\includegraphics[height=2in]{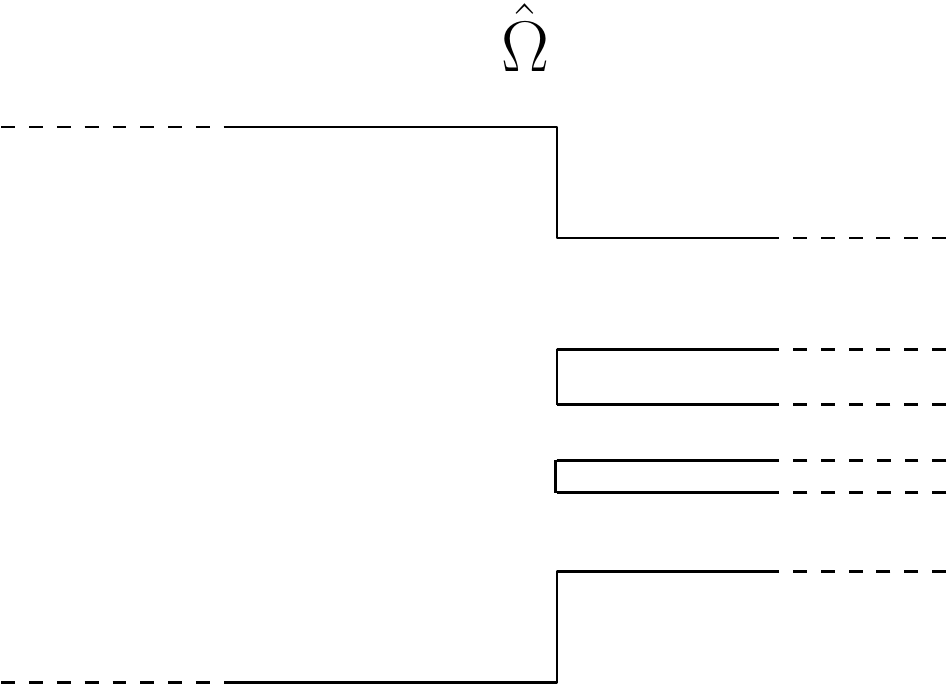}%{CFBlocking/images/omegam1.pdf}
\hspace{0.5in}
(ii) \includegraphics[height=2in]{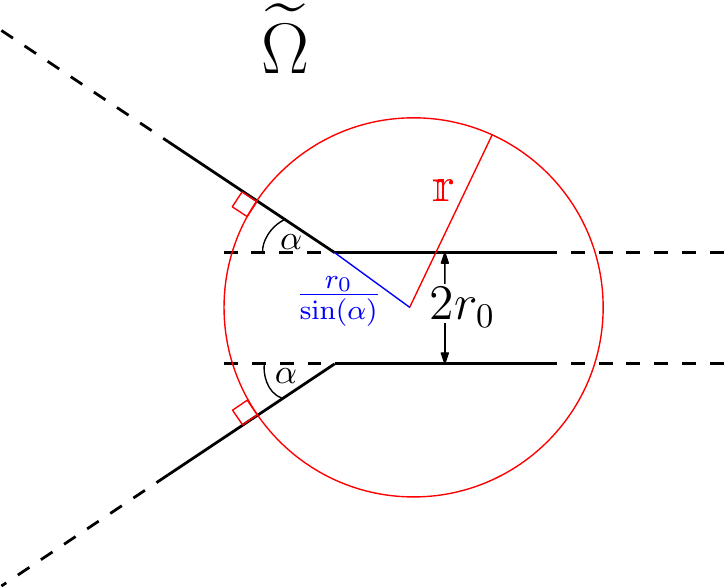}%{CFBlocking/images/omegam2new.pdf}
\end{center}
\caption{(i)Two dimensional representation of 
the domain $\widehat{\Omega}$. 
For $x_1>0$, we have multiple cylindrical domains that open into a 
single cylinder for $x_1<0$.
(ii) Two dimensional representation of the domain $\widetilde{\Omega}$. 
It is composed of a cylindrical component on $x_1>0$ and a (truncated) 
cone for $x_1<0$.}
\label{figure:dr}
\end{figure} 
\begin{proof}[Sketch of proof]
The proof follows the same pattern as that of
Theorem \ref{teo:simplifyversion}; first we dominate the solution by one
with a larger initial condition, $\widetilde{p}$, then we put an 
interface with width of order $\epsilon |\log(\epsilon)|$ around 
the set on which the initial condition takes the value
$(1-\epsilon \nu)/2$ and we reproduce the
proof of Lemma~\ref{Lemma:pushcomp} to see how this interface moves. 

The initial condition that we take satisfies
\begin{enumerate}
    \item $\widetilde{p}(x) = 1$ for all $x \in \widetilde{\Omega}$ such 
that $x_1 \geq 0$;
    \item $\widetilde{p}(x) = \frac{1-\gamma_\epsilon}{2}$ for all 
$x \in \widetilde{N}_\mathbbm{r}$.
    \item $\widetilde{p}(x) > \frac{1-\gamma_\epsilon}{2}$ if 
$\widetilde{d}(x) >0$, and $\widetilde{p}<\frac{1-\gamma_\epsilon}{2}$ 
if $\widetilde{d}(x)<0$.
    \item $\widetilde{p}(x)$ is continuous and there exist $\mu,\eta>0$ 
such that $|\widetilde{p}(x)-\frac{1-\gamma_\epsilon}{2}| \geq 
\mu(\text{dist}(x,\widetilde{N}_\mathbbm{r})\wedge \eta)$.
\end{enumerate}
The conditions of Proposition~\ref{blocking in a cone}
are precisely what is required for these conditions to be compatible. 

Our choice of $\widetilde{N}_{\mathbbm{r}}$ % and the resulting smoothness of 
%$\widetilde{d}$ (which allows us to apply It\^o's lemma) 
enables us to
prove the analogue of Theorem \ref{teo:decayv2} for 
$\mathbb{P}_x^\epsilon[\mathbb{V}_{\widetilde{p}}^{\gamma}(\W(t))=1]$ (using 
the same method).
%Applying It\^o's lemma to $\widetilde{d}(W_s)$ gives a similar result to
The obvious modification of
Lemma~\ref{teo:coup} can then be used to obtain analogues of 
Lemma~\ref{Lemma:pushcomp} and~\ref{prop:ineqonetwo}. 

\end{proof}

%Note that the condition~(\ref{condition on alpha for conical domain})
%mirrors that of~\cite{matano/nakamura/lou:2006} described in 
%Section~\ref{intro:detcase}, except that we are 
%working in a multidimensional setting. 
%The key to the proof is that one can fit a portion of a spherical shell that
%has radius less than $(\dim -1)/\nu$, that is orthogonal to the boundary
%of the domain where the two meet, and which divides the domain into two
%disjoint pieces. Since all our arguments are local, the same proof
%will extend to Theorem~\ref{saw-tooth}.

Armed with the proof of Proposition~\ref{blocking in a cone} 
the main argument behind Theorem~\ref{saw-tooth} is easily understood. 
We recall that the main condition we impose 
is~(\ref{blocking condition saw  domain}), that is, 
\begin{equation} 
\inf_{z>0} \left\{  H + h(z) - 
\left( \frac{\mathbbm{d}-1}{\nu} \right) 
\frac{h'(z)}{\sqrt{1+h'(z)^2}}
\right\} < 0   
\label{blocking condition saw domain recall} 
\end{equation}
\begin{figure}[h]
\begin{center}
(i)\includegraphics[height=1.7in]{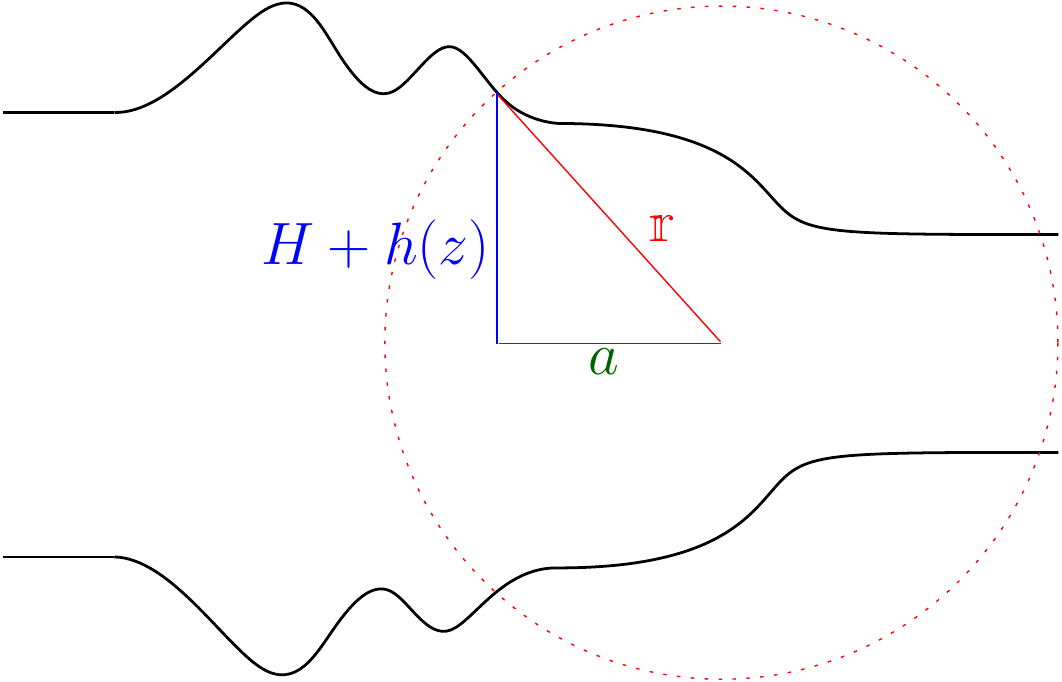}
\hspace{0.5in}
(ii) \includegraphics[height=1.7in]{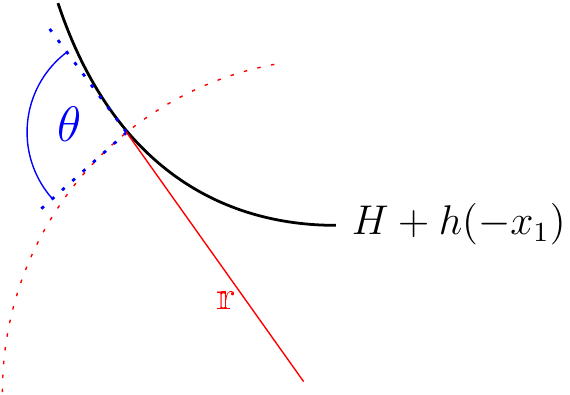}
\end{center}
\caption{(i) Illustration of the set-up used in 
Proposition~\ref{proposition of coupling for blocking in saw domain}. 
The chosen value of $a$ should be obvious from this picture.
(ii) Illustration of $\theta$, the angle at which the spherical shell
intersects the boundary of $\Omega$}
\label{figure:sawdomex1}
\end{figure} 

In what follows we recall the notation $\rho^* := (\mathbbm{d}-1)/\nu$. Let 
\[ N^*_{\mathbbm{r},a} = \left\{ x = (x_1,x') \in \Omega : x_1 < 0, 
\Vert  x-(a,0,...,0) \Vert = \mathbbm{r} \right\}, \]
and define $d^*_{\mathbbm{r},a}$ to be the distance function to 
$N^*_{\mathbbm{r},a}$ (chosen to be negative as $x_1 \rightarrow - \infty$). 

\begin{proposition} 
\label{proposition of coupling for blocking in saw domain}
Let $W$ be Brownian motion and set
\[ T_\beta = 
\inf \left\{ t \geq 0 : |d^*_{\mathbbm{r},a}(W_t)| \geq \beta \right \}. 
\]
Then, if (\ref{blocking condition saw  domain recall}) holds
for some $z>0$, there exist $C, \beta^*>0$, $0 < \mathbbm{r} < \rho^*$,
$a < 0$ and a Brownian motion $\widehat{B}$, 
started from $d_{\mathbbm{r},a}^*(W_0)$ such that, for all 
$0 < \beta \leq \beta^*$ and $0 \leq s \leq T_\beta$,
\begin{equation} 
d_{\mathbbm{r},a}^* (W_s) \leq \widehat{B}_s - 
\frac{s(\mathbbm{d}-1)}{\mathbbm{r}+\beta}  + C \beta s.
\label{coupling for blocking of saw domain} 
\end{equation}
\end{proposition}
\begin{proof} [Sketch of proof]
Let $z$ be such that 
$$
H + h(z) - \left( \frac{\mathbbm{d}-1}{\nu} \right) 
\frac{h'(z)}{\sqrt{1+h'(z)^2}} < 0,
$$
and set 
\begin{equation} 
\mathbbm{r} = \frac{1}{2}\left(\frac{(H + h(z))\sqrt{1+h'(z)^2}}{h'(z)}+ 
\rho^*\right). \label{election of radius for saw domain} \end{equation}
Note that, by our choice of $z$, 
we have $\mathbbm{r} < \rho^*$. 
Let $a = -z+ \sqrt{\mathbbm{r}^2 - (H+h(z))^2}$ and note 
that $N_{\mathbbm{r},a}^*$ intersects the boundary of the domain 
exactly at the point where the first coordinate is $-z$. See 
Figure~\ref{figure:sawdomex1}(i) for an illustration of this situation.

Choose $\beta^*$ small enough that for all 
$\beta < \beta^*$ the distance function $d_{\mathbbm{r},a}^*$ satisfies 
the conditions ($\mathscr{C}1$)-($\mathscr{C}3$) introduced 
above Theorem~\ref{theorem ac to cf} in 
$U_\beta = \{x: |d_{\mathbbm{r},a}^*(x)|\leq \beta \}$.
We will now reduce $\beta$ further if necessary to ensure that when 
$N_{\mathbbm{r},a}^*$ intersects the boundary of the domain, the domain
is still `opening out' sufficiently fast.
%make sure we intersect the boundary at the right angle. 
More precisely, let $x=(x_1,x') \in \partial \Omega$. 
Write $v_1(x)$ for the vector pointing from $x$ to $(a,0,...0)$, and
$v_2(x)$ the normal vector to $\partial \Omega$ at $x$. We require that the 
angle  
$\theta$ between these two vectors is at least $\pi/2$. 
See Figure~\ref{figure:sawdomex1}(ii) for an illustration.
Some cumbersome computations that we defer to 
Appendix~\ref{cumbersome geometric computation} give that,
\begin{equation}
\langle v_1(x), v_2(x) \rangle =  H + h(-x_1)+h'(-x_1)(-z-x_1-\sqrt{\mathbbm{r}^2-(H+h(z))^2}). \label{point product in boundary}
\end{equation}
Note that if $x_1 = -z$ this gives,
\begin{align*}
 &H+h(z) - h'(z)\sqrt{\mathbbm{r}^2-(H+h(z))^2} \\ 
& < H+h(z) - h'(z)
\sqrt{\frac{(H+h(z))^2\big(1+h'(z)^2)}{(h'(z))^2}-(H+h(z))^2}=0,
% \\ 
%&=   H+h(z) - h'(z)(H+h(z))
%\sqrt{\frac{\cos(\arctan(h'(z))^2}{\sin(\arctan(h'(z))^2}} = 0,
\end{align*}
which gives that $\theta$ is bigger than $\pi/2$. In particular, 
by smoothness of $h$, we can reduce $\beta^*$, so that for all 
$\beta < \beta^*$ we have that, for 
$x=(x_1,x') \in U_\beta \cap \partial \Omega$, 
 \begin{equation}  H + h(-x_1)+
h'(-x_1)(-z-x_1-\sqrt{\mathbbm{r}^2-(H+h(z))^2}) \leq 0. 
\label{condition angle boundary blocking} 
\end{equation}
Equation (\ref{condition angle boundary blocking}) encapsulates 
that the surface hits the boundary of the domain in such a way
that $\theta$ is at least $\pi/2$. %sufficiently large. %\textit{the right angle}.

Let $d(x) = \Vert x - a \Vert - \mathbbm{r}$ be the distance of $x$ from the
circle of radius $\mathbbm{r}$ and centre $a$ in 
$\mathbb{R}^\mathbbm{d}$. We can deduce the following fact: 
there exists a smooth function $f: \Omega \rightarrow \IR$ such that,
\[ d^*_{\mathbbm{r},a}(x) - d(x) = \beta f(x)  \, \, \, \, \, \, 
\forall x \in U_\beta. \]
Indeed, as $d^*_{\mathbbm{r},a}$ is a perturbation of at 
most $\beta$ of $d(x)$ the last equation must hold. 
In particular, for $x \in U_\beta$, there exists a constant $C>0$ 
such that,
\[ \Delta d_{\mathbbm{r},a}^*(x) \leq \Delta d(x) + \beta C = \frac{\mathbbm{d}-1}{\Vert x \Vert} + \beta C. \]
Then, from It\^o's formula, we obtain that for all %$s \leq t \wedge T_\beta$, 
$t \leq T_\beta$,
\begin{align}
d_{\mathbbm{r},a}^*(W_t) \leq d_{\mathbbm{r},a}^*(W_0) &+ 
\int_0^t \nabla d_{\mathbbm{r},a}^*(W_s)\cdot dW_s + 
\int_0^t  \frac{(\mathbbm{d}-1)}{\Vert W_s \Vert} ds + \beta C t 
\nonumber \\ 
&+  \int_0^t \langle \nabla d_{\mathbbm{r},a}^*(W_s), \hat{n} \rangle 
dL^{\partial \Omega}_s(W_s), 
\label{ito for saw domain}
\end{align}
where $\hat{n}$ is the inward pointing normal.
The first two terms of the right hand side of (\ref{ito for saw domain}) 
are the Brownian motion (by L\'evy's characterisation). 
Since $t < T_{\beta}$, the third term on the right hand side 
of~(\ref{ito for saw domain}) can be bounded by the drift term 
on the right hand side of equation~(\ref{coupling for blocking of saw domain}).
To deal with the integral against the local time in~(\ref{ito for saw domain}),
we just need to show that 
$\langle \nabla d_{\mathbbm{r},a}^* (W_s), \hat{n} \rangle \leq 0$. 
Indeed, if the line segment from $W_s$ to $(a,0...,0)$ intersects 
$\partial \Omega$, then $\nabla d_{\mathbbm{r},a}^* (W_s)$ is a vector 
parallel to the surface, in which case the inequality holds trivially. 
On the other hand, if said line segment does not intersect 
$\partial \Omega$, then $\nabla d_{\mathbbm{r},a}^* (W_s)$ is a vector 
going from $W_s$ to $(a,0,...,0)$, in which case 
$\langle \nabla d_{\mathbbm{r},a}^* (W_s), \hat{n} \rangle \leq 0$, 
since~(\ref{condition angle boundary blocking}) holds from our choice 
of $\beta^*$. Combining these bounds 
gives~(\ref{coupling for blocking of saw domain}) and completes the proof.
\end{proof}

From the last result the approach to prove Theorem~\ref{saw-tooth} should 
be obvious. We proceed as in Proposition~\ref{blocking in a cone} by 
taking an initial condition similar to $\widetilde{p}$, but now it will 
be $(1-\gamma_\epsilon)/2$ for $x \in N_{\mathbbm{r},a}^*$. 
By Proposition~\ref{proposition of coupling for blocking in saw domain} we 
can choose $\mathbbm{r}$ and $a$ such 
that~(\ref{coupling for blocking of saw domain}) holds. 
From this the proof of the result is totally analogous to the one used to 
prove Proposition~\ref{blocking in a cone}, giving Theorem~\ref{saw-tooth}.

Invasion under condition (\ref{invasion condition saw  domain}) can then be 
handled in a very similar way to the proof of 
Theorem~\ref{saw-tooth}. 
\begin{proof}[Sketch of proof Theorem~\ref{saw-tooth-invasion}]
If ~(\ref{invasion condition saw  domain}) holds, then, for all $a$, 
we can choose $\mathbbm{r}$ as in~(\ref{election of radius for saw domain}),
so that~(\ref{condition angle boundary blocking}) holds with the reverse 
inequality for points that are close enough to $N_{\mathbbm{r},a}^*$. 
From this and the fact $d_{\mathbbm{r},a}^*(x) \geq d(x)$, we can 
obtain the existence of a Brownian motion $\widehat{B}_s$ and a 
constant $C>0$, such that,
\begin{equation}
 d_{\mathbbm{r},a}^* (W_s) \geq \widehat{B}_s - 
\frac{s(\mathbbm{d}-1)}{\mathbbm{r}-\beta} \geq \widehat{B}_s - 
s \nu - s C, 
\label{equation of coupling for invasion in saw domain}\end{equation}
where the last inequality holds for sufficiently small $\beta$ and uses 
that, in this case, $\mathbbm{r} > \rho^*$. Therefore, if we 
consider $u^\epsilon$i, the solution to $(AC_\epsilon)$ with initial 
condition $u^\epsilon(x,0)=(1-\epsilon)1_{B(a,\mathbbm{r})}$, we can prove 
an analogue of Proposition~\ref{circles get bigger},
using (\ref{equation of coupling for invasion in saw domain}) in place
of~(\ref{coupling for first invasion}). Iterating, we can then deduce an 
analogue of Theorem~\ref{no_blocking} 
if~(\ref{invasion condition saw  domain}) holds, 
thus giving the invasion result.
\end{proof}

\section{Stochastics}\label{sec:stochastic}

We now turn to the SLFVS. As we have seen in the deterministic setting, the
crucial step in determining whether or not there will be blocking
is to understand the interplay between the selection against
heterozygosity and the selective advantage of $aa$-homozygotes over 
$AA$-homzogotes in a small neighbourhood of a critical radius in which the 
influence of the boundary of the domain is unimportant. In this section we
therefore focus on the SLFVS on the whole of Euclidean space, relegating a
discussion of other domains (and, in particular, a suitable definition of 
the SLFVS on domains with boundary) to 
Appendix~\ref{SLFV with reflecting bdry}.

\subsection{The dual process} \label{sec:defdual}

At the core of the proof of Theorem 1.15 was the duality between the 
deterministic equation~$(AC_\varepsilon)$ and ternary branching 
(reflected) Brownian motion endowed with a voting mechanism.
In an entirely analogous way, we wish to exploit the duality between the SLFVS 
and a system of branching and coalescing lineages endowed with a 
similar voting mechanism.

The process of branching and coalescing lineages is driven by 
(the time-reversal of) the Poisson Point Process of events that determined
the dynamics in the SLFVS. 
Recall that $\Pi^n$ is a Poisson point process on 
$\mathbb{R}_+ \times \mathbb{R}^{\mathbbm{d}} \times (0,\mathcal{R}_n)$, 
whose intensity is given by~(\ref{eq:slfvs_intensity_intro_weak}) in the 
weak noise/selection ratio regime,
and by~(\ref{eq:slfvs_intensity_intro_strong}) in the strong noise/selection ratio regime. 
To emphasize that the dual process runs `backwards in time', we
shall write $\overleftarrow{\Pi}^n$ for the time-reversal of 
$\Pi^n$, which of course has the same intensity as $\Pi^n$. 
The impact, asymmetry and selection parameters 
$(u_n)_{n \in\mathbb{N}}$, $(\gamma_n)_{n \in\mathbb{N}}$, 
$(\v{s}_n)_{n \in\mathbb{N}}$ are given in~(\ref{scalings_weak}) in the 
weak noise/selection ratio regime, or fulfil the 
condition~(\ref{scaling_strong_sn}) in the 
strong noise/selection ratio regime.

\begin{definition} (SLFVS dual) 
\label{SLFVS dual}
For $n \in \mathbb{N}$, 
the process $(\mathcal{P}_t^n)_{t \geq 0}$ is the 
$\bigcup_{l \geq 1} (\mathbb{R}^\dim)^l$-valued Markov process with 
dynamics defined as follows.

The process starts from a single individual at the point $x\in\IR^\dim$.
We write $\mathcal P^n_t = (\xi^n_1(t),\ldots , \xi^n_{N(t)}(t))$, 
where the random number $N(t)\in \mathbb{N}$ is 
the number of individuals alive at time $t$, and $\{\xi^n_i(t)\}_{i=1}^{N(t)}$
are their locations. 
For each $(t,x,r)\in \overleftarrow{\Pi}^n$, 
the corresponding 
event is neutral with probability $1-(1+\gamma_n)\v{s}_n$, in which case: 
\begin{enumerate}
\item for each $\xi_i^n(t-)\in B(x,r)$, independently mark 
the corresponding  individual with probability $u_n$;
\item if at least one individual is marked, all marked individuals 
coalesce into a single offspring, whose location is 
chosen uniformly in $B(x,r)$.
\end{enumerate}
With the complementary probability $(1+\gamma_n)\v{s}_n$, 
the event is selective, 
in which case:
\begin{enumerate}
\item for each $\xi^n_i(t-)\in B(x,r)$, independently mark the 
corresponding individual with probability $u_n$;
\item if at least one individual is marked, all of the marked 
individuals are replaced by a total of 
{\em three} offspring, whose locations are 
drawn independently and uniformly in $B(x,r)$.
\end{enumerate}
In both cases, if no individual is marked, then nothing happens.
\end{definition}
\begin{remark}
From the perspective of the SLFVS, it would be more natural to call the
individuals created during a reproduction event in the dual process 
`parents' (or `potential parents'), as they are situated at the locations
from which the parental alleles are sampled. We choose to call them 
offspring in order to emphasize that the dual process plays the same 
role as ternary branching Brownian motion in the deterministic setting,
and, indeed, much of the proof of Theorem~\ref{teo:decayv2} carries over
with minimal changes to the SLFVS setting.
\end{remark}
Just as for the deterministic setting, the duality relation that we exploit 
is between the SLFVS and the {\em historical process} of branching and 
coalescing lineages,  
$$\Xi^n(t):= (\mathcal P^n_s)_{0\leq s\leq t}.$$
We write $\PP_x$ for the law of $\Xi^n$ when $\mathcal P^n_0$ is  
the single point $x$, and $\EE_x$ for the corresponding expectation. 

Just as for the branching process, we can use Ulam-Harris labels to define
lines of descent from the root individual at $x$ through $\mathcal{P}^n$. 
More precisely, 
for $\v{i}=(i_1,i_2,\ldots )\in \{1,2,3\}^{\IN}$, we 
write $(\xi^n_{\v{i}}(\cdot))_{0\leq s\leq t}\subseteq\Xi^n(t)$ for 
the $\RR^\mathbbm{d}$-valued path which jumps to the 
location of the (unique) offspring when the individual in $\mathcal{P}^n_s$ at 
its location is affected by a neutral event, and to the location of 
the $i_k$th offspring, the $k$th time that it is affected 
by a selective event. From the perspective of the SLFVS, $\xi_{\v{i}}^n$ is 
an ancestral {\em lineage}, and we shall use the terminology lineage below.

Let $p:\RR^\mathbbm{d} \to [0,1]$.
The voting procedure on $\Xi^n(t)$ is a natural modification of the one
that we defined for the ternary branching Brownian motion: 
%of Definition~\ref{def:vp}. 
%be a 
%fixed function. Recalling that the set of individuals in $\mathcal{P}^n_t$ is $\{\xi^n_1(t),\ldots , \xi^n_{N(t)}(t)\}$, for each $j\leq N(t)$, the individual $\xi_j^n(t)$ votes $1$ with probability $p(\xi_j^n(t))$ and otherwise votes $0$; votes from different individuals are independent. As we trace backwards in time through $\Xi(t)$,
\begin{enumerate}
\item Each leaf of $\Xi^n(t)$ independently votes $1$ with 
probability $p(\xi_i(t))$, and $0$ otherwise;
\item at each neutral event in $\overleftarrow{\Pi}^n$, 
all marked individuals adopt 
the vote of the offspring;
\item at each selective event in $\overleftarrow{\Pi}^n$, 
all marked individuals adopt the majority vote of the three offspring,
unless precisely one vote is $1$, in which case they all vote $1$ 
with probability $\frac{2 \gamma_n}{3+3\gamma_n}$, otherwise they vote $0$.
\end{enumerate}
This defines an iterative voting procedure, which runs inwards 
from the `leaves' of $\Xi^n(t)$ to the ancestral individual $\emptyset$ 
situated at the point $x$. 
\begin{definition}\label{vpdsv}
With the voting procedure described above, we define $ \mathbb{V} _p(\Xi^n(t)) $ to be the vote associated to the root $\emptyset$. 
\end{definition}

We should like to have an analogue of the stochastic representation of
the solution to~$(AC_\epsilon)$ of Proposition~\ref{SREP} for the SLFVS, but 
recall from Remark~\ref{slfvs only lebesgue ae}
that the SLFVS is only defined up to a Lebesgue null set, and so we cannot 
expect such a representation at every point of $\IR^\dim$. However, a
weak version of the representation is valid. The following result is 
easily
proved using the approach introduced to prove Proposition 1.7
in~\cite{etheridge/veber/yu:2020} 
(the corresponding result in the case
of genic selection).

\begin{theorem} 
\label{teo:dual}
The SLFVS driven by $\Pi^n$, $(w^n_t(x), x\in\RR^\mathbbm{d})_{t\geq 0}$, 
is dual to the process $(\Xi^n(t))_{t\geq 0}$ of Definition~\ref{SLFVS dual}
in the sense that for every 
$\psi\in C(\mathbb{R}^\dd)\cap L^1(\mathbb{R}^\dd)$, we have
\begin{equation}
\EE_{p}\left[\int_{\Omega}  \psi(x)w^n_t(x)\, dx \right] 
= \int_{\Omega} \psi(x)\EE_x\left[ \VV_{p}\left( \Xi^n(t)\right)\right]\, dx
= \int_{\Omega} \psi(x)\PP_x\left[ \VV_{p}\big( \Xi^n(t)\big)=1\right]\, dx.
\label{dual formula}
\end{equation}
\end{theorem}
\begin{remark}
Note that the expectations on the left and right of 
Equation~(\ref{dual formula}) are
taken with respect to different measures. The subscripts on the expectations
are the initial values for the processes on each side.
\end{remark}
The duality reduces the proof of Theorem~\ref{noisy circles} to the 
following analogue of Theorem~\ref{teo:decayv2}. 

\begin{theorem}
\label{noisy circles:dual version}
Define $\rho_* = (\dd-1)/\nu$ and let $\widehat{p}(x) = 1_{B(0,\rho_*)}(x)$.
\begin{enumerate}
\item Under the weak noise/selection ratio regime, 
for any $k\in\mathbb{N}$, there exist $n_*(k)<\infty$, 
and $d_*(k)\in(0,\infty)$, such that for all $n\geq n_*$ and all $t > 0$, 
\begin{align*}
 &\text{for all } x \in \mathbb{R}^\dd \text{ with } \Vert x \Vert \geq \rho_* + d_* \epsilon_n |\log \epsilon_n|, \text{ we have } \mathbb{P}_x\left[\VV_{\widehat{p}} (\Xi ^n (t))=1 \right]\leq \epsilon_n^k. \\
 &\text{for all } x \in \mathbb{R}^\dd \text{ with } \Vert x \Vert \leq \rho_* -d_* \epsilon_n |\log \epsilon_n|, \text{ we have } \mathbb{P}_x\left[\VV_{\widehat{p}} (\Xi ^n (t))=1 \right]\geq 1- \epsilon_n^k.
\end{align*}
\item Under the strong noise/selection ratio regime, there is a constant
$\sigma^2>0$ such that for every fixed $t>0$
and $\epsilon>0$
%as $n$ goes to infinity, the dual process is approximately a Brownian motion. That is, 
there is a Brownian motion, 
$(W_s)_{s \geq 0}$, 
and $n_*\in \IN$, such that for all $n \geq n_*$ 
\begin{equation} \label{eq:bd2}
\Big|\mathbb{P}_x\big[\VV_{\widehat{p}} (\Xi ^n (t))=1 \big]  - 
\PP_x\big[\Vert W(\sigma^2 t) \Vert \leq \rho_*\big] \Big| \leq \epsilon.
\end{equation}
 \end{enumerate}
\end{theorem}

\subsubsection*{Motion of a single lineage}

Our first task is to investigate 
the motion of a single lineage. 
%Before proceeding into the details of Theorem \ref{noisy circles:dual version} we write explicitly the motion of an ancestral lineage in both regimes and argue why it should be close to a Brownian motion. 
It is a pure jump process.
By spatial homogeneity, to establish the transition rates, 
it suffices to 
calculate the rate at which a lineage currently at the origin will jump to the 
point $z\in\IR^\dim$. This is given by
\begin{align} 
\label{R^d_rateancestry} 
m_n(dz) &= n
u_n \chi_n n^{\mathbbm{d}\beta} \int_0^{\mathcal{R}_n} \frac{V_r(0,z)}{V_r} d\mu^n(r) dz, \end{align} 
where $\chi_n = 1/\widehat{u}_n$ in the strong noise regime and 
$1$ otherwise, $V_r$ is the volume of $B(0,r)$,
and $V_r(0,z)$ is the volume of $B(0,r) \cap B(z,r)$. To understand the expression~(\ref{R^d_rateancestry}), first 
observe that in order for the lineage to jump from $0$ to $z$, 
it must be affected by an event 
that covers both $0$ and $z$. The possible centres for an event of
radius $r$ covering both $0$ and $z$ is the region given by 
the intersection of $B(0,r)$ and $B(z,r)$, which has volume $V_r(0,z)$, and 
so events fall in this region with rate $n\chi_n V_r(0,z)$.
If an event occurs 
with centre $x$ in the region, then the lineage jumps with 
probability $u_n$, and if it does so, then it jumps to
a point chosen uniformly from $B(x,r)$; the chance of 
that point being $z$ is $dz/V_r$. 
\begin{remark}
\label{stochastic scaling}
We can now explain why our choice of scalings in the weak and strong 
regimes are appropriate for investigating the effect of increasing 
genetic drift.
Notice that $n u_n\chi_n= u n^{2\beta}$ in the weak scaling and
$n^{2\beta}$ in the strong scaling, so setting $u=1$ 
in~(\ref{scalings_weak}) the transition 
probabilities for a lineage 
will coincide under our two regimes. 
If the coefficients $\v{s}_n$ are the same, then so too will be the rate at
which a lineage is affected by a selective event. From the perspective of
lineages, the difference between the scalings is the probability of
coalescence, driven by
$u_n$. The parameter $u_n$ determines the strength of the genetic drift 
(it can be thought of as proportional to
the inverse of the population density). 
In particular, in two dimensions, Theorem~\ref{noisy circles}
says that increasing the strength of the noise can indeed
break down the structure of the solution that results from the selection
term in the Allen-Cahn equation.
\end{remark}
Integrating out over $z$ in (\ref{R^d_rateancestry}) and `undoing' the scaling of $\mu^n(dr)$, we find that
the total jump rate of the lineage is
\begin{align*}
\int_{\mathbb{R}^\dd} m_n(dz) &= n u_n \chi_n V_1 \int_0^\mathcal{R} r^d \mu(dr),
\end{align*}
where, as noted above, the prefactor is (up to a factor of $u$ coming
from~(\ref{scalings_weak})) $n^{2\beta}$.
Since each jump is of size $\Theta(n^{-\beta})$, we recognise the diffusive
scaling. We can identify the diffusion constant for the limiting 
Brownian motion, by noting that
\begin{align}
\label{diffusion constant}
\sigma^2=
\int_{\RR^\dd} \Vert z \Vert^2 m_n (dz) = \frac{u_n n \chi_n}{n^{2\beta} 2 \dd} \int_0^\mathcal{R} \int_{\RR^\dd} \Vert z \Vert^2 \frac{V_r(0,z)}{V_r(0)} dz \mu(dr).
\end{align}
The following coupling is Lemma~3.8 in~\cite{etheridge/freeman/penington:2017}.
\begin{lemma} 
\label{lemma:coupRd}
Let $(\xi^n(t))$ be a pure jump process with transition rates given by $m_n$ 
and suppose that $\sigma^2$ is defined by~(\ref{diffusion constant}). 
For fixed $t>0$ there is a coupling of a 
Brownian motion $W$ and $\xi^n$ under which
\[ \PP_x[|\xi^n(t)-W(\sigma^2 t)| \geq n^{-\beta/6}] = \mathcal{O}(n^{-\beta}(t \vee 1)) \]
(where $\xi^n$ and $W$ both start from the point $x$).
\end{lemma}
We need to be able to couple the lineage at the time when it is first affected
by a selective event (and so branches)
with a Brownian motion. This is where we first use 
Assumption~\ref{cond on epsilon n},
$(\log n)^{1/2}\epsilon_n\to \infty$, which guarantees that 
$\v{s}_n=o\big(\log n/n^{2\beta}\big)$.
We abuse our Ulam-Harris based notation $\xi_{\v{i}}$ and write
simply $\xi_1(\tau)$, $\xi_2(\tau)$, and $\xi_3(\tau)$, for the 
three possible positions
of lineages at the first branching time of $\Xi^n$ (corresponding
to the positions of the three offspring of the first selective event 
to affect the lineage).

\begin{corollary}\label{cor:coupRdtime}
Let $\tau$ be the first branching of $\Xi^n$. 
Then there is a Brownian motion $W$ and a coupling of 
$\Xi^n$ and $W$ under which the following holds. 
The branching time $\tau$ and $W$ are independent, 
$\tau \sim \mathtt{Exp(}\eta (1+\gamma_n) \epsilon_n^{-2})$ with $\eta = u V_1 \int_0^\mathcal{R} r^\dd \mu(dr) $, and for $i=1,2,3$,
\[ \PP_x\left[|\xi_i^n(\tau)-W(\sigma^2 \tau)| \geq n^{-\beta/6} \right] = \mathcal{O}(n^{-\beta}). \]
where $\xi^n$ and $W$ both have the same starting point $x$.
\end{corollary}
\begin{proof}[Sketch of proof]

By Poisson thinning, we can express the position of the lineage at 
time $t$ as the sum of the jumps due to neutral events and those due
to selective events. Lemma~\ref{lemma:coupRd}
allows us to couple the part corresponding to neutral events with a 
Brownian motion at time $\sigma^2(1-\v{s}_n)t$. Since 
$\v{s}_n=o\big(\log n/n^{2\beta}\big)$, Chebyshev's inequality gives
$$\IP\left[\|W(\sigma^2t)-W(\sigma^2(1-\v{s}_n)t)\|\geq n^{-\beta/6}\right]
=o\left(\frac{\log n}{n^{2\beta}}n^{\beta/3}(t\vee 1)\right).$$
The proof is completed by an application of the triangle inequality,
using Poisson thinning to partition over the value of $\tau$, and using that 
$|\xi_i^n(\tau)-\xi_1^n(\tau-)|\leq 2 {\mathcal R}_n =2n^{-\beta}{\mathcal R}$.
\end{proof}

\subsection{Proof of Theorem~\ref{noisy circles:dual version}, weak regime} 
\label{sec:teofv_weakregime}

We outline the main steps of the proof of 
Theorem~\ref{noisy circles:dual version} in the weak scaling regime, that is
one in which selection overwhelms noise. 

There are two ways in which the dual to the SLFVS differs from our ternary
branching Brownian motion. The first is that lineages can coalesce
during reproduction events; the second is that lineages follow a
continuous time and space random walk which only converges to Brownian
motion in the scaling limit.

Following~\cite{etheridge/freeman/penington:2017},
the first step in the proof of the limiting result in the weak scaling
regime is to show that with high probability $\Xi^n(t)$ can be coupled
to a branching jump process.

\begin{definition}[Branching Jump Process] 
For a given $n \in \mathbb{N}$ and starting point $x \in \IR^\dim$, 
$(\Psi^n(t), t \geq 0)$ is the historical process of the branching 
random walk described as follows.
\begin{enumerate}
\item Each individual has an independent lifetime, which is
exponentially distributed with parameter $\eta (1+\gamma_n) \epsilon_n^{-2}$.
\item During its lifetime, each individual, independently, evolves 
according to a pure jump process with jump rates given by 
$(1-(1+\gamma_n)\v{s}_n)m_n$.
\item At the end of its lifetime an individual branches into three offspring. 
The locations of these offspring are determined as follows. First choose 
$r \in (0,\mathcal{R}_n]$ according to 
$$\frac{r^\mathbbm{d} \mu^n(dr)}{\int_0^{\mathcal{R}_n} 
\widetilde{r}^\mathbbm{d} \mu^n(d\widetilde{r})}.$$ 
If the individual is at point $z$ then the location of each offspring is  
sampled independently and uniformly from $B(z,r)$.
\end{enumerate}
\end{definition}
The process $\Psi^n(t)$ differs from $\Xi^n(t)$ only in that we have 
suppressed the coalescence events. Since, in this weak noise regime,
coalescence events are extremely rare, we can couple the two processes in such 
a way that they coincide with high probability.
The following is Lemma~3.12 of \cite{etheridge/freeman/penington:2017}.
\begin{lemma} \label{coupbranch}
Let $T^* \in (0,\infty)$, $k \in \mathbb{N}$ and $z \in \RR^\dd$. 
There exists $n_* \in \mathbb{N}$ such that, for all $n \geq n_*$, 
there is a coupling of $\Xi^n$ and $\Psi^n$, both started with one 
particle at $z$ such that, with probability at least $1-\epsilon_n^k$ 
we have:
\[ \Xi^n(T^*) = \Psi^n(T^*). \]
\end{lemma}
\begin{proof}[Sketch of proof]
The key idea is to modify the dual process of 
Definition~\ref{SLFVS dual}
in such a way that individuals are `preemptively' marked. More precisely,
at time zero each individual is marked with probability $u_n$. When 
an individual is in the region covered by a reproduction event, it will be
affected only if it is marked. After the event, the marks of individuals 
within the region covered are removed and new marks are assigned
(including to the offspring, if any) independently with probability $u_n$. 
Once marked, individuals remain marked until they are covered by an event. 
Unless both are marked, any two lineages evolve independently. 
The probability that two individuals are marked at time zero is $u_n^2$.
We must control the probability that for a pair of 
`root to leaf rays'
in $\Xi^n(T^*)$ a reproduction event occurs during $[0,T^*]$
after which both are marked. If they were not both marked at time 
zero, then in order for both to be marked, one of them must be 
affected by an event, after which the probability that they are both
marked is $u_n^2$. Since order $nT^*$ events affect a lineage over $[0,T^*]$,
for any pair of root to leaf rays, the probability that there is a 
reproduction event after which both are marked is 
${\mathcal O}(nu_n^2)={\mathcal O}(n^{4\beta-1})$.
(This is why we restrict to $\beta<1/4$ in the weak noise/selection
regime.)
   
We then use 
Assumption~\ref{cond on epsilon n} for the second time. This time, 
it allows one to control
the total number of such root to leaf rays in $\Xi^n$ (by the number
in a regular ternary tree of height $b\log n$ for a suitable 
constant $b$). A union bound over
pairs of such rays shows that the probability that there is any
time in $[0,T^*]$ when at least two rays are marked (which is 
required for a coalescence event to take place) is 
$\mathcal{O}(n^{-\alpha})$ for some $\alpha >0$.
\end{proof}

Since we have already checked that the motion of a lineage in $\Xi^n$, and 
therefore in $\Psi^n$, is close to a 
Brownian motion at the first branch time, we can already see why 
Theorem~\ref{noisy circles:dual version}
should hold in the weak regime. Of course, there is still some work to do;
as $\epsilon_n\to 0$ the number of branches in $\Psi^n$ is very large, and 
it is not obvious that the convergence to Brownian motion along a single 
lineage will translate into sufficiently rapid convergence on the whole tree.
The proof follows the same pattern as the deterministic case.
 
%As we have suppressed the coalescence events, the genealogical structure of $\Psi^n(t)$ is a ternary branching tree. Hence, we can define $\mathbb{V}_p(\Psi^n(t))$ as we did in Definition \ref{def:vp}. That is, leaves vote independently, voting $1$ with probability $p(\Psi_i(t))$ if they are at location $\Psi_i(t)$, and zero otherwise. The vote is worked towards the root, an internal vertex votes the majority of the offspring unless precisely one vote $1$, in which case the vertex votes $1$ with probability $\frac{2 \gamma_n}{3+3\gamma_n}$. The vote of the root is then denoted by $\mathbb{V}_p(\Psi^n(t))$. By Corollary \ref{cor:coupRdtime} we know that the ancestry lines of $\Psi^n(t)$ are all, approximately, Brownian motion between branching events.

\subsubsection*{Generation of interface for the SLFVS}

The first step is to show that, analogously to 
Proposition~\ref{prop:interface}, the SLFVS generates an interface in a 
time window that is of order $\epsilon_n^2 |\log(\epsilon_n)|$. 
Evidently it suffices to work with $\Psi^n$.
\begin{proposition} 
\label{interfacefv}
Let $k \in \mathbb{N}$. Then there exists 
$n_*(k),a_*(k),b_*(k)>0$ such that, for all $n \geq n_*$, if we set:
\begin{equation} \label{defdelta} 
\delta_*(k,n) := a(k) \epsilon_n^2 |\log(\epsilon_n)| \: \: 
\text{and} \: \: 
\delta_*'(k,n) := (a(k)+\eta^{-1}(k+1))\epsilon_n^2 |\log(\epsilon_n)|, 
\end{equation}
then, for $t \in [\delta,\delta']$, we have that, 
\begin{align*}
 &\text{for any } x \text{ such that } \Vert x \Vert \geq \rho_* + 
d_* \epsilon_n |\log \epsilon_n|, \text{ we have } 
\PP_x^\epsilon[ \mathbb{V}_{\widehat{p}}^\gamma(\Psi^n(t))=1 ]
\leq \epsilon_n^k. \\
 &\text{for any } x \text{ such that } \Vert x \Vert \leq \rho_* 
-d_* \epsilon_n |\log \epsilon_n|, \text{ we have } 
\PP_x^\epsilon[\mathbb{V}_{\widehat{p}}^\gamma(\Psi^n(t))=1 ]
\geq 1- \epsilon_n^k.
 \end{align*}
\end{proposition}
%By Lemma \ref{coupbranch} it is enough to prove the result for $\Psi^n(t)$ instead of $\Xi^n(t)$. However, we still need to control the number of branches in $\Psi^n(t)$. This is given by the next result.
Our proof in the deterministic setting required that we could find a large 
ternary tree sitting within $\mathcal{T}(\v{W})$. Here we also 
need the converse to prove an analogue of~(\ref{matrak:pm1}).
\begin{lemma}[\cite{etheridge/freeman/penington:2017}, Lemma~3.16] 
\label{contbranchpsi}
Let $k \in \mathbb{N}$ and let $A(k)$ be as in Lemma \ref{lemma:biasextg}. There exists $a_*(k),B_*(k) \in (0,\infty)$ and $n_*(k) < \infty$ such that for all $n \geq n_*$ and $\delta_*,\delta_*'$ as defined in (\ref{defdelta}), 
\begin{align}
\label{contbranchpsi1} 
\PP\left[ \mathcal{T}(\Psi^n(\delta_*)) \supseteq \mathcal{T}_{A(k)|\log(\epsilon_n)|}^{reg}\right] &\geq 1 -\epsilon_n^k, \\
\label{contbranchpsi2}  
\PP\left[ \mathcal{T}(\Psi^n(\delta_*')) \subseteq \mathcal{T}_{B(k)|\log(\epsilon_n)|}^{reg}\right] &\geq 1 - \epsilon_n^k.
\end{align}
\end{lemma}
\begin{proof}[Sketch of Proof of Proposition~\ref{interfacefv}]
%We prove the results for $\Psi^n$ instead of $\Xi^n$. The results then follow by Lemma \ref{coupbranch}. 
The proof is based on the proof of Proposition~\ref{prop:interface}, 
where the new ingredient is that now we use Lemma~\ref{lemma:coupRd}
to control the 
distance between the jump process with transitions given by $m_n$
(which governs the lineage motion)
and a Brownian motion, over the time interval $[\delta_*, \delta_*']$.
Since $\epsilon_n^{-2}=o(\log n)$, for any constant $d$ we can arrange
that for large enough $n$, $d\epsilon_n|\log(\epsilon_n)|\geq 2n^{-\beta/6}$.
Combining with our previous estimates for the Brownian motion (Proposition~\ref{prop:maxmov}), 
for any root to leaf ray in $\Psi^n(\delta_*')$, we can use this to 
control the probability that the leaf is more than $\frac{1}{2} d_*
\epsilon_n|\log(\epsilon_n)|$ from its starting point at any time in 
$[\delta_*, \delta_*']$.
We can extend this to the whole of $\Psi^n$ using 
Equation~(\ref{contbranchpsi2}) and a union bound. 

The proof of the first inequality of
Proposition~\ref{interfacefv} now mirrors that of 
Proposition~\ref{prop:interface}.
A symmetric computation gives the second inequality. 
\end{proof}

\subsubsection*{Final steps for the weak noise/selection ratio regime 
of Theorem~\ref{noisy circles:dual version}}

To conclude the proof of the weak noise/selection ratio 
regime of Theorem~\ref{noisy circles:dual version}. 
We need the following modification of Lemma~\ref{Lemma:pushcomp}. 
For simplicity, we write $\Vert x \Vert_{\rho_*} := \Vert x \Vert - \rho_*$.
\begin{lemma} \label{Lemma:pushcomp2}
Let $l \in \mathbb{N}$ with $l \geq 4$, $K_1>0$. There exists $n_*$ 
such that for all $n \geq n_*$, $x \in \RR^\dd$, 
$s \in [\sigma^2 \epsilon_n^{l+3},\sigma^2(l+1)\eta^{-1}\epsilon_n^2 |\log(\epsilon_n)|]$ and $t \in [s,\infty)$,
\begin{align}
& \EE_x\left[g\left(\PP_{\Vert W_s \Vert_{\rho_*}-\nu(t-s)+K_1 \epsilon_n|\log(\epsilon_n)|+3n^{-\beta/6}}[\mathbb{V}^\gamma(\mathbf{B}(t-s)=1]+\epsilon_n^l\right)\right]  \nonumber \\ & \leq \frac{3 + 5 \gamma_n}{4(1+\gamma_n)}\epsilon_n^l + \EE_{\Vert x \Vert_{\rho_*}}\left[g\left(\PP_{B_s-\nu t+K_1 \epsilon_n |\log(\epsilon_n)|}[\mathbb{V}^\gamma(\mathbf{B}(t-s))=1]\right)\right] + 1_{s \leq \epsilon_n^3} \epsilon_n^l. \\
& \EE_x\left[g\left(\PP_{\Vert W_s \Vert_{\rho_*}-\nu(t-s)-K_1 \epsilon_n|\log(\epsilon_n)|+3n^{-\beta/6}}[\mathbb{V}^\gamma(\mathbf{B}(t-s)=0]+\epsilon_n^l\right)\right]  \nonumber \\ & \leq \frac{3 + 5 \gamma_n}{4(1+\gamma_n)}\epsilon_n^l + \EE_{\Vert x \Vert_{\rho_*}}\left[g\left(\PP_{B_s-\nu t-K_1 \epsilon_n |\log(\epsilon_n)|}[\mathbb{V}^\gamma(\mathbf{B}(t-s))=0]\right)\right] + 1_{s \leq \epsilon_n^3} \epsilon_n^l.
\end{align}
\end{lemma}
\begin{proof}
The proof is identical to that of Lemma~\ref{Lemma:pushcomp} except that
we also approximate the jump process of a lineage by Brownian motion. 
Since $n^{-\beta/6} = o(s \epsilon_n |\log(\epsilon_n)|)$, the additional
term $3n^{-\beta/6}$ is negligible for large $n$.
For the case in which $\Vert W_s \Vert_{\rho_*}$ is close to $0$, 
we use Lemma~\ref{teo:coup}, where the distance is now to a 
sphere of fixed radius $\rho_*$ (but the statement of the result 
does not change).
\end{proof}
The equivalent of Proposition~\ref{prop:ineqonetwo} for $\Psi^n$ is:
\begin{proposition} \label{prop:ineqonetwo2}
Let $l \in \mathbb{N}$ with $l \geq 4$. Let $a_*(l)$ and $\delta_*(l,\epsilon)$ given by Proposition \ref{interfacefv}. There exists $K_1(l)$ and $n_*(l,K_1) > 0$ such that, for all $n \geq n_*$, $t \in [\delta_*(l,n),\infty)$ we have:
\begin{align*}
 & \sup_{x \in \RR^\dd} \left(\PP_x^\epsilon[\mathbb{V}_{\widehat{p}}^\gamma(\Psi^n(t))=1]- \PP^{\epsilon_n}_{\Vert x \Vert_{\rho_*} - \nu t+ K_1 \epsilon_n |\log(\epsilon_n)|} [\mathbb{V}^\gamma(\mathbf{B}(t)) = 1] \right) \leq \epsilon_n^l. \\ & \sup_{x \in \RR^\dd} \left(\PP_x^\epsilon[\mathbb{V}_{\widehat{p}}^\gamma(\Psi^n(t))=0]- \PP^{\epsilon_n}_{\Vert x \Vert_{\rho_*} - \nu t - K_1 \epsilon_n |\log(\epsilon_n)|} [\mathbb{V}^\gamma(\mathbf{B}(t)) = 0] \right) \leq \epsilon_n^l.\end{align*}
\end{proposition}
\begin{proof}
The proof is essentially identical to that of 
Proposition~\ref{prop:ineqonetwo}, except that we use 
Lemma~\ref{Lemma:pushcomp2} in place of Lemma~\ref{Lemma:pushcomp},
and Proposition~\ref{interfacefv} in place of Proposition~\ref{prop:interface}.
\end{proof}
\begin{proof}[Proof of Theorem~\ref{noisy circles:dual version}, weak noise regime]
As in the deterministic setting,
it suffices to prove the result for sufficiently large $k \in \mathbb{N}$.
%, and in particular we will show it for $k \geq 5$. 
By Lemma~\ref{coupbranch}, it suffices to work with $\Psi^n$, and 
%for sufficiently large $n$ and $t \in [0,\infty)$ we have
%\[ 
%\left|\PP_x[\VV_{\widehat{p}}^\gamma(\Psi^n(t))=1] 
%- \PP_x[\VV_{\widehat{p}}^\gamma(\Xi^n(t))=1]\right| \leq \epsilon_n^{k+1}. 
%\]
the result then follows from Proposition~\ref{prop:ineqonetwo2}, 
in the same vein as the proof of Theorem~\ref{teo:decayv2}.
\end{proof}

\subsection{Proof of Theorem \ref{noisy circles:dual version}, strong regime} 
\label{sec:teofv_strongregime}

We now turn to the strong noise/selection ratio regime. 

The total rate at which a lineage is affected by a selective event, and so
a new particle is created in the dual, is proportional to:
\[ 
\frac{n}{\widehat{u}_n} u_n \v{s}_n = \v{s}_n n^{2\beta}. 
\]
The first case of (\ref{scaling_strong_sn}) corresponds to not seeing 
any creation of particles in the dual process in the limit, and we include
it only for completeness.
 
The second condition in~(\ref{scaling_strong_sn}) is more interesting, 
and more complex. In this case, the parameter $u_n$ is sufficiently
large relative to the rate of creation of lineages that even though, 
asymptotically, new lineages may be created infinitely fast in the dual, they
are effectively instantly annulled by coalescence. 
This is, for example, the case if $u_n=(\log n)^{-1/2}$ and 
$\v{s}_n = (\log n)^{1/3}/n^{2\beta}$.

We remark that this cannot be achieved in $\dim\geq 3$; the distance
between two lineages in the dual is itself a homogeneous
jump process with bounded
jump size, which will be transient in $\dim\geq 3$, so that there is always
a positive probability of the two lineages `escaping' from one another and
never coalescing. In $\dim=2$, even if two lineages initially move
apart, there will, with probability one, be a later time at which they
are close
enough together to be covered by the same event, and therefore have a 
chance to coalesce.
Indeed, in the strong noise/selection ratio regime,
most lineages will coalesce with their `siblings' very soon after 
being created, with the result that, in a way that we shall make precise
below, the `effective' rate of creation
of new lineages in the dual is of order
\[ 
\frac{n}{\widehat{u}_n} \frac{\v{s}_n u_n}{u_n \log n} = 
\frac{\v{s}_n n^{2\beta}}{u_n \log n}.
\]

Here is a rigorous statement of the key result that we must prove.
\begin{lemma} \label{singleal}
Let 
$\big(\mathcal{P}^n(s)\big)_{s\geq 0}:= 
\big(\xi_1^n(s),...,\xi_{N(s)}^n (s)\big)_{s\geq 0}$ 
be the dual process to the SLFVS introduced in Definition~\ref{SLFVS dual}.
Let $x \in \RR^\dd$ and $t >0$ be fixed. 
In the strong noise/selection 
ratio regime there is a 
Brownian motion $(W_s)_{s \geq 0}$ and a coupling of 
$\mathcal{P}^n$ 
and $W$ such that, for any $\epsilon >0$, there is $n_*$ 
such that for all $n \geq n_*$;
\begin{equation}
\PP_x\left[ \big\{N(t) = 1\big\} %\wedge 
\cap \big\{\|\xi_1^n(t) - W(\sigma^2 t)\| < n^{-\beta/6}\big\}\right] \geq 1 - \epsilon.    
\end{equation}
where $\mathcal{P}^n$ 
starts from a single particle at $x$, which is also the starting point of $W$.
\end{lemma}

Armed with Lemma~\ref{singleal}, the strong noise/ selection 
regime of Theorem~\ref{noisy circles:dual version} follows easily.
\begin{proof}[Proof of Theorem \ref{noisy circles:dual version}, strong noise/selection regime]
From Lemma~\ref{singleal}, there is a Brownian motion $W$ such that, for $n$ 
sufficiently large,
\begin{align*}
&\left|\mathbb{P}_x\left[\VV_{\widehat{p}} (\Xi ^n (t))=1 \right]   - \PP_x[\Vert W(\sigma^2 t)\Vert \leq \rho_*] \right| \
\\ &\leq \left|\PP_x\big[\|\xi_1^n(t) - W(\sigma^2 t)\|\leq n^{-\beta/6}, 
\Vert \xi_1^n(t) \Vert \leq \rho_*\big]  - \PP_x[\Vert W(\sigma^2 t) \Vert \leq \rho_*] \right| + \epsilon
\end{align*}
Note that
\begin{align}
 &\PP_x[\Vert W(\sigma^2 t) \Vert \leq \rho_*] - 
\PP_x\Big[|\Vert W(\sigma^2 t) \Vert - \rho_*| \leq n^{-\beta/6}\Big] 
\nonumber \\ 
&\leq \PP_x\big[|\xi_1^{n}(t)-W(\sigma^2 t)| \leq n^{-\beta/6}, 
\Vert \xi^n_1(t) \Vert \leq \rho_*\big] \nonumber \\ 
&\leq \PP_x\big[\Vert W(\sigma^2 t) \Vert \leq \rho_* \big] + 
 \PP\Big[|\Vert W(\sigma^2 t) \Vert - \rho_*| \leq n^{-\beta/6}\Big], 
\label{controlreqA}
\end{align}
To conclude, observe that
\begin{equation} 
\PP_x[|\Vert W(\sigma^2 t) \Vert - \rho_*| \leq n^{-\beta/6}] 
\leq C(\dd, \rho_*, x) n^{-\beta/6}
\leq \epsilon, \label{ineq:regbound} 
\end{equation}
% \PP_{\Vert x \Vert - \rho_* }
%\Big[|B(t\sigma^2)|\leq n^{-\beta/6}\Big] \nonumber \\ 
%& \leq \frac{\dd n^{-\beta/6}}{\sqrt{2 \pi t \sigma^2}} 
%\end{align}
where %$(B(t))_{t \geq 0}$ is a standard one dimensional Brownian motion, and 
the last inequality is valid for sufficiently large values of 
$n_*$. Using~(\ref{ineq:regbound}) in~(\ref{controlreqA}) gives the result.
\end{proof}
%\begin{remark}
%From our prior computations is not obvious that, as we scale the SLFVS process the covariance of different spatial points do vanish. However, since the ancestry lines are (approximately) Brownian motions that do not meet in $\mathbbm{d}$, it follows that the variance of each point of the SLFVS does decreases with $n$, see Lemma 3.2 of \cite{berestycki/etherdige/veber:2011}.
%\end{remark}

\subsubsection*{Proof of Lemma~\ref{singleal}}

Lemma~\ref{singleal} is a consequence of
Lemma~\ref{lemma:coupRd} and the following proposition.

\begin{proposition} \label{Lemma:boundN}
Let $t>0, \epsilon >0$. Then in the strong noise/selection scaling 
ratio regime there is $n_* \in \mathbb{N}$ such that 
for all $n \geq n_*$
\begin{equation} 
\label{eq:sal} 
\PP_{x}\big[N(t) >1\big] \leq \epsilon, 
\end{equation}
where $N(s)$ is the number of particles alive at time $s$ in
the dual process $(\mathcal{P}_s^n)_{0 \leq s \leq t}$ of 
Lemma~\ref{SLFVS dual}.
\end{proposition}
We focus on the case $u_n \log n \rightarrow \infty$ and 
$\v{s}_n n^{2 \beta}/(u_n \log n) \rightarrow 0$ 
(note that $\liminf u_n \log n$ diverging implies that 
$\lim u_n \log n \to \infty$, and so we will use the latter 
condition in what follows).
The proof is heavily inspired 
by~\cite{etheridge/freeman/penington/straulino:2017}. In that work, 
it was shown that, in the case of genic selection (in which exactly
two offspring are produced in the dual at a selective event),
if $\beta=1/2$, $u_n \equiv u$ and 
$\v{s}_n = \log(n)/n$, there is an equilibrium of branching and 
coalescence in the dual, so that in the limit as $n\to\infty$ we 
see a branching Brownian motion (with an `effective' branching rate). 
In our notation, this condition on $\v{s}_n$ 
corresponds to $\v{s}_nn^{2\beta}/(u_n\log n)$ being ${\mathcal O}(1)$.
In the case $u_n\equiv u$, if a pair of lineages is covered 
by a reproduction event, then
there is a probability of order one that they will coalesce. As a result, 
even for large $n$, they coalesce having been hit by a finite number of 
events. In our setting, if one lineage is affected by an event, the 
probability that the second lineage is also affected is $u_n$ and, as 
a consequence, a pair of lineages must come together of the order of 
$1/u_n$ times before they coalesce. The factor $\log n$ corresponds
to the number of times that two lineages will come close enough together
to be covered by the same event, before they escape to a separation of order
one; once they do that, we can expect to wait a long time before they next 
come close together.

\begin{proof}[Sketch of proof of Proposition \ref{Lemma:boundN}.]

Most of the work was done in~\cite{etheridge/freeman/penington/straulino:2017},
and indeed we don't require the precise estimates obtained there.
First suppose that two lineages, $\xi^{n,1}, \xi^{n,2}$, are created in a 
selective event occurring at time $0$. 
%a selective event. Without loss of generality, the event happens at time $0$. 
Let $\eta^n = \xi^{n,1} - \xi^{n,2}$. 
The key step is to show that with high
probability %$1-{\mathcal O}(1/(u_n\log n)$ 
the two lineages will 
coalesce before time $1/(\log n)^c$ where $c$ can be chosen to be at
least $3$. 
(We shall use a union bound on the complementary event 
to estimate the probability that all three 
lineages created in a selection event coalesce on this timescale.)

Much of the work arises from the fact that $\xi^{n,1}$ and $\xi^{n,2}$ only
evolve independently when their separation is more than $2{\mathcal R}_n$.

Following~\cite{etheridge/freeman/penington/straulino:2017}, we consider
three possible scenarios:
\begin{enumerate}
\item 
The separation $\|\eta^n\|$ exceeds $(\log n)^{-c}$ at some time before
$(\log n)^{-c}$. This occurs  
with probability $\mathcal{O}\left(\frac{1}{u_n\log n}\right)$,
and we say that $\eta^n$ diverges.
\item 
The quantity
$\|\eta^n\|$ does not exceed $(\log n)^{-c}$, and neither do the two
lineages coalesce before time $(\log n)^{-c}$. 
This happens with probability 
$\mathcal{O}\left(\frac{1}{(\log n)^{c-3/2}u_n}\right)$,
and we say that $\eta^n$ overshoots.
\item 
The lineages 
coalesce within time $(\log n)^{-c}$.
%This happens with probability $1-o\left(\frac{1}{\log(n)u_n}\right)$ 
\end{enumerate}

The idea, which can be traced to Lemma~4.2 of~\cite{etheridge/veber:2012}, is
to characterise the behaviour of $\eta^n$ in terms of `inner' and `outer'
excursions, defined through sequences of stopping times.
Set $\tau_0^{\mathrm{out}} = 0$ and define
\begin{align}
\tau_i^{\mathrm{in}} = \inf\{ s > \tau_i^{\mathrm{out}}: |\eta_s^n| \geq 5 \mathcal{R}_n \}, \label{def:inex} \\
\tau_{i+1}^{\mathrm{out}} = \inf\{ s > \tau_i^{\mathrm{in}}: |\eta_s^n| \leq 4 \mathcal{R}_n \}. \label{def:outex}
\end{align}
The interval $[\tau_i^{\mathrm{out}},\tau_i^{\mathrm{in}})$ (and the path of $\eta^n$ during it)
is the $i$th inner excursion, and similarly $[\tau_{i-1}^{\mathrm{in}}, 
\tau_i^{\mathrm{out}})$ 
(and the corresponding path) is the $i$th outer excursion. 
%The first are called $i$-th inner excursions and the second one the $i$-th outer excursions.

We use $\PP_{[r_1,r_2]}$ to denote results that hold 
for any $|\eta_0^n| \in [r_1,r_2]$ and we set
\[ L_n = (\log n)^{-c}, \]
for a fixed $c \geq 3$. 
We then define
\begin{align*}
\tau^{\mathrm{coal}} = \inf \{ s > 0 : |\eta_s^n| = 0 \},\\
\tau^{\mathrm{div}} = \inf\{s > 0: |\eta_s^n| \geq L_n \},\\
\tau^{\mathrm{type}} = \tau^{\mathrm{coal}} \wedge \tau^{\mathrm{div}}
 \wedge L_n.
\end{align*}
We say $\eta^n$ coalesces if $\tau^{type} = \tau^{coal}$, diverges 
if $\tau^{type} = \tau^{div}$, and overshoots otherwise. 
%Noting that this says that, with high probability, all the particles created in selective events coalesce. To obtain this we will rest on some results and the structure seen in \cite{etheridge/freeman/penington/straulino:2017}. In fact, by following the same proof of Lemma 4.4 therein we get the fact that:
%\[ \PP_x[ \eta^n \text{ overshoots}] = \mathcal{O}\left(\frac{1}{\log(n)^{c-3/2}} \right). \]
%The question will be wherever $\eta^n$ achieves a distance of $L_n$ in time $1/\log(n)^{-c}$. 
We are also going to need the stopping times:
\begin{align}
\tau_r = \inf\{ s > 0: |\eta_s^n| \leq r \}, \label{def:lowvale}\\
\tau^r = \inf\{s > 0: |\eta_s^n| \geq r \}. \label{def:highvale}
\end{align}
Lemma~4.7 in \cite{etheridge/freeman/penington/straulino:2017} yields
%and its proofs follows unchanged:
%\begin{lemma}
that, as $n \rightarrow \infty$,
\begin{equation} \label{lemmanicdan}
\PP_{[5 \mathcal{R}_n, 7 \mathcal{R}_n]}[ \tau^{L_n} < \tau_{4 \mathcal{R}_n}] 
= \mathcal{O}\left(\frac{1}{\log n} \right).
\end{equation}
%\end{lemma}
The proof is unchanged in our setting; during outer
excursions, $\eta^n$ behaves like the difference between two independent
jump processes and it takes $\mathcal{O}(\log n)$ such excursions to
achieve separation $L_n$.

In order to control $\PP[\eta^n \text{ diverges}]$, we control the number
of inner excursions before we see a coalescence. 
We first note there exists $M >0$ such that
\begin{equation} \label{ineq}
 \PP_{[0,5 \mathcal{R}_n]}[\tau_0 \leq \tau^{5 \mathcal{R}_n}] \geq M u_n \end{equation}
Indeed, from a separation of $5{\mathcal R}_n$, there is a strictly
positive probability that within $3$ jumps $\|\eta^n\|\leq {\mathcal R}_n$ and
then a strictly positive probability that the next event will cover both 
lineages, in which they coalesce with probability $u_n$.
%first, consider any $3$ jumps of the process. Then, worst-case scenario, the process has to take two jumps to be at distance $[0,2 \mathcal{R}_n]$ and then one have to jump while the other is also affected for them to coalesce by the event, which has probability $u_n$. Analysing more jumps of the process is done by the Markov property, and so have a smaller probability than the case just presented. \\
This allows us to bound above
the number of inner excursions before coalescence by
a geometric random variable with success probability $M u_n$, and hence
the probability of divergence has order at most $1/(u_n\log n)$. Again the 
details follow exactly as in~\cite{etheridge/freeman/penington/straulino:2017}.

The proof that the probability that $\eta^n$ overshoots is 
${\mathcal O}\left(\frac{1}{(\log n)^{c-3/2}}\right)$ is 
a minor modification of that
of~\cite{etheridge/freeman/penington/straulino:2017}, Lemma~4.4. The
idea is that if $\|\eta^n_t\|=x$, with $x\in (0, 5{\mathcal R}_n)$, 
then with strictly positive probability $\theta$, $\eta^n$ will either 
coalesce or exit $B(0, 5{\mathcal R}_n)$ within three jumps.
The number of jumps that $\eta^n$ makes before either exiting 
$B(0, 5{\mathcal R}_n)$ or coalescence is therefore
stochastically bounded by 
three times 
a geometric random variable with success probability $\theta$.
Since $\eta^n$ jumps at least as fast as $\xi^{n,1}$, which
jumps at rate $m_n$ given by~(\ref{R^d_rateancestry}),
this allows us to estimate 
$$\IP_{(0, 5{\mathcal R}_n)}\left[\tau_0\wedge \tau^{5{\mathcal R}_n}>n^{-\beta}\right]
={\mathcal O}\left(n^{-\beta/3}+(1-\theta)^{n^{2\beta/3}}\right),$$
where the first term is an estimate of the probability that the 
sum of $n^{2\beta/3}$ independent exponential random variables with 
parameter $m_n$ exceeds $n^{-\beta}$
and the second is the probability that a Geometric random variable
with parameter $\theta$ exceeds $n^{2\beta/3}$.

Now let $i^*$ denote the index of the excursion during which 
$\tau^{\mathrm{type}}$ occurs.
From the argument that we used to control the probability of 
divergence, we can bound the number of inner excursions
before coalescence above by a geometric
random variable with success probability $M u_n$. 
Let $n$ be large enough that
$$\frac{(\log n)^{1/2}}{u_n}\left(n^{-\beta}+(\log n)^{-c-2}\right)
\leq (\log n)^{-c}.$$
Note that this is possible since $u_n\log n \to \infty$ as $n\to \infty$, by
assumption.
If $\eta^n$ overshoots, and $i^*<(\log n)^{1/2}/u_n$, then at
least one inner excursion must have lasted at least $n^{-\beta}$ or
at least one outer excursion must have lasted at least $(\log n)^{-c-1}$.
Again using that when $\|\eta^n\|>2{\mathcal R}_n$, $\eta^n$ behaves as 
the difference of two independent walks 
(and Skorohod embedding), Lemma~4.7
of~\cite{etheridge/freeman/penington/straulino:2017}
shows that 
$$\IP_{[5{\mathcal R}_n, 7{\mathcal R}_n]}\left[
\tau^{L_n}\wedge \tau_{4{\mathcal R}_n} >(\log n)^{-c-2}\right]
=\mathcal{O}\left(\frac{1}{(\log n)^{c}}\right).$$
Combining the above,
\begin{multline*}
\IP\left[\eta^n\mbox{ overshoots}\right]
\leq \frac{(\log n)^{1/2}}{u_n}\Bigg\{
\IP_{(0, 5{\mathcal R}_n)}\left[\tau_0\wedge \tau^{5{\mathcal R}_n}>n^{-\beta}\right]
\\+
\IP_{[5{\mathcal R}_n, 7{\mathcal R}_n]}\left[
\tau^{L_n}\wedge \tau_{4{\mathcal R}_n} >(\log n)^{-c-2}\right]\Bigg\}
+\IP\left[i^*>\frac{(\log n)^{1/2}}{u_n}\right]\\
\leq 
\frac{(\log n)^{1/2}}{u_n}\left(n^{-\beta}+(\log n)^{-c}\right)
+(1-M u_n)^{(\log n)^{1/2}/u_n}
=\mathcal{O}\Big(\frac{1}{(\log n)^{c-3/2}u_n}\Big).
\end{multline*}

To conclude the proof of Proposition~\ref{Lemma:boundN}, let us write
$\xi$ for a root to leaf ray connecting $\xi_0=x$ to $\xi^n_1(t)$ (the first
individual in ${\cal P}^n(t)$).
Let $S_m$ denote the time of the $m$-th selective event to affect the lineage,
creating offspring 
$\xi^{n,1},\xi^{n,2},\xi^{n,3}$. Then, by the Markov property,
\begin{align*}
\PP&\left[\text{all individuals created at time } S_m \text{ have not coalesced by time }
S_m+\frac{1}{(\log n)^c}\right]\\ & %\leq \bigcup_{i,j \in \{1,2,3\}, i \not = j} \PP[(\xi^{n,i}-\xi^{n,j}) \text{ does not coalesce }] \\ & =
\leq 3\PP[\eta^n=(\xi^{n,1}-\xi^{n,2}) \text{ diverges or overshoots}]\\ 
& = \mathcal{O}\left(\frac{1}{u_n\log n} \right) + 
\mathcal{O}\left(\frac{1}{u_n(\log n)^{c-3/2}}\right)
=\mathcal{O}\left(\frac{1}{u_n\log n}\right).
\end{align*}
Since selective events affect $\xi$ according to a 
Poisson process of rate proportional to 
$\v{s}_n u_n n/\widehat{u}_n=\v{s}_n n^{2 \beta}$,
\begin{align*}
\PP[N_t^n\not = 1] &\leq \PP\Big[\text{a particle created in } 
[0,t-(\log n)^{-c}] \text{ did not coalesce}\Big] \\
&+ \PP\Big[\text{a selective event occurred in } [t-(\log n)^{-c},t]\Big] \\ 
%&= \v{s}_n n^{2\beta}\left( \mathcal{O}\left( \frac{1}{\log(n) u_n} \right) 
%+ \mathcal{O}\left(\frac{1}{\log(n)^{c-3/2}} \right) \right) 
%+ \mathcal{O}(\v{s}_n n^{2\beta} \log(n)^{-c}) \\ 
&= \mathcal{O}\left(\frac{\v{s}_n n^{2 \beta}t}{u_n\log n} \right) 
+ \mathcal{O}\left(\frac{\v{s}_n n^{2\beta}}{(\log n)^{-c}}\right)
\leq \mathcal{O}\left(\frac{\v{s}_n n^{2\beta}(t\vee 1)}{u_n \log n}\right), 
\end{align*}
where our conditions on $\v{s}_n$ guarantee that all the terms on 
the right hand side tend to $0$ as $n\to\infty$, and the proof is complete.
\end{proof}

\begin{appendix}

\section{Proof of Lemma~\ref{lemma:biasextg}}
\label{Proof of bias amplification Lemma}

For completeness, in the section we reproduce the proof of
Lemma~\ref{lemma:biasextg} from~\cite{gooding:2018}.

We begin with the first statement. We shall carry out two phases of iteration of $g$. First, we will show that it takes $\bigO{|\log \epsilon|}$ iterations to obtain 
\begin{equation}\label{eq:g_iter_1}
g^{(n)}\big(\tfrac{1+\gamma_\epsilon}{2}+\epsilon\big)\geq \tfrac{1}{2}
+\sqrt{\tfrac{1+\gamma_\epsilon^2}{8}}.
\end{equation}
Note that \(\tfrac{1+\gamma}{2}<\tfrac{1}{2}+\sqrt{\tfrac{1+\gamma^2}{8}}<1\) 
for \(\gamma\in(0,1)\). After establishing this, we note 
that $\bigO{k |\log \epsilon|}$ iterations are required to obtain 
\begin{equation}\label{eq:g_iter_2}
g^{(n)}\left(\tfrac{1}{2}+\sqrt{\tfrac{1+\gamma_\epsilon^2}{8}}\right)
\geq 1-\epsilon^k, 
\end{equation}
and then, since $g$ is monotone, combining the two phases 
will complete the proof of the first statement.

For the first phase, note that if 
$\delta\in(0,\sqrt{(1+\gamma_\epsilon^2)/8}-\gamma_\epsilon/2)$ 
then a straightforward calculation 
gives 
\[\begin{split}
g(\tfrac{1+\gamma_\epsilon}{2}+\delta)=&\tfrac{1+\gamma_\epsilon}{2}+
\tfrac{\delta}{2(1+\gamma_\epsilon)}(3+2\gamma_\epsilon-(\gamma_\epsilon+2\delta)^2)
\\\geq&\tfrac{1+\gamma_\epsilon}{2}+\tfrac{5-\gamma_\epsilon}{4}\delta.
\end{split}\]
Thus if 
$g^{(n)}(\frac{1+\gamma_\epsilon}{2}+\epsilon)-\frac{1}{2}
<\sqrt{\tfrac{1+\gamma_\epsilon^2}{8}}$, setting
$\delta= g^{(n)}(\frac{1+\gamma_\epsilon}{2}+\epsilon)-
\frac{1+\gamma_\epsilon}{2}$,
we have
\begin{align*}
g^{(n+1)}(\tfrac{1+\gamma_\epsilon}{2}+\epsilon)-\tfrac{1+\gamma_\epsilon}{2}
%&=\tfrac{3}{2}f^{(n)}(\epsilon)-2f^{(n)}(\epsilon)^3
\geq \tfrac{5-\gamma_\epsilon}{4}\left(g^{(n)}(\tfrac{1+\gamma_\epsilon}{2}+\epsilon)-\tfrac{1+\gamma_\epsilon}{2}\right)
\geq (\tfrac{5-\gamma_\epsilon}{4})^{n}\epsilon.
\end{align*} 
It follows immediately that $\bigO{|\log \epsilon|}$ iterations are required to achieve~\eqref{eq:g_iter_1}.

For the second phase, as \(g\) is monotone increasing on \([0,1]\), it is easy to see that
\begin{align*}
1-g(1-\delta)=&\tfrac{\delta}{1+\gamma_\epsilon}\left((3-\gamma_\epsilon)\delta+2\gamma_\epsilon-2\delta^2\right) \notag
\\\leq& \tfrac{\delta}{1+\gamma_\epsilon}\left((3-\gamma_\epsilon)\left(\tfrac{1}{2}-\sqrt{\tfrac{1+\gamma_\epsilon^2}{8}}\right)+2\gamma_\epsilon\right) \notag
\\:=&a_{\gamma_\epsilon+}\delta,
\label{eqn:high_g_bound}
\end{align*} 
where the inequality holds for \(0\leq\delta\leq\tfrac{1}{2}-\sqrt{\tfrac{1+\gamma_\epsilon^2}{8}}\).
Another simple calculation shows that \(0<a_{\gamma_\epsilon+}<1\), which means that the sequence created by iterating \(1-g(1-\delta)\) starting with \(0\leq\delta\leq\tfrac{1}{2}-\sqrt{\tfrac{1+\gamma_\epsilon^2}{8}}\) remains in the interval  \(\left[0,\tfrac{1}{2}-\sqrt{\tfrac{1+\gamma_\epsilon^2}{8}}\right]\).
Thus,
\[1-g^{(n+1)}\left(\tfrac{1}{2}+\sqrt{\tfrac{1+\gamma_\epsilon^2}{8}}\right)\leq a_{\gamma_\epsilon+} 
\left(1-g^{(n)}\left(\tfrac{1}{2}+\sqrt{\tfrac{1+\gamma_\epsilon^2}{8}}\right)\right)\leq a_{\gamma_\epsilon+}^{n+1}\left(\tfrac{1}{2}-\sqrt{\tfrac{1+\gamma_\epsilon^2}{8}}\right).\]
Noting again that \(0<\tfrac{1}{2}-\sqrt{\tfrac{1+\gamma_\epsilon^2}{8}}<\tfrac{1}{2}\), it follows easily that the number of iterations required to obtain~\eqref{eq:g_iter_2} is $\bigO{k|\log\epsilon|}$.

We now turn to the second statement. The proof of this 
statement is very similar to the first phase, but it does not simply follow by symmetry, so we include it here. Again, we split the proof into two phases. For the first phase, if $\delta\in(0,\sqrt{(1+\gamma_\epsilon^2)/8}+\gamma_\epsilon/2)$ then almost the same calculation as above shows that 
\[\begin{split}
g(\tfrac{1+\gamma_\epsilon}{2}-\delta)=&\tfrac{1+\gamma_\epsilon}{2}-\tfrac{\delta}{2(1+\gamma_\epsilon)}(3+2\gamma_\epsilon-(\gamma_\epsilon-2\delta)^2)
\\\leq&\tfrac{1+\gamma_\epsilon}{2}-\tfrac{5-\gamma_\epsilon}{4}\delta.
\end{split}\]
Thus if $\frac{1}{2}-g^{(n)}(\frac{1+\gamma_\epsilon}{2}-\epsilon)<\sqrt{\tfrac{1+\gamma_\epsilon^2}{8}}$, we have
\begin{align*}
\tfrac{1+\gamma_\epsilon}{2}-g^{(n+1)}(\tfrac{1+\gamma_\epsilon}{2}-\epsilon)
\geq \tfrac{5-\gamma_\epsilon}{4}\left(\tfrac{1+\gamma_\epsilon}{2}-g^{(n)}(\tfrac{1+\gamma_\epsilon}{2}-\epsilon)\right)
\geq (\tfrac{5-\gamma_\epsilon}{4})^{n}\epsilon.
\end{align*} 
It follows immediately that $\bigO{|\log \epsilon|}$ iterations are required to achieve
\begin{equation}\label{eqn:g_iter_3}
g^{(n)}(\tfrac{1+\gamma_\epsilon}{2}-\epsilon)\leq \tfrac{1}{2}-\sqrt{\tfrac{1+\gamma_\epsilon^2}{8}}.
\end{equation}

This time for the second phase observe that for 
\(0\leq\delta\leq\tfrac{1}{2}-\sqrt{\tfrac{1+\gamma_\epsilon^2}{8}}\), we have
\begin{align*}
g(\delta)\leq \frac{3+\gamma_\epsilon}{1+\gamma_\epsilon}\delta^2 \leq \frac{3+\gamma_\epsilon}{1+\gamma_\epsilon}\left(\tfrac{1}{2}-\sqrt{\tfrac{1+\gamma_\epsilon^2}{8}}\right)\delta:=a_{\gamma_\epsilon-}\delta.
\end{align*} 
Another simple calculation shows that \(0<a_{\gamma-}<1\), which means that \(g^{(n)}(\delta)\in[0,\tfrac{1}{2}-\sqrt{\tfrac{1+\gamma_\epsilon^2}{8}}]\) if \(0\leq\delta\leq\tfrac{1}{2}-\sqrt{\tfrac{1+\gamma_\epsilon^2}{8}}\).
Thus,
$$g^{(n+1)}\left(\tfrac{1}{2}-\sqrt{\tfrac{1+\gamma_\epsilon^2}{8}}\right)\leq a_{\gamma_\epsilon-}
g^{(n)}\left(\tfrac{1}{2}-\sqrt{\tfrac{1+\gamma_\epsilon^2}{8}}\right)\leq a_{\gamma_\epsilon-}^{n+1}\left(\tfrac{1}{2}-\sqrt{\tfrac{1+\gamma_\epsilon^2}{8}}\right).$$
Noting again that \(0<\tfrac{1}{2}-\sqrt{\tfrac{1+\gamma_\epsilon^2}{8}}<\tfrac{1}{2}\), it follows easily that it takes $\bigO{k|\log\epsilon|}$ iterations to achieve
\begin{equation}\label{eqn:g_iter_4}
g^{(n)}\left(\tfrac{1}{2}-\sqrt{\tfrac{1+\gamma_\epsilon^2}{8}}\right)\leq \epsilon^k. 
\end{equation}
Combining~\eqref{eqn:g_iter_3} and~\eqref{eqn:g_iter_4} gives the second statement, which concludes the proof.

\section{A geometric computation} \label{cumbersome geometric computation}
\begin{proposition}
Let $\Omega = \{(x_1,x'), x_1 \in \RR, x' \in \phi(x_1) \subseteq \RR^{\dd-1} $ with $\phi(x_1) = H-h(x_1)$. Let $z>0$, $a = -z+\sqrt{\mathbbm{r}^2-(H+h(z))^2}$. Let $x \in \partial \Omega$, $v_1(x)$ be the vector pointing from $x$ to $(a,0,..,0)$ and $v_2(x)$ the (inward) normal vector of $\partial \Omega$ at $x$. Then,
\begin{equation}
\langle v_1(x), v_2(x) \rangle =  H + h(-x_1)+h'(-x_1)(-z-x_1-\sqrt{\mathbbm{r}^2-(H+h(z))^2}). \label{point product in boundary comp}
\end{equation}
\end{proposition}
\begin{proof}
This is mostly a matter of writing the vectors and then the dot product explicitly. Without loss of generality, by rotational symmetry, we can assume $x=(x_1,x_2,0, \ldots, 0)$. In fact, since $x$ is in the boundary of $\partial \Omega$, we can assume that $x = (x_1, H+h(-x_1),0, \ldots, 0)$. Therefore we can write $v_1(x)$ as,
\[ v_1(x) =  (a-x_1,-(H+h(-x_1)), 0, \ldots, 0). \] 
On the other hand, since the boundary of the domain has a contour given by $h(-x_1)$, one can see $v_2(x)$ is given by,
\[ v_2(x) = (h'(-x_1),-1,0,\ldots,0). \]
Hence we have that,
\begin{align*}
\langle v_1(x), v_2(x) \rangle &= h'(-x_1)(a-x_1)+(H-h(-x_1))  \\ &= H + h(-x_1)+h'(-x_1)(-z-x_1-\sqrt{\mathbbm{r}^2-(H+h(z))^2}).
\end{align*}
\end{proof}

\section{A SLFVS with reflecting boundary condition}
\label{SLFV with reflecting bdry}

Since our main focus in the deterministic setting was on the 
interaction between selection and the shape of the domain, for 
completeness, we should like analogous results in the stochastic
setting. As is evident from the results in the deterministic setting, 
such results should depend on local effects around a spherical shell, 
and therefore follow from our work in
Section~\ref{sec:stochastic}. 
However, the SLFVS has only previously been studied on 
the whole of Euclidean space, or on a torus.
In this section we therefore suggest a way in which the SLFVS can be 
extended to a process with reflecting boundary conditions, at least
for a class of domains that includes $\Omega$ of Figure~\ref{fig:omega}.
The idea is simple: when an event overlaps the boundary of the domain, 
we sample parents in such a way that the transition densities of 
the motion of ancestral 
lineages will be what we obtained by \textit{reflecting} the motion each 
time it intersects the domain. To be more precise, to obtain a reflected 
motion we will use the method of images, a classical way of obtaining 
solutions of equations with reflecting boundary conditions, see
e.g.~\cite{smythe:1988}. 
In the interests of simplifying the notation, we present the details 
in $\mathbbm{d}=2$, but the ideas are easily adapted to $\mathbbm{d} \geq 3$. 

We reserve the symbol $\Omega$ for the domain of Figure~\ref{fig:omega}, 
%$(\mathbb{R}^- \times B_{\mathbbm{d}-1}(0,R_0)) 
%\cup  ((\mathbb{R}^+ \cup \{0\})\times B_{\mathbbm{d}-1}(0,r_0))$, 
and use $\D$ to refer to the class of domains for which we now define
the SLFVS with reflecting boundary conditions.

\subsubsection*{Geometric definitions.} 
%\label{sec:geodef}

We will need two notions: {\em reflected points} at the boundary, and 
a classification of balls depending on the way in which they intersect
the boundary of $\D$. Both will rest on a decomposition of the 
boundary of $\D$ into lines, and we only consider domains for which
this is possible (which of course includes $\Omega$).
In higher dimensions we would replace `line' by `plane' and below we shall use
the term `affine hyperplane' for a doubly infinite straight line and `line'
to mean a segment of such a line. 

We define a line (in dimension $2)$ as a closed convex subset of an affine hyperplane. 
We suppose that $\partial \D$ can be decomposed as a finite union of lines. 
That is, there exist lines $(L_i)_{i=1}^N$ such that
\begin{equation} \label{p1} \partial \D = \bigcup_{i=1}^N L_i. \end{equation}
We suppose that the lines are maximal, in the sense that 
\begin{equation} \label{p2} |L_j \cap L_i| \leq 1, \text{ and if } L_j \cap L_i \not = \emptyset \text{ then } H_{L_i} \not = H_{L_j}\: \: \: \forall j \not = i, \end{equation} 
where $H_{L_i}$ is the affine hyperplane that contains $L_i$. 
Note that $\partial \Omega$ admits a unique decomposition into lines.
Indeed, in that case, it is elementary to write down explicit expressions;
for example, one line and its corresponding hyperplane are given by
\[ L = \{(x_1,R_0) \in \Omega | x_1 \leq 0\} \subseteq H_L = \{(x_1,R_0)| x_1 \in \mathbb{R} \} . \]
From now on we write $\mathcal{L}_\D$ for the set of lines such 
that~(\ref{p1}) and~(\ref{p2}) hold for a domain $\mathcal{D}$.

Let $L \in \mathcal{L}_\D$, then let $H_L$ be the affine hyperplane 
such that $L \subseteq H_L$, and write $\widehat{n}_{H_L}$ for the 
corresponding normal vector. The orthogonal projection of $z$ into $H_L$ 
is given by $z-\langle z,\widehat{n}_{H_L}\rangle \widehat{n}_{H_L}$. 
If this element lies in $L$, then we define the reflection of $z$ with 
respect to $L$ as its reflection
with respect to $H_L$; that is the unique (other) point with the same 
orthogonal projection and distance to $H_L$ as $z$. 
%We can then define the point of $z$ reflected by $P$ by using the last observation. 
We also need to define reflection with respect to a corner of our domain.
A formal definition is below, and we provide an illustration in
Figure~\ref{figure:exampleRPRC}.
%%%%%%%%%%%%%%%%%%%%%%%%%%%%%%%%%%%%%%%%%%%%%%%%%%%%%%%%%%%%%%%%%%%%
\begin{figure}
\begin{center}
(i)\includegraphics[height=1.2in]{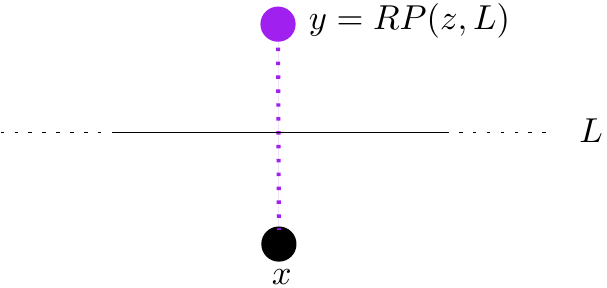}
(ii) \includegraphics[height=1.5in]{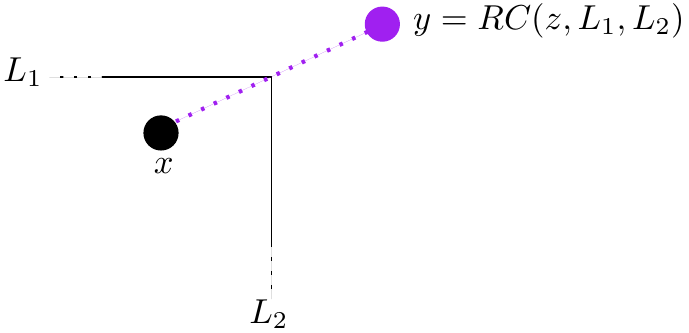}
\end{center}
\caption{(i) Example of how we reflect a point with respect to a line.
(ii) Example of how we reflect a point with respect to a corner. }
\label{figure:exampleRPRC}
\end{figure} 
\begin{definition}[Reflected points]
For $L \in \mathcal{L}_\D$ and $z \in \D$ we define 
$RP(z,L)$, the reflected point of $z$ with respect to $L$ as:
\[ RP(z,L) := \{w \in \D^c | z-\langle z,\widehat{n}_{H_L}\rangle \widehat{n}_{H_L} = w-\langle w,\widehat{n}_{H_L}\rangle \widehat{n}_{H_L} \in L \text{ and } d(z,H_L) = d(w,H_L) \}. \]
For $L_1,L_2 \in \mathcal{P}_\Omega$ with $L_1 \cap L_2 \not = \emptyset$ we define the set of reflected points with respect to the corner, $RC(z,L_1,L_2)$, as:
\begin{align*}
RC(z,L_1,L_2) = \{ & w \in \D^c | (w \not \in RP(z,L_1) \cup RP(z,L_2)) \wedge (\exists \widehat{z} \in RP(z,L_2);  \\ & \widehat{z}-\langle \widehat{z},\widehat{n}_{H_{L_1}}\rangle \widehat{n}_{H_{L_1}} = w-\langle w,\widehat{n}_{H_{L_1}}\rangle \in \Omega^c) \wedge (d(\widehat{z}, H_{L_1}) = d(w,H_{L_1})) \}.  
\end{align*}
We also define the `completed' versions:
\[ \overline{RP(z,L)} = RP(z,L) \cup \{ z\}, \]
\[ \overline{RC(z,L_1,L_2)} = RC(z,L_1,L_2) \cup RC(z,L_2,L_1) \cup RP(z,L_1) \cup RP(z,L_2) \cup \{ z \}. \] 
\end{definition}
We note that $|RP(z,L)| \leq 1$, as for any hyperplane there 
are two points with the same orthogonal projection 
and distance to it. 
Also we remark that for $L_1,L_2 \in \mathcal{L}_\D$ 
with $L_1 \cap L_2 \not = \emptyset$, and $w \in \mathbb{R}$, setting
\[ \overline{RP}^{-1}(w,L_1) = \{z \in \Omega| z \in \overline{RP}(w,L_1) \}, \]
\[ \overline{RC}^{-1}(w,L_1,L_2) = \{z \in \Omega| z \in \overline{RC}(w,L_1,L_2) \}, \]
for any $z \in \Omega$ we have that
\begin{equation} \label{eq:invomega}
\overline{RP}^{-1}(z,L_1) = \overline{RC}^{-1}(z,L_1,L_2) = \{ z \}.
\end{equation}
We now define a classification of balls depending on their 
intersection with $\mathcal{D}$. 
For any $A \subset \mathbb{R}^\mathbbm{d}$ we set
\[ A\cap \mathcal{L}_\D := \{L \in \mathcal{L}_\D |A \cap L \not = \emptyset \text{ or } A \cap H_L \cap \D^c \not = \emptyset \}.\]
\begin{definition} \label{def:cballs}[Classification of balls in $\mathbb{R}^2$]
Let $x \in \mathbb{R}^2$ and $r >0$. 
\begin{enumerate}
    \item We say $B(x,r)$ is inside $\D$ if $B(x,r) \subset \D$.
    \item We say it is outside $\D$ if $B(x,r) \subset \D^c$.
    \item We say $B(x,r)$ intersects a side of $\D$ if $|B(x,r) \cap \mathcal{L}_\D|=1$ and $B(x,r)$ is not outside of $\D$.
    \item We say $B(x,r)$ intersects a corner if $|B(x,r) \cap \mathcal{L}_\D| = |\{ L_1, L_2 \}|= 2$ with $L_1 \cap L_2 \not = \emptyset$ and $B(x,r)$ is not outside of $\D$.
    \item We say $B(x,r)$ intersects a hallway if $|B(x,r) \cap \mathcal{L}_\D| = |\{ L_1, L_2 \}|= 2$ with $L_1 \cap L_2 = \emptyset$ and $B(x,r)$ is not outside of $\D$.
\end{enumerate}
\end{definition}
Note that Definition \ref{def:cballs} provides a complete classifications of balls in $\mathbb{R}^2$, of radius at most $r$, given that $|B(x,r) \cap \mathcal{L}_\D| \leq 2$ for all $x \in \mathbb{R}^2$. With this we can define the concept of reflected sampling in $B(x,r)$.
\begin{definition}[Reflected sampling] \label{def: refsam}
Let $x \in \mathbb{R}^2$ and $r>0$ such that $B(x,r)$ is not outside $\D$, does not intersect a hallway and $|B(x,r) \cap \mathcal{L}_\mathcal{D}| \leq 2$. We define the reflected sampling in $B(x,r)$ as follows:
\begin{enumerate}
    \item If $B(x,r)$ is inside $\D$ then it is just the uniform sampling in $B(x,r)$.
    \item If $B(x,r)$ intersects a side of $\D$, that is $B(x,r) \cap \mathcal{L}_\D = \{ P \}$, then we pick a point $z$ uniformly at random in $B(x,r)$, take the only point in $\overline{RP}^{-1}(z,P)$.
    \item If $B(x,r)$ intersects a corner of $\D$, that is $B(x,r) \cap \mathcal{L}_\D = \{ L_1,L_2 \}$, then we pick a point $z$ uniformly at random in $B(x,r)$, then pick a point uniformly at random in $\overline{RC}^{-1}(z,L_1,L_2)$.
\end{enumerate}
\end{definition}
Reflected sampling can be thought of as sampling from the complete ball and 
then, if the ball intersects the boundary, {\em folding it along
the boundary} in such a way that the content is all inside $\D$.
This is illustrated for the domain $\Omega$ in Figure~\ref{fig:refsam}.
\begin{figure*}[t!] 
    \centering
	\includegraphics[]{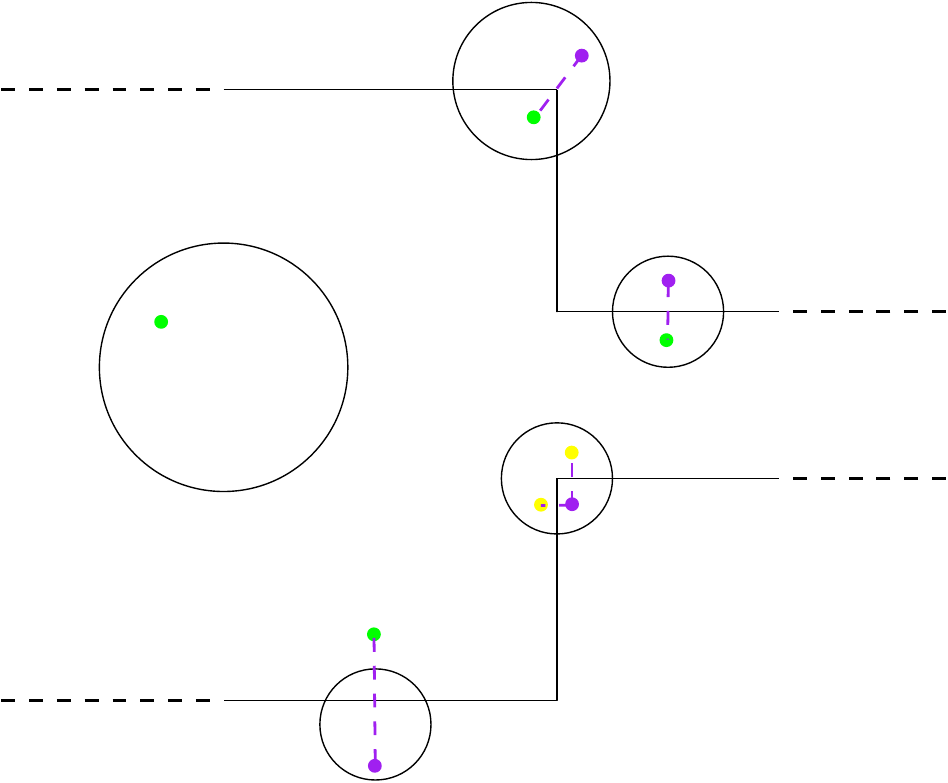}%{CFBlocking/images/omegaexref.pdf}
    \caption{{\em Reflected sampling in $\Omega$.} 
Each of the balls intersects $\Omega$. For the ball contained 
completely within $\Omega$, the green point is uniformly sampled in the usual
way. For balls intersecting the boundary, points in purple are sampled 
uniformly from the intersection of the ball with $\Omega^c$,
corresponding to a uniformly sampled point falling outside the domain.
The points in green correspond to the cases in 
which reflected sampling leads to a unique point within the domain.
The points in yellow correspond to a case in which reflection gives 
two possible outcomes for the reflected sampling and we 
need to choose one of them uniformly at random. The purple lines join
sampled points to the corresponding possible outcomes of reflected sampling.}
    \label{fig:refsam}
\end{figure*} 

\subsubsection*{The SLFVS on $\Omega$.}
%\label{def:secdeffv}

The SLFVS on $\Omega$ can now be defined in the obvious
way. It is driven by the same Poisson Point Process $\Pi$ of events 
as in Definition~\ref{FVSdefn}, but
now for a point $(t,x,r)\in \Pi$, if $B(x,r)\cap\Omega=\emptyset$ then 
nothing happens, and when $B(x,r)$ intersects the boundary $\partial\Omega$,
parental locations
are chosen according to {\em reflected sampling}, and the allele frequencies
are, of course, only updated in $\Omega$.
%In this section we define the Spatial Lambda Fleming Viot with Selection in $\Omega$ (from now on \FV\, for short).
This construction will work for any domain $\D$ that fulfils the 
following conditions:
\begin{enumerate}
    \item There exists a set of lines $\mathcal{L}_\D$ such that $\D$ fulfils (\ref{p1}) and (\ref{p2}).
    \item There is $\mathcal{R}_\D>0$ such that, for all $x \in \mathbb{R}^2$, we have $|B(x,\mathcal{R}_\D) \cap \mathcal{L}_\D| \leq 2$ and $B(x,\mathcal{R}_\mathcal{D})$ does not intersect a hallway.
    \item For any $r < \mathcal{R}_\D$ and $x \in \mathbb{R}^2$, if $B(x,r)$ intersects a corner given by $L_1,L_2$ then $\overline{RC}^{-1}(w,L_1,L_2) \not = \emptyset$ for all $w \in B(x,r)$. If $B(x,r)$ intersects a side given by the line $L$, then $\overline{RP}^{-1}(w,P) \not = \emptyset$ for all $w \in B(x,r)$.
\end{enumerate}
We can then construct the SLFVS on $\D$, provided that events are assumed
to have radius bounded by $\mathcal{R}_{\D}$.
These conditions are sufficient to ensure that 
the reflected sampling that we use to choose 
a parental location will be well defined. 
The third condition prevents there being no pre-image through 
the reflection of the parental location; the second
removes the possibility of an event intersecting two walls, for which 
we have not defined the corresponding reflected points.
We could, of course, extend our definitions further, but we are primarily
interested in the domain $\Omega$, and 
one can check that $\Omega$ fulfils these conditions with
$\mathcal{R}_\Omega \leq r_0 \wedge R_0 - r_0$. 

\subsubsection*{Blocking in the stochastic setting}

Using the approach adopted in 
Theorem~\ref{noisy circles} we can prove a stochastic analogue
of Theorem~\ref{teo:simplifyversion}. While we omit the details, the key tool is a branching and 
coalescing dual for the SLFVS on $\Omega$. This mirrors the dual for the
process on the whole Euclidean space, introduced in 
Definition~\ref{SLFVS dual},
except that uniform sampling of offspring locations is replaced by 
reflected sampling. 

The key ideas in the proof of 
Theorem~\ref{noisy circles} 
can then be adapted in the obvious way to this setting.
Since we deliberately constructed the reflected sampling in such a way that
transition probabilities of the jump process followed by a lineage could be
obtained from those on Euclidean space via the method of images, it should be
no surprise that Lemma~\ref{lemma:coupRd}
can be translated to this setting on 
replacing Brownian motion by reflected Brownian motion. 
The technical details will appear in the doctoral thesis of the third
author.

\end{appendix}

\printbibliography

\end{document}